\documentclass[svgnames]{article}
\usepackage[margin=1.5in]{geometry}
\usepackage[utf8]{inputenc}
\usepackage[english]{babel}
\usepackage{hyperref}
\usepackage{csquotes}
\usepackage{amsmath,amssymb,amsthm}
\usepackage{mathtools}
\usepackage{enumerate}
\usepackage{tensor}
\usepackage{extarrows}
\usepackage{tikz-cd}
\usepackage[capitalize]{cleveref}
\usepackage{subcaption}
\usepackage{import}
\usepackage{todo}

\newtheorem{thm}{Theorem}
\numberwithin{thm}{section}
\newtheorem{prp}[thm]{Proposition}
\newtheorem{cor}[thm]{Corollary}
\newtheorem{lem}[thm]{Lemma}

\theoremstyle{definition}
\newtheorem{dfn}[thm]{Definition}

\theoremstyle{remark}
\newtheorem{remark}[thm]{Remark}
\newtheorem{exm}[thm]{Example}

\numberwithin{equation}{section}
\numberwithin{figure}{section}

\begin{document}

\newcommand{\C}{\mathbb{C}}
\newcommand{\R}{\mathbb{R}}
\newcommand{\Q}{\mathbb{Q}}
\newcommand{\Z}{\mathbb{Z}}
\newcommand{\N}{\mathbb{N}}
\newcommand{\M}{\mathbb{M}}
\newcommand{\D}{\mathbb{D}}
\newcommand{\fl}{\mathfrak{fl}}
\newcommand{\op}[1]{\operatorname{#1}}
\newcommand{\cat}[1]{\mathbf{#1}}
\newcommand{\VectK}{\mathrm{Vect}_K}
\newcommand{\vectK}{\mathrm{vect}_K}
\renewcommand{\href}[2]{#2}

\DeclarePairedDelimiterX\set[1]\lbrace\rbrace{\def\given{\;\delimsize\vert\;}#1}

\newcommand{\CSpine}{MediumSeaGreen}
\newcommand{\CHorn}{OrangeRed}
\newcommand{\CCyl}{RoyalBlue}

\title{Structure and Interleavings of Relative Interlevel Set Cohomology}

\author{Ulrich Bauer \and Magnus Bakke Botnan \and Benedikt Fluhr}

\maketitle

\begin{abstract}
  The \emph{relative interlevel set cohomology (RISC)}
  is an invariant of real-valued continuous functions
  closely related to the \emph{Mayer--Vietoris pyramid}
  introduced by Carlsson, de Silva, and Morozov.
  As such, the relative interlevel set cohomology
  is a parametrization
  of
  the cohomology vector spaces of all open interlevel sets
  relative complements of closed interlevel sets.
  We provide a structure theorem,
  which applies to the RISC
  of real-valued continuous functions
  whose open interlevel sets have finite-dimensional cohomology in each degree.
  Moreover, we show this tameness assumption is in some sense equivalent
  to \emph{$q$-tameness} as introduced by Chazal, de Silva, Glisse, and Oudot.
  Furthermore, we provide the notion of an \emph{interleaving}
  for RISC and we show that it is stable
  in the sense that any space with two functions
  that are $\delta$-close
  induces a $\delta$-interleaving of the corresponding
  relative interlevel set cohomologies.
  Finally, we provide an elementary form
  of \enquote{quantitative homotopy invariance}
  for RISC.
\end{abstract}

\section{Introduction}

In the present work we study an invariant
of real-valued continuous functions closely related to
and mostly inspired by the \emph{Mayer--Vietoris pyramid}
introduced by \cite{Carlsson:2009:ZPH:1542362.1542408}.
We name this invariant
\emph{relative interlevel set cohomology (RISC)}.
Roughly speaking, the Mayer--Vietoris pyramid is a graded
square shaped diagram and the RISC arises
from gluing consecutive layers of the Mayer--Vietoris pyramid
in a functorial way to form one large diagram.
This procedure has already been suggested by
\cite{Carlsson:2009:ZPH:1542362.1542408}
and moreover,
the results by \cite{MR3031814} underpin that such a construction
with all squares \emph{Mayer--Vietoris diamonds}
should be possible.
More specifically, \cite[Figures 4 and 6]{MR3031814}
show that the supports of the indecomposables of each layer align exactly.
This raises the question, whether a decomposition
given by their structure theorem \cite[Theorem 1]{MR3031814}
is compatible with the connecting homomorphisms
gluing consecutive layers as suggested by
\cite{Carlsson:2009:ZPH:1542362.1542408}.
As it turns out,
even though this cannot be said about an arbitrary choice of a decomposition
as in \cite[Theorem 1]{MR3031814},
we show that such a global decomposition indeed exists.
To this end, we provide a structure theorem
for the relative interlevel set cohomology itself,
which yields the same indecomposables as \cite[Theorem 1]{MR3031814}
on each of the corresponding layers of the pyramid.

The strong tameness assumptions in \cite[Theorem 1]{MR3031814}
were weakened by \cite[Theorem 10.1]{botnan2020local}
to all layers of the pyramid being pointwise finite-dimensional (pfd).
The assumption to our structure theorem \ref{thm:decomp}
is that all open interlevel sets
have finite-dimensional cohomology in each degree.
We call a continuous function \emph{$K$-tame}, if it satisfies this property
with respect to the field $K$.
This is very reassuring, as it shows there is no loss of generality
when passing from the individual layers of the Mayer--Vietoris pyramid
as an invariant
to the whole relative interlevel set cohomology.
Moreover, in \cref{cor:qtameKtame}
we show that our tameness assumption is in some sense equivalent
to another tameness assumption referred to as \emph{$q$-tameness}
by \cite[Section 1.1]{MR3524869}.
Furthermore, following \cite[Section 2.5]{MR3413628}
we provide a \emph{(super)linear family} on the indexing poset $\M$
and we obtain a stability theorem in \cref{sec:interleavings}.
The proof of our structure theorem, which is \cref{thm:decomp},
is inspired by \cite{MR3323327}.

The restriction of RISC to the subposet corresponding
to the south face of the pyramid yields essentially the same data
as Mayer--Vietoris systems introduced by \cite{2019arXiv190709759B}.
In this restricted setting the authors also provide a structure theorem
as well as a stability result.
While this restriction to the south face retains all information
on the level of objects,
we lose some information on the level of homomorphisms.
In particular, there are different RISC interleavings
restricting to identical interleavings of Mayer--Vietoris systems,
as shown in \cref{exm:hood}.

The tameness assumptions
from \cite{Carlsson:2009:ZPH:1542362.1542408,MR3031814}
were also weakened in \cite{MR3924175} by using measure theory.
Roughly speaking, the authors bypass the step
involving interleavings of generalized persistence modules \cite{MR3413628}
and map functions directly to measures,
which they compare
with the bottleneck distance of persistence diagrams \cite{MR2279866}.

We also note that our construction
of the relative interlevel set cohomology,
which applies to any cohomology theory,
has an analogous homological construction,
which is dual in the following sense.
For a homology theory valued in graded vector spaces
sending weak equivalences to isomorphisms,
we may consider the corresponding dual cohomology theory.
The resulting invariant is pfd iff
this is also the case for the corresponding invariant
defined in terms of homology.
Moreover, as the duality of vector spaces restricts to an equivalence
on finite-dimensional vector spaces,
any decomposition of this RISC yields a decomposition
of the corresponding \enquote{homological} invariant.
This way, homological decompositions can be obtained by duality
and hence are not treated explicitly in this paper.

In \cref{sec:risc} we introduce the relative interlevel set cohomology (RISC)
as an invariant of $K$-tame real-valued continuous functions.
Given a continuous function $f \colon X \to \R$,
the study of \emph{interlevel set persistent cohomology}
concerns the cohomology (with field coefficients)
of preimages $f^{-1}(I)$ of open intervals $I$.
This construction can be extended
to the relative cohomology of pairs $f^{-1}(I, C)$,
where $I\subseteq \mathbb R$ is an open interval
and $C \subseteq I$ is the complement of a closed interval.
Now taking the difference
\begin{equation*}
  (I, C) \mapsto I \setminus C
\end{equation*}
yields a bijection
between the set of all non-empty intervals in $\R$
and the set of all such pairs $(I, C)$ with $I \neq C$.
Moreover, for any pair of open subspaces $(U, V)$ of $\R$
with $U \setminus V = I \setminus C$
the cohomologies of $f^{-1}(U, V)$ and $f^{-1}(I, C)$
are naturally isomorphic by excision.
From our perspective the pair $(I, C)$ is a particularly convenient choice
to represent the interval $I \setminus C$,
see also \cref{prp:rho} below.
Furthermore,
given any pair of open subspaces $(U, V)$ of $\R$
such that any connected component of $U$ contains finitely many
connected components of $V$,
the cohomology of $f^{-1}(U, V)$
is naturally isomorphic to a product of cohomologies for pairs
$f^{-1}(I, C)$ as above.
More specifically,
for each such factor
the difference $I \setminus C$ is a connected components of $U \setminus V$.
We parametrize the set of all such pairs $(I, C)$
as well as the cohomological degrees
by a lattice $\M$.
As it turns out, any continuous function $f \colon X \to \R$ induces
a contravariant functor from $\M$ to the category of vector spaces,
with some of the internal maps induced by inclusions and the other maps
being differentials of a corresponding Mayer--Vietoris sequence.
The existence of these differentials is one of our motives
to consider preimages of open subsets as opposed to closed subsets of $\R$.
We refer to this functor as the
\emph{relative interlevel set cohomology (RISC)} of $f$
when $f \colon X \rightarrow \R$ is $K$-tame
and we show that it
satisfies certain exactness properties.
We call such functors \emph{cohomological};
this is \cref{dfn:cohomological}.
Furthermore, we show in \cref{lem:pfd,prp:cont}
that the RISC is pfd and \emph{sequentially continuous} (\cref{dfn:cont}).
As a byproduct,
any cohomology class from the RISC determines a natural transformation
from an indecomposable and vice versa.
Dually, natural transformations from a corresponding
homological construction to a \emph{sequentially cocontinuous} indecomposable
of a certain kind
are one-to-one with elements of the dual space of a corresponding
homology group.
However, this dual space is naturally isomorphic to cohomology.
Thus, cohomology even appears in the analogous construction
of a decomposition of the corresponding homological invariant.
This is part of the reason why we work with cohomology in place of homology.
As noted above, this
is no limitation in our context.

In \cref{sec:decomp} we show in \cref{thm:decomp}
that any pfd sequentially continuous cohomological functor
decomposes into a direct sum of indecomposables of a certain type.
Each indecomposable can be characterized by its support,
which is a maximal axis-aligned rectangle
as shown in \cref{fig:contraBlock}.
A posteriori, the upper left vertex of this rectangle
gives the corresponding vertex in the
\emph{extended persistence diagram}
as we define
it in \cref{dfn:diagramFunction}.
This close relationship between the indecomposables
and the extended persistence diagram
as well as the fact that there is just one type of indecomposable
was a major motivation for us
to glue the layers of the Mayer--Vietoris pyramid
to a single diagram.
We note that at this point,
one may also invoke \cite[Theorem 2.11]{botnan2020local}
in place of \cref{thm:decomp} to obtain a decomposition
of the RISC.
We are convinced that our proof of \cref{thm:decomp}
is relevant nevertheless,
as it is comparatively simple and more elementary
than the proof of \cite[Theorem 2.11]{botnan2020local}.
We also note that one may obtain
\emph{interlevel set cohomology} from RISC by restriction
to a subposet of $\M$.
Thus, under the assumption that $f \colon X \rightarrow \R$ is $K$-tame,
its interlevel set cohomology decomposes as well.
Similar results have been shown by
\cite[Section 9.3]{MR4057439}, \cite[Section 5]{MR4143378},
and \cite[Theorem 2.19]{2019arXiv190709759B}.

In \cref{sec:interleavings} we use the framework provided by
\cite[Section 2.5]{MR3413628}
to define the notion of an interleaving for contravariant functors on $\M$;
this is \cref{dfn:interleaving}.
Moreover, we show a stability result with \cref{thm:stability}.
In order to prove this theorem,
we cannot apply the framework by \cite{MR3413628} directly.
The reason for this is that the canonical \emph{(super)linear family}
on the indexing poset $\M$
does not preserve the \enquote{cohomological degree}.
As a result, the interleaving homomorphisms
will map some cohomology classes to a cohomology class
of one degree higher.
Resolving these subtleties requires us to study
the interplay of the elementary geometry of $\M$
and the relative interlevel set cohomology.
Furthermore, we provide \cref{exm:hood},
which shows that the induced interleavings of RISC
capture more information than the corresponding interleavings
of extended persistence or Mayer--Vietoris systems.
We end this paper with an elementary form
of \enquote{quantitative homotopy invariance} for RISC;
this is \cref{prp:quantHomotopyInv}.
The last two sections \ref{sec:decomp} and \ref{sec:interleavings}
can be read independently.

\section{The Relative Interlevel Set Cohomology}
\label{sec:risc}

We start with specifying the indexing poset $\M$
for the relative interlevel set cohomology.
To this end, let $\R$ and $\R^{\circ}$ denote the posets given by
the orders $\leq$ and $\geq$ on $\R$, respectively.
Then we may form the product poset $\R^{\circ} \times \R$,
which is a lattice and
whose underlying set is the Euclidean plane.
Let $l_0$ and $l_1$ be two lines of slope $-1$
in $\R^{\circ} \times \R$
with $l_1$ sitting above $l_0$
as shown in \cref{fig:incidenceT}.
Then the indexing poset is the sublattice $\M \subset \R^{\circ} \times \R$
given by the convex hull of $l_0$ and $l_1$.
The degree-shift in cohomology will correspond to the
\href{
  https://en.wikipedia.org/wiki/Center_(group_theory)
}{central}\footnote{
  By a central automorphism we mean an automorphism
  that commutes with any other lattice automorphism of $\M$.
}
automorphism $T \colon \M \rightarrow \M$
with the following defining property
(also see \cref{fig:incidenceT}):
\begin{quote}
  Let $u \in \M$,
  $h_0$ be the horizontal line through $u$,
  let $g_0$ be the vertical line through $u$,
  let $h_1$ be the horizontal line through $T(u)$, and let
  $g_1$ be the vertical line through $T(u)$.
  Then the lines $l_0$, $h_0$, and $g_1$ intersect in a common point, and
  the same is true for the lines $l_1$, $g_0$, and $h_1$.  
\end{quote}
We also note that $T$ is a glide reflection along
the bisecting line between $l_0$ and $l_1$,
and the amount of translation is the distance
of $l_0$ and $l_1$.
Moreover, as a space, $\M / \langle T \rangle$ is a Möbius strip;
see also \cite{Carlsson:2009:ZPH:1542362.1542408}.

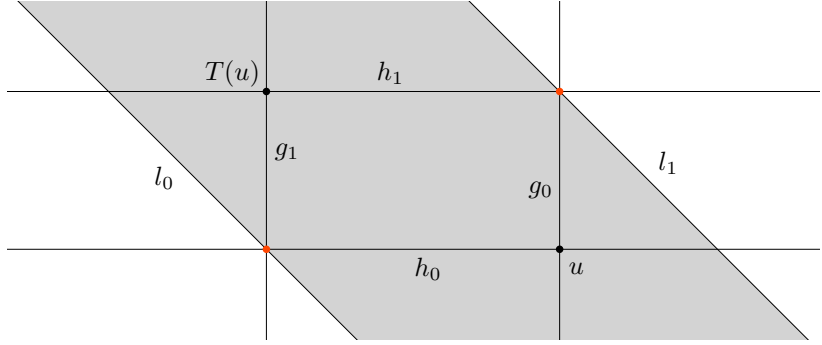
\begin{figure}[t]
  \centering
\begingroup%
  \makeatletter%
  \providecommand\color[2][]{%
    \errmessage{(Inkscape) Color is used for the text in Inkscape, but the package 'color.sty' is not loaded}%
    \renewcommand\color[2][]{}%
  }%
  \providecommand\transparent[1]{%
    \errmessage{(Inkscape) Transparency is used (non-zero) for the text in Inkscape, but the package 'transparent.sty' is not loaded}%
    \renewcommand\transparent[1]{}%
  }%
  \providecommand\rotatebox[2]{#2}%
  \newcommand*\fsize{\dimexpr\f@size pt\relax}%
  \newcommand*\lineheight[1]{\fontsize{\fsize}{#1\fsize}\selectfont}%
  \ifx\svgwidth\undefined%
    \setlength{\unitlength}{306bp}%
    \ifx\svgscale\undefined%
      \relax%
    \else%
      \setlength{\unitlength}{\unitlength * \real{\svgscale}}%
    \fi%
  \else%
    \setlength{\unitlength}{\svgwidth}%
  \fi%
  \global\let\svgwidth\undefined%
  \global\let\svgscale\undefined%
  \makeatother%
  \begin{picture}(1,0.41666667)%
    \lineheight{1}%
    \setlength\tabcolsep{0pt}%
    \put(0,0){\includegraphics[width=\unitlength,page=1]{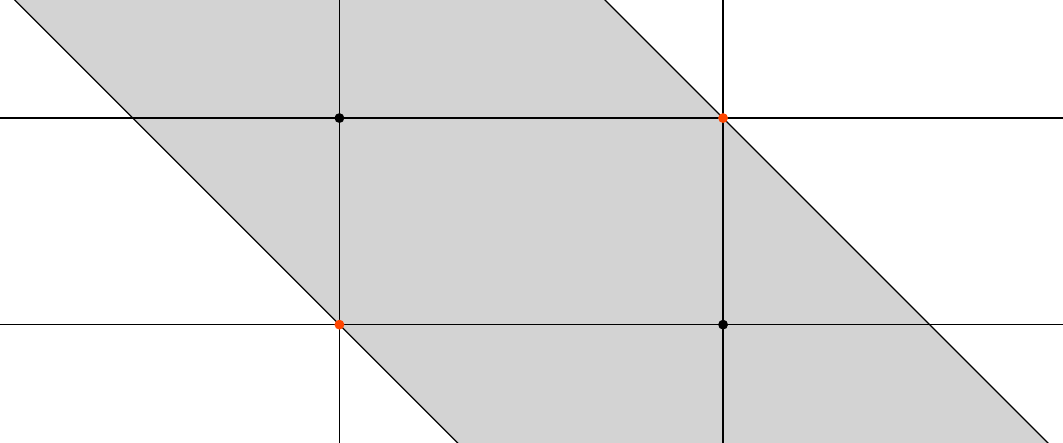}}%
    \put(0.34490735,0.22685196){\makebox(0,0)[t]{\lineheight{1.25}\smash{\begin{tabular}[t]{c}$g_1$\end{tabular}}}}%
    \put(0.4722223,0.31944436){\makebox(0,0)[t]{\lineheight{1.25}\smash{\begin{tabular}[t]{c}$h_1$\end{tabular}}}}%
    \put(0.65740735,0.18055564){\makebox(0,0)[t]{\lineheight{1.25}\smash{\begin{tabular}[t]{c}$g_0$\end{tabular}}}}%
    \put(0.51851863,0.07870368){\makebox(0,0)[t]{\lineheight{1.25}\smash{\begin{tabular}[t]{c}$h_0$\end{tabular}}}}%
    \put(0.2777777,0.31944436){\makebox(0,0)[t]{\lineheight{1.25}\smash{\begin{tabular}[t]{c}$T(u)$\end{tabular}}}}%
    \put(0.70138897,0.08101863){\makebox(0,0)[t]{\lineheight{1.25}\smash{\begin{tabular}[t]{c}$u$\end{tabular}}}}%
    \put(0.81454216,0.20860588){\makebox(0,0)[t]{\lineheight{1.25}\smash{\begin{tabular}[t]{c}$l_1$\end{tabular}}}}%
    \put(0.19392279,0.19033652){\makebox(0,0)[t]{\lineheight{1.25}\smash{\begin{tabular}[t]{c}$l_0$\end{tabular}}}}%
  \end{picture}%
\endgroup%

  \caption{
    Incidences defining $T$.
    The indexing poset $\M$ is shaded in grey.
  }
  \label{fig:incidenceT}
\end{figure}

\begin{figure}[t]
  \centering
\begingroup%
  \makeatletter%
  \providecommand\color[2][]{%
    \errmessage{(Inkscape) Color is used for the text in Inkscape, but the package 'color.sty' is not loaded}%
    \renewcommand\color[2][]{}%
  }%
  \providecommand\transparent[1]{%
    \errmessage{(Inkscape) Transparency is used (non-zero) for the text in Inkscape, but the package 'transparent.sty' is not loaded}%
    \renewcommand\transparent[1]{}%
  }%
  \providecommand\rotatebox[2]{#2}%
  \newcommand*\fsize{\dimexpr\f@size pt\relax}%
  \newcommand*\lineheight[1]{\fontsize{\fsize}{#1\fsize}\selectfont}%
  \ifx\svgwidth\undefined%
    \setlength{\unitlength}{306bp}%
    \ifx\svgscale\undefined%
      \relax%
    \else%
      \setlength{\unitlength}{\unitlength * \real{\svgscale}}%
    \fi%
  \else%
    \setlength{\unitlength}{\svgwidth}%
  \fi%
  \global\let\svgwidth\undefined%
  \global\let\svgscale\undefined%
  \makeatother%
  \begin{picture}(1,0.41666667)%
    \lineheight{1}%
    \setlength\tabcolsep{0pt}%
    \put(0,0){\includegraphics[width=\unitlength,page=1]{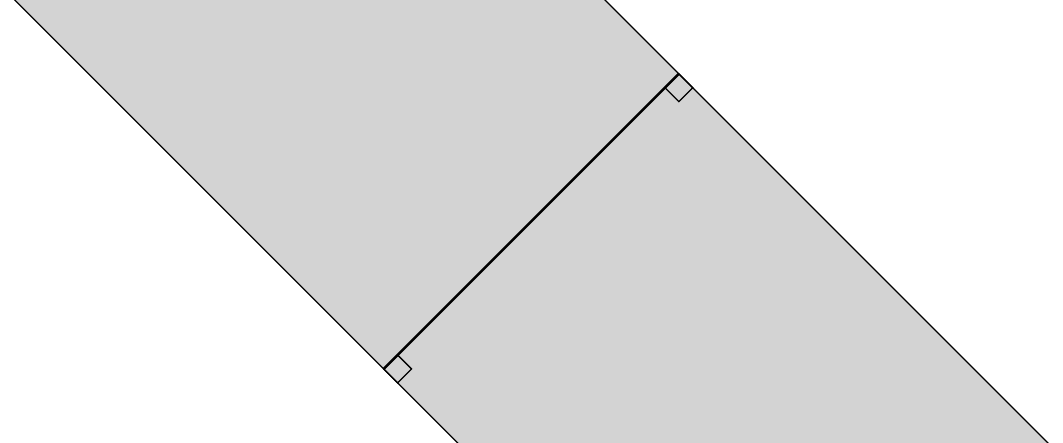}}%
    \put(0.41203701,0.26851863){\makebox(0,0)[t]{\lineheight{1.25}\smash{\begin{tabular}[t]{c}$\M$\end{tabular}}}}%
    \put(0.7777777,0.23148137){\makebox(0,0)[t]{\lineheight{1.25}\smash{\begin{tabular}[t]{c}$ \pi $\end{tabular}}}}%
    \put(0.2222223,0.17129632){\makebox(0,0)[t]{\lineheight{1.25}\smash{\begin{tabular}[t]{c}$-\pi~~$\end{tabular}}}}%
    \put(0,0){\includegraphics[width=\unitlength,page=2]{embedding-reals.pdf}}%
    \put(0.54629632,0.17129632){\makebox(0,0)[t]{\lineheight{1.25}\smash{\begin{tabular}[t]{c}$\op{Im} \blacktriangle$\end{tabular}}}}%
    \put(0.72288015,0.30026789){\makebox(0,0)[t]{\lineheight{1.25}\smash{\begin{tabular}[t]{c}$l_1$\end{tabular}}}}%
    \put(0.2855848,0.09867451){\makebox(0,0)[t]{\lineheight{1.25}\smash{\begin{tabular}[t]{c}$l_0$\end{tabular}}}}%
  \end{picture}%
\endgroup%

  \caption{
    The strip $\M$ and
    the image of the embedding
    $\blacktriangle \colon \overline{\R} \rightarrow \M$.
  }
  \label{fig:embeddingReals}
\end{figure}

The region of $\M$ indexing the Mayer--Vietoris pyramid in degree $0$
yields a fundamental domain $D$ with respect to the action of
$\langle T \rangle \cong \Z$ on $\M$,
which we specify now.
To this end,
suppose $l_0$ and $l_1$ intersect the $x$-axis in $-\pi$ and $\pi$,
respectively.
With this we embed the extended reals $\overline{\R} := [-\infty, \infty]$
into the strip $\M$
by precomposing the diagonal map $\Delta \colon {\R} \to {\R}^2,
t \mapsto (t,t)$
with the homeomorphism
$\arctan: \overline{\R} \to [-\pi/2,\pi/2]$,
yielding a map
\[{\blacktriangle = \Delta \circ \arctan \colon \overline{\R} \rightarrow \M,\,
t \mapsto (\arctan t, \arctan t)}\]
such that $\op{Im} \blacktriangle$ is a perpendicular
line segment through the origin joining $l_0$ and $l_1$,
see \cref{fig:embeddingReals}.
We specify the fundamental domain as
shown in \cref{fig:defFundamentalDomain} by
\begin{equation*}
  D :=
  (\downarrow \op{Im} \blacktriangle) \setminus
  T^{-1} (\downarrow \op{Im} \blacktriangle)
  ,
\end{equation*}
where $\downarrow \op{Im} \blacktriangle$ is the downset
of the image of $\blacktriangle$.
\begin{figure}[t]
  \centering
\begingroup%
  \makeatletter%
  \providecommand\color[2][]{%
    \errmessage{(Inkscape) Color is used for the text in Inkscape, but the package 'color.sty' is not loaded}%
    \renewcommand\color[2][]{}%
  }%
  \providecommand\transparent[1]{%
    \errmessage{(Inkscape) Transparency is used (non-zero) for the text in Inkscape, but the package 'transparent.sty' is not loaded}%
    \renewcommand\transparent[1]{}%
  }%
  \providecommand\rotatebox[2]{#2}%
  \newcommand*\fsize{\dimexpr\f@size pt\relax}%
  \newcommand*\lineheight[1]{\fontsize{\fsize}{#1\fsize}\selectfont}%
  \ifx\svgwidth\undefined%
    \setlength{\unitlength}{246.50001526bp}%
    \ifx\svgscale\undefined%
      \relax%
    \else%
      \setlength{\unitlength}{\unitlength * \real{\svgscale}}%
    \fi%
  \else%
    \setlength{\unitlength}{\svgwidth}%
  \fi%
  \global\let\svgwidth\undefined%
  \global\let\svgscale\undefined%
  \makeatother%
  \begin{picture}(1,0.51724135)%
    \lineheight{1}%
    \setlength\tabcolsep{0pt}%
    \put(0,0){\includegraphics[width=\unitlength,page=1]{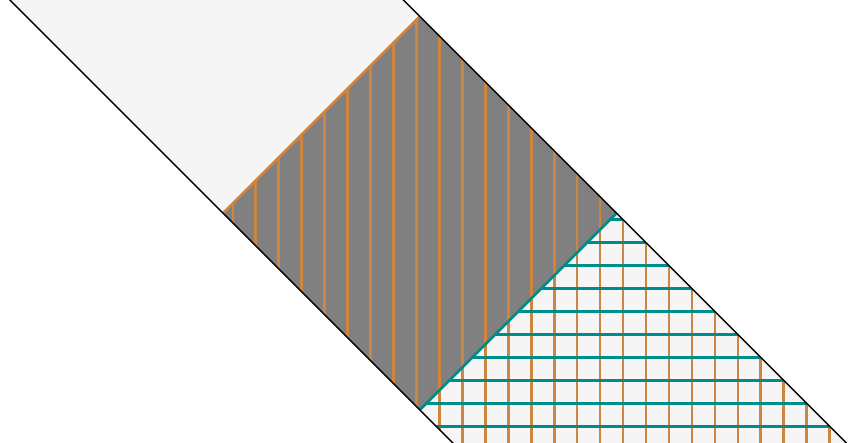}}%
    \put(0.66052876,0.3647586){\makebox(0,0)[t]{\lineheight{1.25}\smash{\begin{tabular}[t]{c}$l_1$\end{tabular}}}}%
    \put(0.16168965,0.32106905){\makebox(0,0)[t]{\lineheight{1.25}\smash{\begin{tabular}[t]{c}$l_0$\end{tabular}}}}%
  \end{picture}%
\endgroup%

  \caption{
    The fundamental domain
    $\textcolor{DimGrey}{D} :=
    \textcolor{Peru}{\downarrow \op{Im} \blacktriangle} \setminus
    \textcolor{DarkCyan}{T^{-1} (\downarrow \op{Im} \blacktriangle)}$.
  }
  \label{fig:defFundamentalDomain}
\end{figure}
\cref{fig:tessellation} shows the tessellation of $\M$
induced by $T$ and $D$.

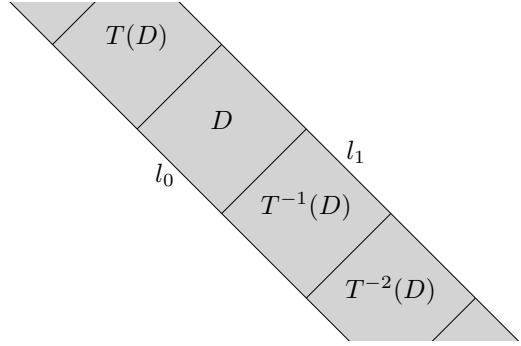
\begin{figure}[t]
  \centering
\begingroup%
  \makeatletter%
  \providecommand\color[2][]{%
    \errmessage{(Inkscape) Color is used for the text in Inkscape, but the package 'color.sty' is not loaded}%
    \renewcommand\color[2][]{}%
  }%
  \providecommand\transparent[1]{%
    \errmessage{(Inkscape) Transparency is used (non-zero) for the text in Inkscape, but the package 'transparent.sty' is not loaded}%
    \renewcommand\transparent[1]{}%
  }%
  \providecommand\rotatebox[2]{#2}%
  \newcommand*\fsize{\dimexpr\f@size pt\relax}%
  \newcommand*\lineheight[1]{\fontsize{\fsize}{#1\fsize}\selectfont}%
  \ifx\svgwidth\undefined%
    \setlength{\unitlength}{194.4375bp}%
    \ifx\svgscale\undefined%
      \relax%
    \else%
      \setlength{\unitlength}{\unitlength * \real{\svgscale}}%
    \fi%
  \else%
    \setlength{\unitlength}{\svgwidth}%
  \fi%
  \global\let\svgwidth\undefined%
  \global\let\svgscale\undefined%
  \makeatother%
  \begin{picture}(1,0.6557377)%
    \lineheight{1}%
    \setlength\tabcolsep{0pt}%
    \put(0,0){\includegraphics[width=\unitlength,page=1]{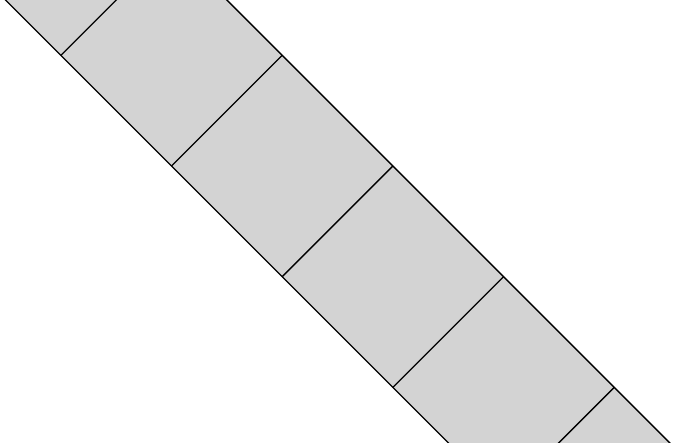}}%
    \put(0.74590164,0.08196721){\makebox(0,0)[t]{\lineheight{1.25}\smash{\begin{tabular}[t]{c}$T^{-2}(D)$\end{tabular}}}}%
    \put(0.58196721,0.24590164){\makebox(0,0)[t]{\lineheight{1.25}\smash{\begin{tabular}[t]{c}$T^{-1}(D)$\end{tabular}}}}%
    \put(0.25409836,0.57377049){\makebox(0,0)[t]{\lineheight{1.25}\smash{\begin{tabular}[t]{c}$T(D)$\end{tabular}}}}%
    \put(0.41803279,0.40983607){\makebox(0,0)[t]{\lineheight{1.25}\smash{\begin{tabular}[t]{c}$D$\end{tabular}}}}%
    \put(0.6789188,0.35222874){\makebox(0,0)[t]{\lineheight{1.25}\smash{\begin{tabular}[t]{c}$l_1$\end{tabular}}}}%
    \put(0.3084351,0.306319){\makebox(0,0)[t]{\lineheight{1.25}\smash{\begin{tabular}[t]{c}$l_0$\end{tabular}}}}%
  \end{picture}%
\endgroup%

  \caption{The tessellation of $\M$ induced by $T$ and $D$.}
  \label{fig:tessellation}
\end{figure}

Now each point in the Mayer--Vietoris pyramid
corresponds to a pair of subspaces,
which is a preimage of a pair of subspaces of $\R$.
To specify such a pair for each point in $D$
the following proposition characterizes a monotone map $\rho$
from $\M$ to the poset of pairs of open subspaces of $\R$,
which is locally constant on $\M \setminus D$;
a schematic image of $\rho$ is shown in \cref{fig:rho}.

\begin{prp}
  \label{prp:rho}
  Let $\mathcal{P}$ denote the set of pairs of open subspaces
  of $\R$.
  Then there is a unique monotone map
  \begin{equation*}
    \rho = (\rho_1, \rho_0) \colon \M \rightarrow \mathcal{P}
  \end{equation*}
  with the following four properties:
  \begin{enumerate}[(1)]
  \item
    For any $t \in \R$ we have
    $(\rho \circ \blacktriangle)(t) = (\R, \R \setminus \{t\})$.
  \item
    For any $u \in \partial \M$
    we have $\rho_1 (u) = \rho_0 (u)$.
  \item
    For any axis-aligned rectangle contained in $\uparrow D$
    the corresponding joins and meets are preserved by $\rho_1$.
  \item
    For any axis-aligned rectangle contained in $\downarrow D$
    the corresponding joins and meets are preserved by $\rho_0$.
  \end{enumerate}
\end{prp}

\begin{figure}[t]
  \centering
  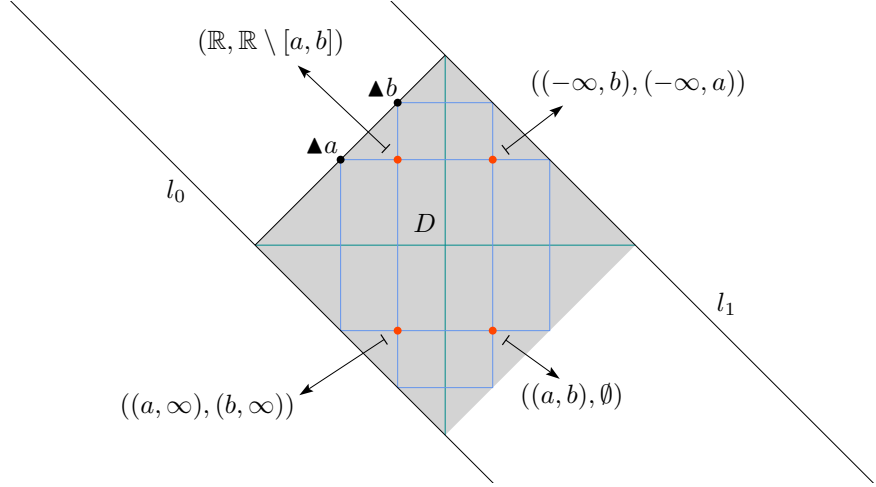
  \caption{A schematic image of $\rho$.}
  \label{fig:rho}
\end{figure}

Moreover, we have the explicit formula
\begin{equation*}
  \rho(u) =
  \R \cap \blacktriangle^{-1} \left(
    \op{int} (\downarrow T(u)), \M \setminus \uparrow u
  \right)
  ,
\end{equation*}
where $\op{int} (\downarrow T(u))$ is the interior
of the downset of $T(u)$.

Now let $f \colon X \rightarrow \R$ be a continuous function
and let $\mathcal{H}^{\bullet}$ be a cohomology theory
with values in the category of graded vector spaces over a fixed field $K$,
which sends weak equivalences to isomorphisms.
Then we obtain the contravariant functor
\begin{equation*}
  F' \colon D \rightarrow \left(\mathrm{Vect}_K^{\Z}\right)^\circ,
  u \mapsto \mathcal{H}^{\bullet} (f^{-1}(\rho(u)))
  ,
\end{equation*}
where $\mathrm{Vect}_K^{\Z}$ is the category of $\Z$-graded vector spaces
over $K$
and the circle $\circ$ as an exponent is used to denote
the corresponding opposite category.
Now the degree-shift yields the endofunctor
\begin{equation*}
  \Sigma \colon
  \left(\mathrm{Vect}_K^{\Z}\right)^\circ \rightarrow
  \left(\mathrm{Vect}_K^{\Z}\right)^\circ,
  M^{\bullet} \mapsto M^{\bullet-1}
\end{equation*}
on graded vector spaces.
We intend to replace the contravariant functor $F'$ from $D$ to the category
of graded vector spaces with a contravariant functor from $\M$
to mere vector spaces $\mathrm{Vect}_K$
carrying the same (and more) information,
with precomposition by $T$ taking the place of $\Sigma$.
As an intermediate step,
we extend $F'$ to a functor
\begin{equation*}
  F \colon \M \rightarrow \left(\mathrm{Vect}_K^{\Z}\right)^\circ
  ,
\end{equation*}
which is $\Z$-equivariant or strictly stable in the sense that
\begin{equation}
  \label{eq:strictlyStable}
  F \circ T = \Sigma \circ F
  .
\end{equation}
Now as a map into the objects of $\mathrm{Vect}_K^{\Z}$,
such a functor $F$ carries no new information
in comparison to $F'$.
Moreover, by \eqref{eq:strictlyStable} most of the information
carried by $F$ is redundant and we may discard all redundant information
by post-composition with the evaluation at $0$:
\begin{equation*}
  \op{ev}^0 \colon \mathrm{Vect}_K^{\Z} \rightarrow \mathrm{Vect}_K, \,
  M^{\bullet} \mapsto M^0
  ;
\end{equation*}
see also \cref{lem:2adj}.

\begin{figure}[t]
  \centering
\begingroup%
  \makeatletter%
  \providecommand\color[2][]{%
    \errmessage{(Inkscape) Color is used for the text in Inkscape, but the package 'color.sty' is not loaded}%
    \renewcommand\color[2][]{}%
  }%
  \providecommand\transparent[1]{%
    \errmessage{(Inkscape) Transparency is used (non-zero) for the text in Inkscape, but the package 'transparent.sty' is not loaded}%
    \renewcommand\transparent[1]{}%
  }%
  \providecommand\rotatebox[2]{#2}%
  \newcommand*\fsize{\dimexpr\f@size pt\relax}%
  \newcommand*\lineheight[1]{\fontsize{\fsize}{#1\fsize}\selectfont}%
  \ifx\svgwidth\undefined%
    \setlength{\unitlength}{307.5bp}%
    \ifx\svgscale\undefined%
      \relax%
    \else%
      \setlength{\unitlength}{\unitlength * \real{\svgscale}}%
    \fi%
  \else%
    \setlength{\unitlength}{\svgwidth}%
  \fi%
  \global\let\svgwidth\undefined%
  \global\let\svgscale\undefined%
  \makeatother%
  \begin{picture}(1,0.48780488)%
    \lineheight{1}%
    \setlength\tabcolsep{0pt}%
    \put(0,0){\includegraphics[width=\unitlength,page=1]{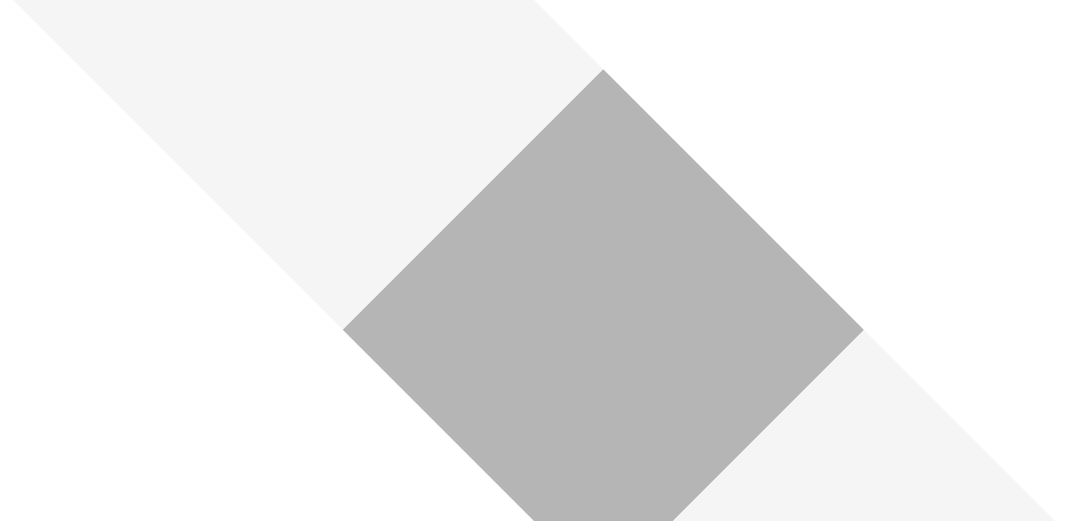}}%
    \put(0.36585366,0.34959341){\makebox(0,0)[t]{\lineheight{1.25}\smash{\begin{tabular}[t]{c}$\hat{u}$\end{tabular}}}}%
    \put(0.27439024,0.40243902){\makebox(0,0)[t]{\lineheight{1.25}\smash{\begin{tabular}[t]{c}$T(D)$\end{tabular}}}}%
    \put(0.51219512,0.2195122){\makebox(0,0)[t]{\lineheight{1.25}\smash{\begin{tabular}[t]{c}$w$\end{tabular}}}}%
    \put(0.67886171,0.24390244){\makebox(0,0)[t]{\lineheight{1.25}\smash{\begin{tabular}[t]{c}$v_2$\end{tabular}}}}%
    \put(0.49187,0.08943098){\makebox(0,0)[t]{\lineheight{1.25}\smash{\begin{tabular}[t]{c}$v_1$\end{tabular}}}}%
    \put(0.66666659,0.09349585){\makebox(0,0)[t]{\lineheight{1.25}\smash{\begin{tabular}[t]{c}$u$\end{tabular}}}}%
    \put(0,0){\includegraphics[width=\unitlength,page=2]{constr-boundary.pdf}}%
    \put(0.77409512,0.25273415){\makebox(0,0)[t]{\lineheight{1.25}\smash{\begin{tabular}[t]{c}$l_1$\end{tabular}}}}%
    \put(0.23334146,0.21787805){\makebox(0,0)[t]{\lineheight{1.25}\smash{\begin{tabular}[t]{c}$l_0$\end{tabular}}}}%
  \end{picture}%
\endgroup%

  \caption{
    The axis-aligned rectangle $u \preceq v_1, v_2 \preceq w \in D$
    determined by $(w, \hat{u}) \in R_D$.
  }
  \label{fig:constrBoundary}
\end{figure}

Now in order to obtain such a strictly stable functor $F$
from $F'$ we need to glue consecutive layers
using connecting homomorphisms.
To this end, let
\[R_D :=
  \{
  (w, \hat{u}) \in D \times T(D) \mid
  w \preceq \hat{u} \preceq T(w)
  \}
\]
as in \cref{dfn:R_D} in \cref{sec:stableFun}.
As shown in \cref{fig:constrBoundary}, any pair $(w, \hat{u}) \in R_D$
determines an axis-aligned rectangle $u \preceq v_1, v_2 \preceq w \in D$
with $T(u) = \hat{u}$.
Moreover,
as this rectangle is contained in $D$,
the corresponding join $w = v_1 \vee v_2$ and meet $u = v_1 \wedge v_2$
are preserved by $\rho$ by \cref{prp:rho}.(3-4).
Furthermore, since taking preimages is a homomorphism of boolean algebras,
$f^{-1}$ also preserves joins and meets,
which in this case are the componentwise unions and intersections.
This means that $f^{-1} (\rho(u))$ is the componentwise intersection
of $f^{-1} (\rho(v_1))$
and $f^{-1} (\rho(v_2))$,
while $f^{-1} (\rho(w))$ is their union.
With this we obtain the triad
$f^{-1} (\rho(w); \rho(v_1), \rho(v_2))$
of pairs of open subsets of $X$,
which is excisive in each component by excision.
Thus, we have the differential
\begin{equation*}
  \delta'_{(w, \hat{u})} \colon
  \left(\mathcal{H}^{\bullet-1} \circ f^{-1} \circ \rho\right)(u)
  \rightarrow
  \left(\mathcal{H}^{\bullet} \circ f^{-1} \circ \rho\right)(w)
\end{equation*}
of the corresponding Mayer--Vietoris sequence
as described in \cite[Section 17.1.4]{MR2456045}.
Now let
\begin{equation*}
  \partial'_{(w, \hat{u})} :=
  \big(\delta'_{(w, \hat{u})}\big)^{\circ} \colon
  \left(\mathcal{H}^{\bullet} \circ f^{-1} \circ \rho\right)(w) \rightarrow
  \left(\mathcal{H}^{\bullet-1} \circ f^{-1} \circ \rho\right)(u)
\end{equation*}
be the corresponding homomorphism
in the opposite category $\left(\mathrm{Vect}_K^{\Z}\right)^\circ$
for all $(w, \hat{u}) \in R_D$.
Moreover, let
$\op{pr}_1 \colon R_D \rightarrow D$
and $\op{pr}_2 \colon R_D \rightarrow T(D)$
be the projections to the first and the second component, respectively.
Then $\partial'$ is a natural transformation as in the diagram
\begin{equation*}
  \begin{tikzcd}
    R_D
    \arrow[rrr, "\op{pr}_1"]
    \arrow[ddd, "\op{pr}_2"']
    &[-18pt] & &[-23pt]
    D
    \arrow[ddd, "F'"]
    \\[-15pt]
    & &
    {}
    \arrow[ld, Rightarrow, "{\partial'}"]
    \\
    &
    {}
    \\[-21pt]
    T(D)
    \arrow[rrr, "\Sigma \circ F' \circ T^{-1}"']
    & & &
    \left(\mathrm{Vect}_K^{\Z}\right)^\circ
    .
  \end{tikzcd}
\end{equation*}
Thus,
the functor
$F' \colon D \rightarrow \left(\mathrm{Vect}_K^{\Z}\right)^\circ$
and $\partial'$ determine a unique strictly stable functor
$F \colon \M \rightarrow \left(\mathrm{Vect}_K^{\Z}\right)^\circ$
by \cref{prp:fundaExt}.
\sloppy
To obtain a functor of type
${\M^{\circ} \rightarrow \mathrm{Vect}_K}$ from $F$ we post-compose
${F^{\circ} \colon \M^{\circ} \rightarrow \mathrm{Vect}_K^{\Z}}$
with the evaluation at $0$
and we define
\begin{equation*}
  h(f) := \op{ev}^0 \circ \, F^{\circ}
  .
\end{equation*}

\begin{dfn}
  We say that $f \colon X \rightarrow \R$ is \emph{$\mathcal{H}^{\bullet}$-tame}
  if all open interlevel sets of $f$
  have finite-dimensional cohomology in each degree, i.e.
  \begin{equation*}
    \dim_K \mathcal{H}^n \big(f^{-1}(I)\big) < \infty
  \end{equation*}
  for any integer $n \in \Z$ and any open interval $I \subseteq \R$.
  Moreover, if $\mathcal{H}^{\bullet}$ is singular cohomology
  with coefficients in $K$,
  we say that $f$ is \emph{$K$-tame}.
\end{dfn}

If $\mathcal{H}^{\bullet}$ is singular cohomology with coefficients in $K$
and if $f \colon X \rightarrow \R$ is $K$-tame,
then we name $h(f)$
the \emph{relative interlevel set cohomology (RISC) of $f$
  with coefficients in $K$}.
Even though $h(f) \colon \M^{\circ} \rightarrow \VectK$
is well-defined for any continuous function
$f \colon X \rightarrow \R$,
we refer to $h(f)$ as the relative interlevel set cohomology
only if $f$ is $K$-tame.
This assumption is needed in order for $h(f)$ to be
\emph{sequentially continuous} as we will see
in \cref{dfn:cont,prp:cont} below.
In case $f \colon X \rightarrow \R$ is not $K$-tame,
then it may be more reasonable to consider the \enquote{reflection}
of $h(f)$ into the full subcategory of sequentially continuous functors.
For this reason, we refrain from referring to $h(f)$ as the RISC of $f$,
when $f$ is not $K$-tame.

We note that $h$ extends to a contravariant functor
from the category of spaces over the reals $\R$
to the category of contravariant functors on $\M$
in the following way.
For a commutative triangle
\begin{equation*}
  \begin{tikzcd}
    X
    \arrow[rr, bend left, "\varphi"]
    \arrow[dr, "f"']
    &
    &
    Y
    \arrow[dl, "g"]
    \\
    &
    \R
  \end{tikzcd}
\end{equation*}
of topological spaces,
the map $\varphi$ yields a continuous map of pairs
\begin{equation*}
  \left(f^{-1} \circ \rho\right)(u)
  \rightarrow
  \left(g^{-1} \circ \rho\right)(u)
  ,
\end{equation*}
which is natural in $u \in D$.
By the functoriality of $\mathcal{H}^{\bullet}$
and the naturality of the Mayer--Vietoris sequence
the collection of these maps induces a natural transformation
\begin{equation*}
  h(\varphi) \colon h(g) \rightarrow h(f)
  .
\end{equation*}

By construction any axis-aligned rectangle
$u \preceq v_1, v_2 \preceq w \in D$
as shown in \cref{fig:constrBoundary}
yields a long exact sequence
\begin{equation*}
  \begin{tikzcd}
    &
    \cdots
    \arrow[r]
    &
    h(f)(T(u))
    \arrow[dll, out=0, in=180, looseness=2, overlay]
    \\
    h(f)(w)
    \arrow[r]
    &
    h(f)(v_1) \oplus h(f)(v_2)
    \arrow[r, "(1 ~\, -1)"]
    &
    h(f)(u)
    \arrow[dll, out=0, in=180, looseness=2, overlay]
    \\
    h(f)(T^{-1}(w))
    \arrow[r]
    &
    \cdots
    .
  \end{tikzcd}
\end{equation*}
By \cref{prp:cohomological}.1 this is one way of characterizing
\emph{cohomological functors} on $\M$,
which we define in \cref{dfn:cohomological}.
We note that the characterization \cref{prp:homological}.4
has been stated by \cite{Carlsson:2009:ZPH:1542362.1542408}
and proven by \cite{Carlsson:2009:ZPH:1542362.1542408,MR3031814}
for any axis-aligned rectangle contained within a tile
of the tessellation shown in \cref{fig:tessellation}.


\begin{lem}
  \label{lem:pfd}
  The function $f \colon X \rightarrow \R$ is $\mathcal{H}^{\bullet}$-tame
  iff the functor
  $h(f) \colon \M^{\circ} \rightarrow \mathrm{Vect}_K$
  is pointwise finite-dimensional (pfd).
\end{lem}

\begin{proof}
  As ${\mathcal{H}^n \big(f^{-1}(I)\big)}$
  appears as a value of ${h(f) \colon \M^{\circ} \rightarrow \mathrm{Vect}_K}$
  for any open interval ${I \subseteq \R}$
  and any integer ${n \in \Z}$,
  the function $f$ is $\mathcal{H}^{\bullet}$-tame
  if $h(f)$ is pfd.
  Now suppose ${f \colon X \rightarrow \R}$ is $\mathcal{H}^{\bullet}$-tame,
  let ${u \in T^{-n}(D)}$ for some ${n \in \Z}$,
  let $X_u$ be the absolute component
  of ${\left(f^{-1} \circ \rho \circ T^n\right)(u)}$,
  and let $A_u$ be the relative component.
  Then we obtain the exact sequence
  \begin{equation*}
    \mathcal{H}^{n-1}(X_u) \rightarrow
    \mathcal{H}^{n-1}(A_u) \rightarrow
    h(f)(u) \rightarrow
    \mathcal{H}^{n}(X_u) \rightarrow
    \mathcal{H}^{n}(A_u)
    .
  \end{equation*}
  Now $X_u$ is an open interlevel set of $f$ and $A_u$
  is the disjoint union of at most two open interlevel sets.
  As $f$ is $\mathcal{H}^{\bullet}$-tame
  all four cohomology groups surrounding $h(f)(u)$
  in above exact sequence are finite-dimensional.
  As a result, $h(f)(u)$ is finite-dimensional as well.
\end{proof}

We end this section by showing that
$h(f) \colon \M^{\circ} \rightarrow \mathrm{Vect}_K$
satisfies the following form of continuity,
when $f \colon X \rightarrow \R$ is $\mathcal{H}^{\bullet}$-tame.

\begin{dfn}
  \label{dfn:cont}
  We say that a contravariant functor
  $F \colon \M^{\circ} \rightarrow \mathrm{Vect}_K$
  is \emph{sequentially continuous},
  if for any increasing sequence
  $(u_k)_{k=1}^{\infty}$ in $\M$
  converging to $u$
  the natural map
  \begin{equation}
    \label{eq:seqLim}
    F(u) \rightarrow \varprojlim_{k} F(u_k)
  \end{equation}
  is an isomorphism.
  Dually, a covariant functor
  $F \colon \M \rightarrow \mathrm{Vect}_K$
  is \emph{sequentially continuous},
  if for any decreasing sequence
  $(u_k)_{k=1}^{\infty}$ in $\M$
  converging to $u$
  the natural map \eqref{eq:seqLim}
  is an isomorphism,
  see also \cref{remark:dual} below.
\end{dfn}

\begin{prp}
  \label{prp:cont}
  If $f \colon X \rightarrow \R$ is $\mathcal{H}^{\bullet}$-tame,
  then the functor
  $h(f) \colon \M^{\circ} \rightarrow \mathrm{vect}_K$
  is sequentially continuous.
\end{prp}

\begin{proof}
  Let $(u_k)_{k=1}^{\infty}$ be an increasing sequence in $\M$
  converging to $u$.
  Without loss of generality we assume that $(u_k)_{k=1}^{\infty}$
  is contained in a single tile $T^{-n}(D)$
  of the tessellation
  induced by $T$ and $D$ as shown in \cref{fig:tessellation}.
  We write $X_k$ for the absolute component
  of $(f^{-1} \circ \rho \circ T^n)(u_k)$
  and $A_k$ for the relative component.
  With this we have the commutative diagram
  \begin{equation}
    \label{eq:seqContLatter}
    \begin{tikzcd}[column sep=2.5ex]
      \mathcal{H}^{n-1} (\bigcup_k X_k)
      \arrow[r]
      \arrow[d]
      &
      \mathcal{H}^{n-1} (\bigcup_k A_k)
      \arrow[r]
      \arrow[d]
      &
      h(f)(u)
      \arrow[r]
      \arrow[d]
      &
      \mathcal{H}^n (\bigcup_k X_k)
      \arrow[r]
      \arrow[d]
      &
      \mathcal{H}^n (\bigcup_k A_k)
      \arrow[d]
      \\
      \displaystyle \varprojlim_k \mathcal{H}^{n-1} (X_k)
      \arrow[r]
      &
      \displaystyle \varprojlim_k \mathcal{H}^{n-1} (A_k)
      \arrow[r]
      &
      \displaystyle \varprojlim_k h(f)(u_k)
      \arrow[r]
      &
      \displaystyle \varprojlim_k \mathcal{H}^n (X_k)
      \arrow[r]
      &
      \displaystyle \varprojlim_k \mathcal{H}^n (A_k)
      .
    \end{tikzcd}
  \end{equation}
  Now for each $k \in \N$ the subspace $X_k \subseteq X$
  is an open interlevel set of $f \colon X \rightarrow \R$.
  Similarly, $A_k$ is a disjoint union of at most two open interlevel sets.
  As $f \colon X \rightarrow \R$ is $\mathcal{H}^{\bullet}$-tame
  the inverse sequences
  \begin{align}
    \label{eq:seqXk2}
    & \left(
      \mathcal{H}^{n-2} (X_{k+1}) \rightarrow \mathcal{H}^{n-2} (X_{k})
      \right)_{k=1}^{\infty} ,
    \\
    \label{eq:seqAk2}
    & \left(
      \mathcal{H}^{n-2} (A_{k+1}) \rightarrow \mathcal{H}^{n-2} (A_{k})
      \right)_{k=1}^{\infty} ,
    \\
    \label{eq:seqXk1}
    & \left(
      \mathcal{H}^{n-1} (X_{k+1}) \rightarrow \mathcal{H}^{n-1} (X_{k})
      \right)_{k=1}^{\infty} ,
    \\
    \label{eq:seqAk1}
    & \left(
      \mathcal{H}^{n-1} (A_{k+1}) \rightarrow \mathcal{H}^{n-1} (A_{k})
      \right)_{k=1}^{\infty} ,
    \\
    \label{eq:seqXk}
    & \left(
      \mathcal{H}^{n} (X_{k+1}) \rightarrow \mathcal{H}^{n} (X_{k})
      \right)_{k=1}^{\infty} ,
    \\
    \label{eq:seqAk}
    \text{and} \quad
    & \left(
      \mathcal{H}^{n} (A_{k+1}) \rightarrow \mathcal{H}^{n} (A_{k})
      \right)_{k=1}^{\infty}
  \end{align}
  are pfd.
  As the inverse sequences
  \eqref{eq:seqXk1}, \eqref{eq:seqAk1}, \eqref{eq:seqXk}, and \eqref{eq:seqAk}
  are pfd
  and
  as inverse limits of finite-dimensional vector spaces are exact,
  both rows of \eqref{eq:seqContLatter} are exact.
  Moreover,
  as
  \eqref{eq:seqXk2}, \eqref{eq:seqAk2}, \eqref{eq:seqXk1}, and \eqref{eq:seqAk1}
  are pfd,
  they satisfy the Mittag-Leffler condition.
  As a result,
  the four vertical maps
  surrounding
  ${h(f)(u) \rightarrow \varprojlim_{k} h(f)(u_k)}$
  in \eqref{eq:seqContLatter}
  are isomorphisms by \mbox{\cite[Section 19.4]{MR1702278}}.
  With this it follows from the five lemma that
  ${h(f)(u) \rightarrow \varprojlim_{k} h(f)(u_k)}$
  is an isomorphism as well.
\end{proof}

\section{Decomposition}
\label{sec:decomp}

Having defined the relative interlevel set cohomology
as a contravariant functor
\({h(f) \colon \M^{\circ} \rightarrow \mathrm{Vect}_K}\),
we now formalize the notion of an \emph{extended persistence diagram},
originally due to \cite{MR2472288},
as an invariant of sequentially continuous cohomological functors
$F \colon \M^{\circ} \rightarrow \mathrm{vect}_K$.
Here $\mathrm{vect}_K$ denotes the category
of finite-dimensional vector spaces over $K$.
The \emph{persistence diagram} $\op{Dgm} (F)$ is a multiset
$\mu \colon \M \rightarrow \N_0$, which counts,
for each point $u = (x, y) \in \M$, the maximal number $\mu(u)$
of linearly independent vectors in $F(u)$ born at $u$;
for the functor $h(f)$, these are cohomology classes.
Before we provide a more explicit definition,
we introduce some notation.
We note that a contravariant functor
$F \colon \M^{\circ} \rightarrow \mathrm{Vect}_K$
vanishing on $\partial \M$
can equivalently be thought of as a bifunctor
$F \colon \R \times \R^{\circ} \rightarrow \mathrm{Vect}_K$
supported on the interior $\op{int} \M \subset \R \times \R$,
which is covariant in its first argument
and contravariant in the second.
For axis-parallel internal maps of a functor
$F \colon \M^{\circ} \rightarrow \mathrm{Vect}_K$
we use similar notation as with bifunctors.
More specifically,
if we have $u, v \in \M$ with $u \preceq v$,
$u = (x, y)$, and $v = (s, t)$,
then we use notation as in the commutative diagram
\begin{equation*}
  \begin{tikzcd}[column sep=11ex, row sep=8ex]
    F(v) = F(s, t)
    \arrow[r, "{F(s \leq x, t)}"]
    \arrow[d, "{F(s, y \leq t)}"']
    \arrow[rd, "{F(u \preceq v)}" description]
    &
    F(x, t)
    \arrow[d, "{F(x, y \leq t)}"]
    \\
    F(s, y)
    \arrow[r, "{F(s \leq x, y)}"']
    &
    F(x, y) = F(u)
    .
  \end{tikzcd}
\end{equation*}
Now we define the persistence diagram as
\begin{equation}
  \label{eq:diagram}
  \op{Dgm} (F) : \op{int} \M \rightarrow \N_0,
  \,
  u \mapsto
  \dim_K F(u) -
  \dim_K \sum_{v \succ u} \op{Im} F(u \preceq v)
\end{equation}
for any sequentially continuous cohomological functor
$F \colon \M^{\circ} \rightarrow \mathrm{vect}_K$.
In the last term, $v$ ranges over all $v \in \M$ with $v \succ u$.
Moreover, we note that
\begin{equation*}
  \sum_{v \succ u} \op{Im} F(u \preceq v)
  =
  \left(\bigcup_{s < x} \op{Im} F(s \leq x, y)\right) +
  \left(\bigcup_{t > y} \op{Im} F(x, y \leq t)\right),
\end{equation*}
for $u = (x, y)$.

\begin{dfn}[Extended Persistence Diagram]
  \label{dfn:diagramFunction}
  We assume $l_0$ and $l_1$ intersect the $x$-axis in $-\pi$
  and $\pi$ respectively
  and that $\mathcal{H}^{\bullet}$ is singular cohomology
  with coefficients in $K$.
  Moreover, let ${f \colon X \rightarrow \R}$ be a $K$-tame continuous function.
  The \emph{extended persistence diagram of $f$ (over $K$)} is
  ${\op{Dgm}(f) := \op{Dgm}(f; K) := \op{Dgm}(h(f)).}$
\end{dfn}

Originally the extended persistence diagram
was defined in a different way by
\cite{MR2472288};
see \cref{sec:connExt} for details
on the connection
between these two definitions.
Up to isomorphism of ambient sets,
the multiset defined by \cite{MR2472288}
and the multiset defined here
are the same.
We note that $\op{Dgm}(f)$ is supported in the downset
${\downarrow \op{Im} \blacktriangle \subseteq \M}$.
\cref{dfn:diagramFunction} is consistent
with \cite[Definition 2.2]{2021arXiv200701834B}
in the sense that both definitions yield the same multiset
for $X$ a finite simplicial complex and $f$ piecewise linear.

Next we show that sequentially continuous
pfd cohomological functors $\M^{\circ} \rightarrow \vectK $
decompose into the following type of indecomposables,
see also \cref{fig:contraBlock}.

\begin{dfn}[Contravariant Block]
  \label{dfn:contraBlock}
  For $v \in \M$ we define
  \begin{equation*}
    B_v \colon \M^{\circ} \rightarrow \mathrm{Vect}_K,
    u \mapsto
    \begin{cases}
      K & u \in (\downarrow v) \cap \op{int}\left(\uparrow T^{-1}(v)\right)
      \\
      \{0\} & \text{otherwise},
    \end{cases}
  \end{equation*}
  where $\op{int}\left(\uparrow T^{-1}(v)\right)$ is the interior of the upset
  of $T^{-1}(v)$ in $\M$.
  The internal maps are identities whenever both domain and codomain are $K$,
  otherwise they are zero.
\end{dfn}

\begin{figure}[t]
  \centering
  \begin{subfigure}[b]{0.49\textwidth}
    \centering
\begingroup%
  \makeatletter%
  \providecommand\color[2][]{%
    \errmessage{(Inkscape) Color is used for the text in Inkscape, but the package 'color.sty' is not loaded}%
    \renewcommand\color[2][]{}%
  }%
  \providecommand\transparent[1]{%
    \errmessage{(Inkscape) Transparency is used (non-zero) for the text in Inkscape, but the package 'transparent.sty' is not loaded}%
    \renewcommand\transparent[1]{}%
  }%
  \providecommand\rotatebox[2]{#2}%
  \newcommand*\fsize{\dimexpr\f@size pt\relax}%
  \newcommand*\lineheight[1]{\fontsize{\fsize}{#1\fsize}\selectfont}%
  \ifx\svgwidth\undefined%
    \setlength{\unitlength}{170.7024765bp}%
    \ifx\svgscale\undefined%
      \relax%
    \else%
      \setlength{\unitlength}{\unitlength * \real{\svgscale}}%
    \fi%
  \else%
    \setlength{\unitlength}{\svgwidth}%
  \fi%
  \global\let\svgwidth\undefined%
  \global\let\svgscale\undefined%
  \makeatother%
  \begin{picture}(1,0.79084969)%
    \lineheight{1}%
    \setlength\tabcolsep{0pt}%
    \put(0,0){\includegraphics[width=\unitlength,page=1]{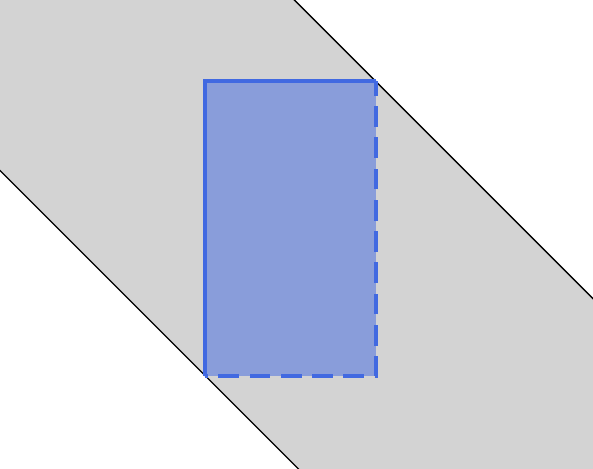}}%
    \put(0.18300657,0.53594785){\makebox(0,0)[t]{\lineheight{1.25}\smash{\begin{tabular}[t]{c}$\{0\}$\end{tabular}}}}%
    \put(0.83660122,0.22222217){\makebox(0,0)[t]{\lineheight{1.25}\smash{\begin{tabular}[t]{c}$\{0\}$\end{tabular}}}}%
    \put(0.49019588,0.38562094){\makebox(0,0)[t]{\lineheight{1.25}\smash{\begin{tabular}[t]{c}$K$\end{tabular}}}}%
    \put(0.32026161,0.67320245){\makebox(0,0)[t]{\lineheight{1.25}\smash{\begin{tabular}[t]{c}$v$\end{tabular}}}}%
    \put(0.67320245,0.10457537){\makebox(0,0)[t]{\lineheight{1.25}\smash{\begin{tabular}[t]{c}$T^{-1}(v)$\end{tabular}}}}%
    \put(0,0){\includegraphics[width=\unitlength,page=2]{contravariant-block-cropped.pdf}}%
    \put(0.83137898,0.50848999){\makebox(0,0)[t]{\lineheight{1.25}\smash{\begin{tabular}[t]{c}$l_1$\end{tabular}}}}%
    \put(0.19553407,0.24237434){\makebox(0,0)[t]{\lineheight{1.25}\smash{\begin{tabular}[t]{c}$l_0$\end{tabular}}}}%
  \end{picture}%
\endgroup%

    \caption{
      The indecomposable $B_v \colon \M^{\circ} \rightarrow \mathrm{Vect}_K$.
    }
    \label{fig:contraBlock}
  \end{subfigure}
  \begin{subfigure}[b]{0.49\textwidth}
    \centering
\begingroup%
  \makeatletter%
  \providecommand\color[2][]{%
    \errmessage{(Inkscape) Color is used for the text in Inkscape, but the package 'color.sty' is not loaded}%
    \renewcommand\color[2][]{}%
  }%
  \providecommand\transparent[1]{%
    \errmessage{(Inkscape) Transparency is used (non-zero) for the text in Inkscape, but the package 'transparent.sty' is not loaded}%
    \renewcommand\transparent[1]{}%
  }%
  \providecommand\rotatebox[2]{#2}%
  \newcommand*\fsize{\dimexpr\f@size pt\relax}%
  \newcommand*\lineheight[1]{\fontsize{\fsize}{#1\fsize}\selectfont}%
  \ifx\svgwidth\undefined%
    \setlength{\unitlength}{171bp}%
    \ifx\svgscale\undefined%
      \relax%
    \else%
      \setlength{\unitlength}{\unitlength * \real{\svgscale}}%
    \fi%
  \else%
    \setlength{\unitlength}{\svgwidth}%
  \fi%
  \global\let\svgwidth\undefined%
  \global\let\svgscale\undefined%
  \makeatother%
  \begin{picture}(1,0.78947368)%
    \lineheight{1}%
    \setlength\tabcolsep{0pt}%
    \put(0,0){\includegraphics[width=\unitlength,page=1]{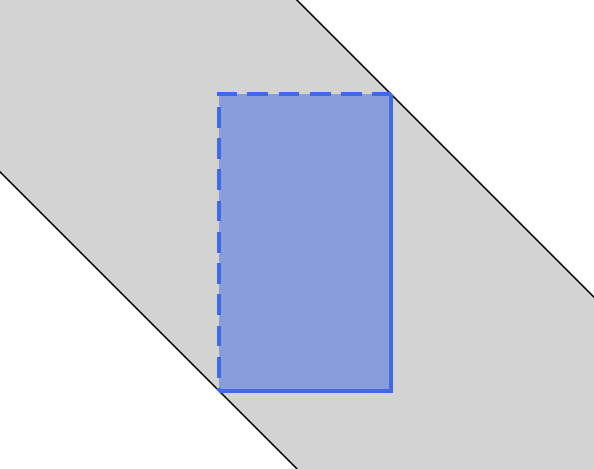}}%
    \put(0.20394737,0.51315789){\makebox(0,0)[t]{\lineheight{1.25}\smash{\begin{tabular}[t]{c}$\{0\}$\end{tabular}}}}%
    \put(0.86184211,0.18421053){\makebox(0,0)[t]{\lineheight{1.25}\smash{\begin{tabular}[t]{c}$\{0\}$\end{tabular}}}}%
    \put(0.51315789,0.36184211){\makebox(0,0)[t]{\lineheight{1.25}\smash{\begin{tabular}[t]{c}$K$\end{tabular}}}}%
    \put(0.32236842,0.66447368){\makebox(0,0)[t]{\lineheight{1.25}\smash{\begin{tabular}[t]{c}$T(v)$\end{tabular}}}}%
    \put(0.68421053,0.09210526){\makebox(0,0)[t]{\lineheight{1.25}\smash{\begin{tabular}[t]{c}$v$\end{tabular}}}}%
    \put(0,0){\includegraphics[width=\unitlength,page=2]{block-cropped.pdf}}%
    \put(0.85658553,0.48551974){\makebox(0,0)[t]{\lineheight{1.25}\smash{\begin{tabular}[t]{c}$l_1$\end{tabular}}}}%
    \put(0.21655746,0.21765307){\makebox(0,0)[t]{\lineheight{1.25}\smash{\begin{tabular}[t]{c}$l_0$\end{tabular}}}}%
  \end{picture}%
\endgroup%

    \caption{
      The indecomposable $B^v \colon \M \rightarrow \mathrm{Vect}_K$.
    }
    \label{fig:block}
  \end{subfigure}
  \caption{The contravariant and the covariant indecomposables.}
\end{figure}

\begin{lem}[Yoneda]
  \label{lem:yonedaContra}
  Let $G \colon \M^{\circ} \rightarrow \mathrm{Vect}_K$
  be a contravariant functor
  vanishing on $\partial \M$ and let ${v \in \op{int} \M}$.
  Then the evaluation at $1 \in K = B_v(v)$
  yields a linear isomorphism
  \begin{equation*}
    \op{Nat}(B_v, G) \cong G(v),
  \end{equation*}
  where $\op{Nat}(B_v, G)$ denotes the vector space
  of natural transformations from $B_v$ to $G$.
\end{lem}

Now let $F \colon \M^{\circ} \rightarrow \mathrm{vect}_K$
be a sequentially continuous cohomological functor,
as defined in \cref{dfn:cont} and \cref{dfn:cohomological}.
For each $v \in \op{int} \M$ we choose a basis
for a complement of $\sum_{w \succ v} \op{Im} F(v \preceq w)$ in $F(v)$.
By the Yoneda \cref{lem:yonedaContra} this yields a natural transformation
\begin{equation*}
  \varphi \colon
  \bigoplus_{v \in \op{int} \M} (B_v)^{\oplus \mu(v)} \longrightarrow F
  ,
\end{equation*}
where $\mu := \op{Dgm}(F)$.

\begin{prp}
  \label{prp:decomp}
  The natural transformation
  \[\varphi \colon
  \bigoplus_{v \in \op{int} \M} (B_v)^{\oplus \mu(v)} \longrightarrow F\]
  is a natural isomorphism.  
\end{prp}

From this proposition, which we prove in the text below,
we obtain the following theorem.

\begin{thm}
  \label{thm:decomp}
  Any sequentially continuous pfd cohomological functor
  $G \colon \M^{\circ} \rightarrow \vectK$
  decomposes as
  \begin{equation*}
    G \cong \bigoplus_{v \in \op{int} \M} (B_v)^{\oplus \nu(v)}
    ,
  \end{equation*}
  where $\nu := \op{Dgm}(G)$.
\end{thm}

\begin{cor}
  Any sequentially continuous cohomological functor
  $G \colon \M^{\circ} \rightarrow \vectK$
  is projective
  in the full subcategory of contravariant functors
  $\M^{\circ} \rightarrow \VectK$
  vanishing on $\partial \M$.
\end{cor}

\begin{proof}
  This follows from \cref{thm:decomp} and the Yoneda \cref{lem:yonedaContra}.
\end{proof}

\begin{figure}[t]
  \centering
\begingroup%
  \makeatletter%
  \providecommand\color[2][]{%
    \errmessage{(Inkscape) Color is used for the text in Inkscape, but the package 'color.sty' is not loaded}%
    \renewcommand\color[2][]{}%
  }%
  \providecommand\transparent[1]{%
    \errmessage{(Inkscape) Transparency is used (non-zero) for the text in Inkscape, but the package 'transparent.sty' is not loaded}%
    \renewcommand\transparent[1]{}%
  }%
  \providecommand\rotatebox[2]{#2}%
  \newcommand*\fsize{\dimexpr\f@size pt\relax}%
  \newcommand*\lineheight[1]{\fontsize{\fsize}{#1\fsize}\selectfont}%
  \ifx\svgwidth\undefined%
    \setlength{\unitlength}{184.42042923bp}%
    \ifx\svgscale\undefined%
      \relax%
    \else%
      \setlength{\unitlength}{\unitlength * \real{\svgscale}}%
    \fi%
  \else%
    \setlength{\unitlength}{\svgwidth}%
  \fi%
  \global\let\svgwidth\undefined%
  \global\let\svgscale\undefined%
  \makeatother%
  \begin{picture}(1,0.67915469)%
    \lineheight{1}%
    \setlength\tabcolsep{0pt}%
    \put(0,0){\includegraphics[width=\unitlength,page=1]{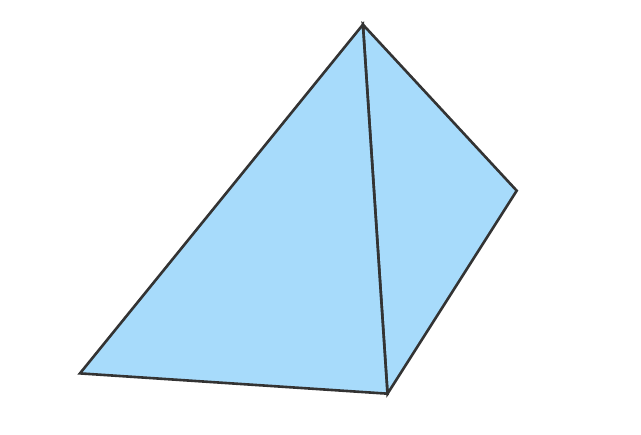}}%
    \put(0.96157527,0.05379434){\makebox(0,0)[t]{\lineheight{1.25}\smash{\begin{tabular}[t]{c}$a$\end{tabular}}}}%
    \put(0,0){\includegraphics[width=\unitlength,page=2]{sole-hood.pdf}}%
    \put(0.96157528,0.37079732){\makebox(0,0)[t]{\lineheight{1.25}\smash{\begin{tabular}[t]{c}$b$\end{tabular}}}}%
    \put(0,0){\includegraphics[width=\unitlength,page=3]{sole-hood.pdf}}%
    \put(0.96157528,0.63016324){\makebox(0,0)[t]{\lineheight{1.25}\smash{\begin{tabular}[t]{c}$c$\end{tabular}}}}%
    \put(0,0){\includegraphics[width=\unitlength,page=4]{sole-hood.pdf}}%
  \end{picture}%
\endgroup%

\begingroup%
  \makeatletter%
  \providecommand\color[2][]{%
    \errmessage{(Inkscape) Color is used for the text in Inkscape, but the package 'color.sty' is not loaded}%
    \renewcommand\color[2][]{}%
  }%
  \providecommand\transparent[1]{%
    \errmessage{(Inkscape) Transparency is used (non-zero) for the text in Inkscape, but the package 'transparent.sty' is not loaded}%
    \renewcommand\transparent[1]{}%
  }%
  \providecommand\rotatebox[2]{#2}%
  \newcommand*\fsize{\dimexpr\f@size pt\relax}%
  \newcommand*\lineheight[1]{\fontsize{\fsize}{#1\fsize}\selectfont}%
  \ifx\svgwidth\undefined%
    \setlength{\unitlength}{308.68904114bp}%
    \ifx\svgscale\undefined%
      \relax%
    \else%
      \setlength{\unitlength}{\unitlength * \real{\svgscale}}%
    \fi%
  \else%
    \setlength{\unitlength}{\svgwidth}%
  \fi%
  \global\let\svgwidth\undefined%
  \global\let\svgscale\undefined%
  \makeatother%
  \begin{picture}(1,0.65599996)%
    \lineheight{1}%
    \setlength\tabcolsep{0pt}%
    \put(0,0){\includegraphics[width=\unitlength,page=1]{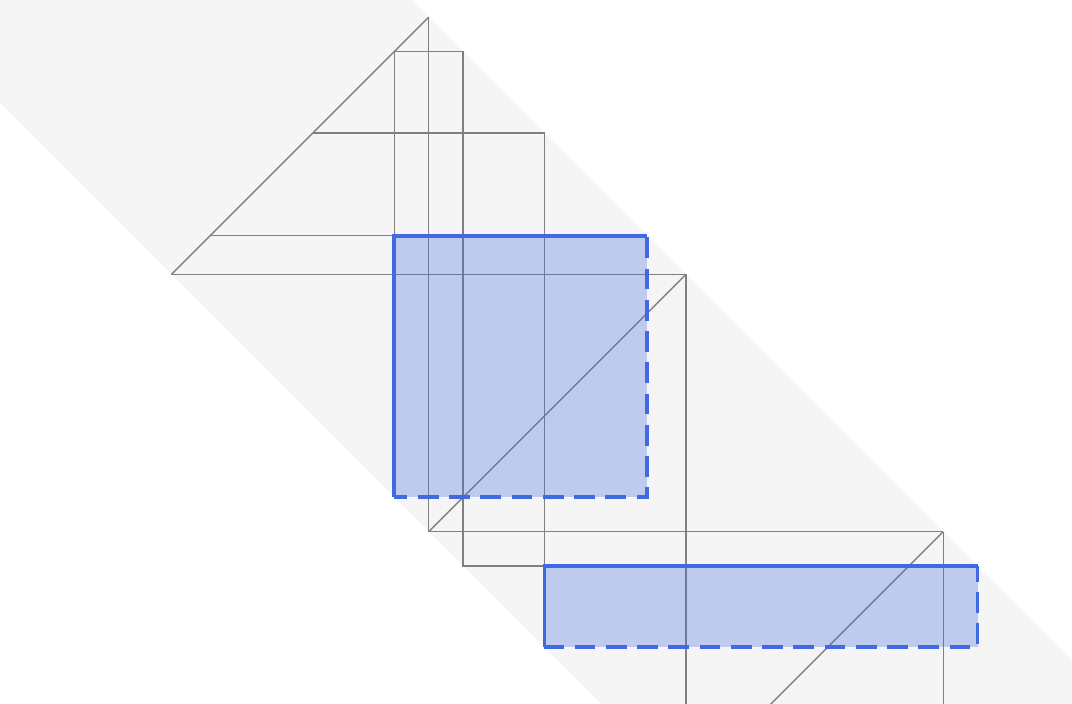}}%
    \put(0.34400006,0.616){\makebox(0,0)[t]{\lineheight{1.25}\smash{\begin{tabular}[t]{c}$\blacktriangle c$\end{tabular}}}}%
    \put(0.27199993,0.53999998){\makebox(0,0)[t]{\lineheight{1.25}\smash{\begin{tabular}[t]{c}$\blacktriangle b$\end{tabular}}}}%
    \put(0.17200003,0.44399997){\makebox(0,0)[t]{\lineheight{1.25}\smash{\begin{tabular}[t]{c}$\blacktriangle a$\end{tabular}}}}%
    \put(0,0){\includegraphics[width=\unitlength,page=2]{sole-hood-diagram.pdf}}%
    \put(0.80066594,0.27133398){\makebox(0,0)[t]{\lineheight{1.25}\smash{\begin{tabular}[t]{c}$l_1$\end{tabular}}}}%
    \put(0.27576441,0.24423559){\makebox(0,0)[t]{\lineheight{1.25}\smash{\begin{tabular}[t]{c}$l_0$\end{tabular}}}}%
  \end{picture}%
\endgroup%

  \caption{
    A geometric simplicial complex in $\R^3$ at the top
    and the two indecompsables of the RISC of its height function
    at the bottom.
  }
  \label{fig:soleHood}
\end{figure}

\begin{figure}[t]
  \centering
\begingroup%
  \makeatletter%
  \providecommand\color[2][]{%
    \errmessage{(Inkscape) Color is used for the text in Inkscape, but the package 'color.sty' is not loaded}%
    \renewcommand\color[2][]{}%
  }%
  \providecommand\transparent[1]{%
    \errmessage{(Inkscape) Transparency is used (non-zero) for the text in Inkscape, but the package 'transparent.sty' is not loaded}%
    \renewcommand\transparent[1]{}%
  }%
  \providecommand\rotatebox[2]{#2}%
  \newcommand*\fsize{\dimexpr\f@size pt\relax}%
  \newcommand*\lineheight[1]{\fontsize{\fsize}{#1\fsize}\selectfont}%
  \ifx\svgwidth\undefined%
    \setlength{\unitlength}{364.27896881bp}%
    \ifx\svgscale\undefined%
      \relax%
    \else%
      \setlength{\unitlength}{\unitlength * \real{\svgscale}}%
    \fi%
  \else%
    \setlength{\unitlength}{\svgwidth}%
  \fi%
  \global\let\svgwidth\undefined%
  \global\let\svgscale\undefined%
  \makeatother%
  \begin{picture}(1,0.46324387)%
    \lineheight{1}%
    \setlength\tabcolsep{0pt}%
    \put(0,0){\includegraphics[width=\unitlength,page=1]{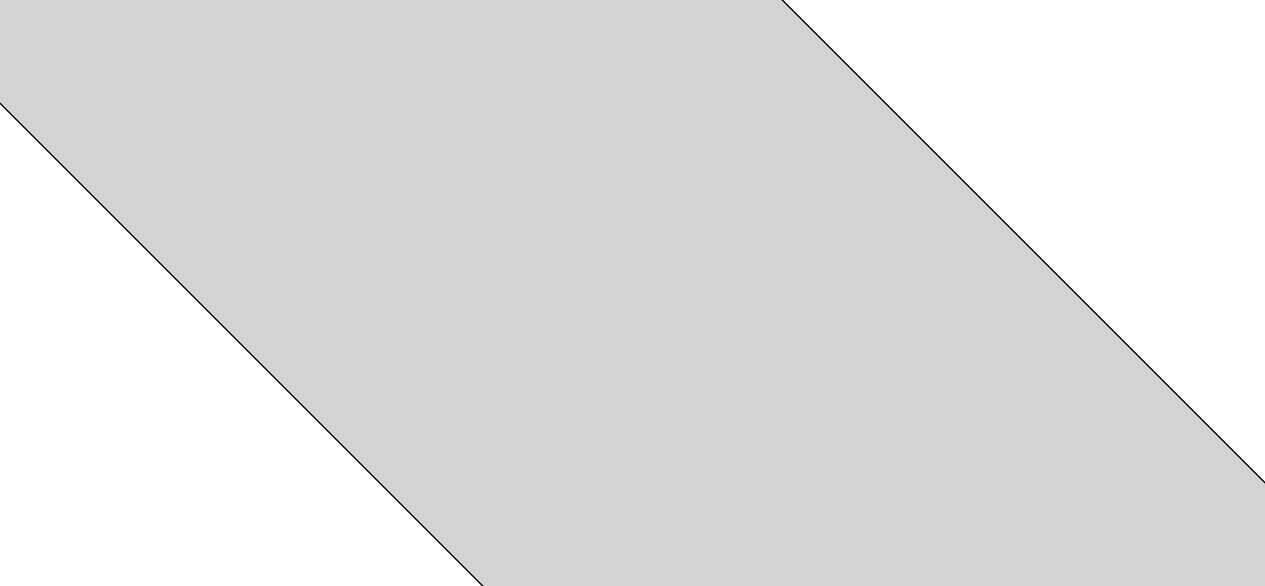}}%
    \put(0.31532702,0.36502155){\makebox(0,0)[t]{\lineheight{1.25}\smash{\begin{tabular}[t]{c}$u_{0,0}$\end{tabular}}}}%
    \put(0.6834031,0.10766014){\makebox(0,0)[t]{\lineheight{1.25}\smash{\begin{tabular}[t]{c}$u$\end{tabular}}}}%
    \put(0.7154919,0.0913241){\makebox(0,0)[lt]{\lineheight{1.25}\smash{\begin{tabular}[t]{l}$y$\end{tabular}}}}%
    \put(0.7154919,0.37960735){\makebox(0,0)[lt]{\lineheight{1.25}\smash{\begin{tabular}[t]{l}$y_0$\end{tabular}}}}%
    \put(0.6979889,0.06798691){\makebox(0,0)[t]{\lineheight{1.25}\smash{\begin{tabular}[t]{c}$x$\end{tabular}}}}%
    \put(0.28615562,0.06798691){\makebox(0,0)[t]{\lineheight{1.25}\smash{\begin{tabular}[t]{c}$x_0$\end{tabular}}}}%
    \put(0,0){\includegraphics[width=\unitlength,page=2]{images-filtration.pdf}}%
    \put(0.86312792,0.25122456){\makebox(0,0)[t]{\lineheight{1.25}\smash{\begin{tabular}[t]{c}$l_1$\end{tabular}}}}%
    \put(0.13802423,0.20269925){\makebox(0,0)[t]{\lineheight{1.25}\smash{\begin{tabular}[t]{c}$l_0$\end{tabular}}}}%
  \end{picture}%
\endgroup%

  \caption{The horizontal and vertical filtrations of $F(u)$
    by images.}
  \label{fig:imagesFiltration}
\end{figure}

In \cref{fig:soleHood} we show a geometric simplicial complex in $\R^3$
and the two indecomposables of the RISC of its height function.
Now in order to show that $\varphi$ is a natural isomorphism,
we have to show pointwise that
\[\varphi_u \colon
  \bigoplus_{v \in \op{int} \M} (B_v (u))^{\oplus \mu(v)} \longrightarrow F(u)\]
is an isomorphism
for all $u := (x, y) \in \op{int} \M$.
To this end, we fix some notation,
which we use in the proof of \cref{prp:decomp}
and auxiliary lemmas. 
As depicted in \cref{fig:imagesFiltration},
let $x_0$ be the $x$-coordinate of the intersection
of $l_0$ and the horizontal line through $u$,
let ${x_1 < x_2 < \dots < x_{k-1}}$
be the points of discontinuity of the function
\begin{equation}
  \label{eq:rank}
  (x_0, x) \rightarrow \N_0, \,
  s \mapsto \op{rank} F(s \leq x, y)
  ,
\end{equation}
and let $x_k := x$.
Similarly,
let $y_0$ be the intersection of $l_1$
and the vertical line through $u$,
let ${y_1 > y_2 > \dots > y_{l-1}}$
be the points of discontinuity of the function
\begin{equation*}
  (y, y_0) \rightarrow \N_0, \,
  t \mapsto \op{rank} F(x, y \leq t)
  ,
\end{equation*}
and let $y_l := y$.
Moreover, we set
$u_{(i,j)} := (x_i, y_j)$
for $i = 0, \dots, k$ and $j = 0, \dots, l$,
then we have
$u = u_{(k,l)}$ and $T (u) = u_{(0,0)}$.
With some abuse of notation, we may also drop the parentheses
and write $u_{i,j}$ in place of $u_{(i, j)}$.
Furthermore, let $\preceq$ be the
\href{
  https://en.wikipedia.org/wiki/Lexicographical_order#Colexicographic_order
}{colexicographic order}
on ${I := \{0, \dots, k\} \times \{0, \dots, l\}}$,
which is defined by
\begin{equation*}
  (i, j) \preceq (i', j')
  \quad
  : \Leftrightarrow
  \quad
  j < j' ~\vee~ (j = j' ~\wedge~ i < i')
  .
\end{equation*}
For any contravariant functor 
$G \colon \M^{\circ} \rightarrow \mathrm{Vect}_K$
vanishing on $\partial \M$
and $\zeta \in I$
we set
\begin{equation*}
  G_{\zeta} := \sum_{\xi \preceq \zeta} \op{Im} G(u \preceq u_{\xi})
\end{equation*}
to obtain the natural filtration
\begin{equation*}
  \bigcup_{\zeta \in I} G_{\zeta} = G(u)
  .
\end{equation*}
For a pair $\zeta := (i, j) \in I$ we will drop the parentheses
in the index and write
$G_{i,j} = G_{\zeta}$ in place of $G_{(i, j)}$.
With this notation we may write the filtration
$\bigcup_{\zeta \in I} G_{\zeta}$ as
\begin{align*}
  \{0\} & = G_{0, 0} = G_{1, 0} = G_{2, 0}
          = \dots = G_{k, 0}
  \\
        & = G_{0, 1} \subseteq G_{1, 1} \subseteq G_{2, 1}
          \subseteq \dots \subseteq G_{k, 1}
  \\
        & = G_{0, 2} \subseteq G_{1, 2} \subseteq G_{2, 2}
          \subseteq \dots \subseteq G_{k, 2}
  \\
        & \qquad \vdots
  \\
        & = G_{0, l} \subseteq G_{1, l} \subseteq G_{2, l}
          \subseteq \dots \subseteq G_{k, l} = G(u)
          ,
\end{align*}
see also \cref{fig:filtration}.
We may describe this filtration more concretely using the equations
\begin{align}
  \label{eq:filtr0}
  G_{i,0} &= 0 & \text{for}~ & i = 0, \dots, k, \\
  \label{eq:filtrWrap}
  G_{k,j-1} &= G_{0,j} & \text{for}~ & j = 1, \dots, l, ~\text{and} \\
  \label{eq:filtr}
  G_{i,j} &= \op{Im} G(u \preceq u_{i,j}) +
            \op{Im} G(u \preceq u_{k,j-1}) & \text{for}~ & j = 1, \dots, l
                                                           ~\text{and}~
                                                           i = 0, \dots, k.
\end{align}

\begin{figure}[t]
  \centering
  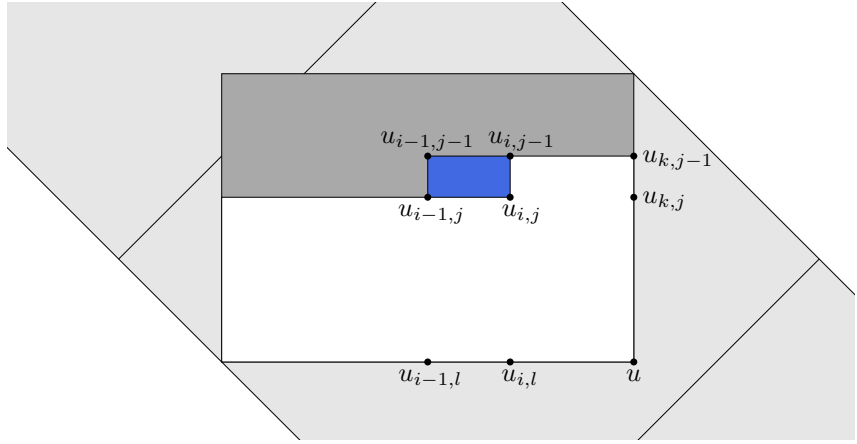
  \caption{
    The filtration of $G(u)$
    in terms of the colexicographic order on $I$.
    The large axis-aligned rectangle contains all points
    such that the corresponding image in $G(u)$
    can be non-zero.
    The subspace $G_{i-1,j} \subseteq G(u)$
    is the sum of the images in $G(u)$
    corresponding to points in the region shaded in dark grey.
    If we add the images corresponding to points in the blue rectangle
    (or just the image corresponding to the lower right vertex $u_{i,j}$),
    then we obtain $G_{i,j}$ as the next step in the filtration.
  }
  \label{fig:filtration}
\end{figure}

\begin{proof}[Proof of \cref{prp:decomp}]
  Now let $H := \bigoplus_{v \in \op{int} \M} (B_v)^{\oplus \mu(v)}$.
  We show that
  \begin{equation*}
    \varphi_{\xi} \colon H_{\xi} \rightarrow F_{\xi}
  \end{equation*}
  is an isomorphism for all $\xi \in I$ by induction on $\xi$.
  By \eqref{eq:filtr0} the map $\varphi_{i, 0}$ is an isomorphism
  for all ${i = 0, \dots, k}$.
  Moreover, $\varphi_{0, j}$ is an isomorphism
  if $\varphi_{k,j-1}$ is an isomorphism
  for all ${j = 1, \dots, l}$ by \eqref{eq:filtrWrap}.
  Thus, in order to complete our proof by induction,
  it suffices to show that
  $\varphi_{i,j} \colon H_{i,j} \rightarrow F_{i,j}$
  is an isomorphism whenever
  $\varphi_{i-1,j} \colon H_{i-1,j} \rightarrow F_{i-1,j}$
  is an isomorphism
  for all $i = 1, \dots, k$ and $j = 1, \dots, l$.
  To this end, we consider the commutative diagram
  \begin{equation*}
    \begin{tikzcd}
      0
      \arrow[r]
      &
      H_{i-1,j}
      \arrow[r]
      \arrow[d, "\varphi_{i-1,j}"']
      &
      H_{i,j}
      \arrow[r]
      \arrow[d, "\varphi_{i,j}"]
      &
      H_{i,j} / H_{i-1,j}
      \arrow[r]
      \arrow[d]
      &
      0
      \\
      0
      \arrow[r]
      &
      F_{i-1,j}
      \arrow[r]
      &
      F_{i,j}
      \arrow[r]
      &
      F_{i,j} / F_{i-1,j}
      \arrow[r]
      &
      0    
    \end{tikzcd}
  \end{equation*}
  with exact rows.
  By the five lemma, it suffices to show that the vertical map
  on the right hand side is an isomorphism.
  To this end, we note that $H$ is cohomological,
  as it is a direct sum of cohomological functors.
  Thus, by \cref{lem:isoCoho} below,
  there is a commutative square
  \begin{equation*}
    \begin{tikzcd}[column sep=4em, row sep=5.5ex]
      \dfrac{
        H(u_{i,j})
      }{
        \op{Im} H(u_{i,j} \preceq u_{i-1,j})
        +
        \op{Im} H(u_{i,j} \preceq u_{i,j-1})
      }
      \arrow[r, "\sim"]
      \arrow[d]
      &
      \dfrac{
        H_{i,j}
      }{
        H_{i-1,j}
      }
      \arrow[d]
      \\
      \dfrac{
        F(u_{i,j})
      }{
        \op{Im} F(u_{i,j} \preceq u_{i-1,j})
        +
        \op{Im} F(u_{i,j} \preceq u_{i,j-1})
      }
      \arrow[r, "\sim"]
      &
      \dfrac{
        F_{i,j}
      }{
        F_{i-1,j}
      }
      ,
    \end{tikzcd}
  \end{equation*}
  where the two vertical maps are induced by $\varphi$.
  As the two horizontal maps are isomorphisms by \cref{lem:isoCoho},
  it remains to show that the vertical map on the left hand side
  is an isomorphism.
  To this end, we consider the commutative square
  \begin{equation*}
    \begin{tikzcd}
      \dfrac{
        H(u_{i,j})
      }{
        \op{Im} H(u_{i,j} \preceq u_{i-1,j})
        +
        \op{Im} H(u_{i,j} \preceq u_{i,j-1})
      }
      \arrow[r]
      \arrow[d]
      &
      \dfrac{H(u_{i,j})} 
      {\sum\limits_{w \succ u_{i,j}} \op{Im} H(u_{i,j} \preceq w)}
      \arrow[d]
      \\
      \dfrac{
        F(u_{i,j})
      }{
        \op{Im} F(u_{i,j} \preceq u_{i-1,j})
        +
        \op{Im} F(u_{i,j} \preceq u_{i,j-1})
      }
      \arrow[r]
      &
      \dfrac{F(u_{i,j})} 
      {\sum\limits_{w \succ u_{i,j}} \op{Im} F(u_{i,j} \preceq w)}
      ,
    \end{tikzcd}
  \end{equation*}
  where the vertical maps are induced by $\varphi$
  and the horizontal maps are induced by the internal maps
  of $H$ and $F$ respectively.
  We have to show that the vertical map on the left hand side
  is an isomorphism.
  The horizontal map at the top
  and the horizontal map at the bottom
  are isomorphisms by \cref{cor:isoH,lem:isoF},
  respectively.
  Thus,
  it suffices to show that the vertical map on the right hand side
  is an isomorphism.
  To this end,
  we consider the commutative diagram
  \begin{equation*}
    \begin{tikzcd}
      &
      K^{\mu(u_{i,j})}
      \arrow[d, "~\sim" {anchor=north, rotate=90}]
      \\
      &
      B_{u_{i,j}}^{\oplus \mu(u_{i,j})} (u_{i,j})
      \arrow[d, "\iota"]
      \\
      \displaystyle\sum\limits_{w \succ u_{i,j}} \op{Im} H(u_{i,j} \preceq w)
      \arrow[r, hook]
      \arrow[d]
      &
      H(u_{i,j})
      \arrow[r]
      \arrow[d, "\varphi_{u_{i,j}}"]
      &
      \dfrac{H(u_{i,j})} 
      {\sum\limits_{w \succ u_{i,j}} \op{Im} H(u_{i,j} \preceq w)}
      \arrow[d]
      \\
      \displaystyle\sum\limits_{w \succ u_{i,j}} \op{Im} F(u_{i,j} \preceq w)
      \arrow[r, hook]
      &
      F(u_{i,j})
      \arrow[r]
      &
      \dfrac{F(u_{i,j})} 
      {\sum\limits_{w \succ u_{i,j}} \op{Im} F(u_{i,j} \preceq w)}
      .
    \end{tikzcd}
  \end{equation*}
  Now the image of $\iota$ is a complement of
  $\sum\limits_{w \succ u_{i,j}} \op{Im} H(u_{i,j} \preceq w)$
  in $H(u_{i,j})$.
  Moreover, by the construction of $\varphi$,
  the composition of the three vertical maps in the center
  map the standard basis of $K^{\mu(u_{i,j})}$
  to a basis for a complement $C$ of
  $\sum\limits_{w \succ u_{i,j}} \op{Im} F(u_{i,j} \preceq w)$
  in $F(u_{i,j})$.
  As a result, the map $\varphi_{u_{i,j}}$
  maps $\op{Im}(\iota)$ isomorphically onto $C$,
  hence the vertical map on the right hand side is an isomorphism.
  As $u \in \op{int} \M$ was arbitrary,
  $\varphi$ is a natural isomorphism.
\end{proof}

\begin{lem}
  \label{lem:isoCoho}
  For any cohomological functor
  ${G \colon \M^{\circ} \rightarrow \mathrm{Vect}_K}$
  and any pair of indices ${(i, j) \in \{1, \dots, k\} \times \{1, \dots, l\}}$
  there is an isomorphism
  \begin{equation*}
    \frac{
      G(u_{i,j})
    }{
      \op{Im} G(u_{i,j} \preceq u_{i-1,j})
      +
      \op{Im} G(u_{i,j} \preceq u_{i,j-1})
    }
    \xlongrightarrow{\sim}
    \frac{
      G_{i,j}
    }{
      G_{i-1,j}
    }
  \end{equation*}
  natural in $G$.
\end{lem}

\begin{proof}
  We consider the commutative diagram
  \begin{equation*}
    \begin{tikzcd}
      G(u_{i-1, j-1})
      \arrow[r]
      \arrow[d]
      &
      G(u_{i, j-1})
      \arrow[r]
      \arrow[d]
      &
      G(u_{k, j-1}) = G(x, y_{j-1})
      \arrow[d]
      \\
      G(u_{i-1, j})
      \arrow[r]
      \arrow[d]
      &
      G(u_{i, j})
      \arrow[r]
      \arrow[d]
      &
      G(u_{k, j}) = G(x, y_j)
      \arrow[d]
      \\
      G(u_{i-1,l})
      \arrow[r]
      &
      G(u_{i,l})
      \arrow[r]
      &
      G(u) = G(x, y)
      .
    \end{tikzcd}
  \end{equation*}
  By \cref{prp:cohomological}.4
  all squares in this diagram are middle exact,
  see also \cref{fig:filtration}.
  Thus, by \cref{prp:4Squares} the map
  $G(u \preceq u_{i,j})$ induces an isomorphism
  \begin{equation*}
    \frac{
      G(u_{i,j})
    }{
      \op{Im} G(u_{i,j} \preceq u_{i-1,j})
      +
      \op{Im} G(u_{i,j} \preceq u_{i,j-1})
    }
    \xrightarrow{\sim}
    \frac{
      \op{Im} G(u \preceq u_{i,j})
      +
      \op{Im} G(u \preceq u_{k,j-1})
    }{
      \op{Im} G(u \preceq u_{i-1,j})
      +
      \op{Im} G(u \preceq u_{k,j-1})
    }
    .
  \end{equation*}
  Moreover, by \eqref{eq:filtr}
  the codomain of this isomorphism is
  $G_{i,j} / G_{i-1,j}$ and thus we may write this isomorphism also as
  \begin{equation*}
    \frac{
      G(u_{i,j})
    }{
      \op{Im} G(u_{i,j} \preceq u_{i-1,j})
      +
      \op{Im} G(u_{i,j} \preceq u_{i,j-1})
    }
    \xlongrightarrow{\sim}
    \frac{
      G_{i,j}
    }{
      G_{i-1,j}
    }
    .
    \qedhere
  \end{equation*}
\end{proof}

Before we prove \cref{cor:isoH,lem:isoF}
we need to establish three auxiliary results.
To this end, we note that the inclusion
\begin{equation*}
  \{x_i \mid i = 0, \dots, k\}
  \hookrightarrow
  [x_0, x]
\end{equation*}
has the upper adjoint
\begin{equation*}
  r_1 \colon
  [x_0, x] \rightarrow \{x_i \mid i = 0, \dots, k\}, \,
  s \mapsto \max \{x_i \mid x_i \leq s\}
  .
\end{equation*}
Similarly
\begin{equation*}
  r_2 \colon
  [y, y_0] \rightarrow \{y_j \mid j = 0, \dots, l\}, \,
  t \mapsto \min \{y_j \mid t \leq y_j\}
\end{equation*}
is the lower adjoint of the inclusion
\begin{equation*}
  \{y_j \mid j = 0, \dots, l\}
  \hookrightarrow
  [y, y_0]
  .
\end{equation*}

\begin{lem}
  \label{lem:discImages}
  We have
  \begin{align*}
    \op{Im} F(r_1(s) \leq x, y)
    & = \op{Im} F(s \leq x, y)
    &
    & \text{for all $s \in [x_0, x]$ as well as}
    \\
    \label{eq:yReflect}
    \op{Im} F(x, y \leq r_2(t))
    & = \op{Im} F(x, y \leq t)
    &
    & \text{for all $t \in [y, y_0]$.}
  \end{align*}
\end{lem}

\begin{proof}
  We prove the first equation,
  the second can be shown in an analogous manner.
  To this end, we consider the filtration
  \begin{equation*}
    \bigcup_{s \leq x} \op{Im} F(s \leq x, y) = F(x, y) = F(u)
    .
  \end{equation*}
  For $s_0 \in [x_0, x)$ the canonical map
  \begin{equation}
    \label{eq:xLim}
    F(s_0, y) \longrightarrow \varprojlim_{s > s_0} F(s, y)
  \end{equation}
  is an isomorphism by the sequential continuity of $F$.
  As a result, the image of the canonical map
  \begin{equation*}
    \varprojlim_{s > s_0} F(s, y) \longrightarrow F(x, y) = F(u)
  \end{equation*}
  and the image $\op{Im} F(s_0 \leq x, y)$ are the same.
  Moreover, the image of \eqref{eq:xLim} and the intersection
  \begin{equation*}
    \bigcap_{s > s_0} \op{Im} F(s \leq x, y)
  \end{equation*}
  are identical, hence
  \begin{equation*}
    \label{eq:xCap}
    \op{Im} F(s_0 \leq x, y) =
    \bigcap_{s > s_0} \op{Im} F(s \leq x, y)
    .
  \end{equation*}
  As a result of this equation,
  the function \eqref{eq:rank} is
  \href{
    https://en.wikipedia.org/wiki/Semi-continuity#Upper_semicontinuity_at_a_point
  }{upper semi-continuous},
  i.e. the superlevel sets of \eqref{eq:rank} are closed.
  Moreover, as $x_1, \dots, x_{k-1}$ are by definition
  the points of discontinuity of \eqref{eq:rank}, we have
  \begin{equation*}
    \op{Im} F(x_i \leq x, y) = \op{Im} F(s \leq x, y)
  \end{equation*}
  for all ${i = 0, \dots, k-1}$ and ${s \in [x_i, x_{i+1})}$.
  Using the upper adjoint
  ${r_1 \colon [x_0, x] \rightarrow \{x_i \mid i = 0, \dots, k\}}$
  we can state this last equation
  without explicit quantification over $\{0, \dots, k-1\}$ as
  \begin{equation*}
    \op{Im} F(r_1(s) \leq x, y) = \op{Im} F(s \leq x, y)
    \quad
    \text{for all $s \in [x_0, x]$.}
    \qedhere
  \end{equation*}
\end{proof}

Now suppose we have
\begin{align*}
  v := (v_1, v_2) \in [u, T(u)] & = [x_0, x] \times [y, y_0]
  \\
  \text{and} \qquad \qquad
  (s, t) \in [v, T(u)] & = [x_0, v_1] \times [v_2, y_0]
                         .
\end{align*}
We consider the commutative square
\begin{equation*}
  \begin{tikzcd}[row sep=7ex, column sep=11ex]
    F(s, t)
    \arrow[r]
    \arrow[d]
    &
    F(v_1, t)
    \arrow[d, "{F(v_1, v_2 \leq t)}"]
    \\
    F(s, v_2)
    \arrow[r, "{F(s \leq v_1, v_2)}"']
    &
    F(v)
  \end{tikzcd}
\end{equation*}
and we define
\begin{equation*}
  F_v (s, t) :=
  \op{Im} F(s \leq v_1, v_2) + \op{Im} F(v_1, v_2 \leq t)
  .
\end{equation*}
Moreover, let
\begin{equation*}
  r := r_1 \times r_2 \colon
  [u, T(u)] \rightarrow
  \{u_{\xi} \mid \xi \in I\}
  ,
\end{equation*}
i.e. $r$ is the lower adjoint to the inclusion
$\{u_{\xi} \mid \xi \in I\} \hookrightarrow [u, T(u)]$.
Then we have
\begin{equation*}
  (F_v \circ r)(s, t) \subseteq F_v (s, t)
  .
\end{equation*}
Furthermore,
in the special case that $v = u$ we have equality:
\begin{equation*}
  (F_u \circ r)(s, t) = F_u (s, t)
  ;
\end{equation*}
as a result of \cref{lem:discImages}.
By the following lemma this is true even when $v \neq u$.

\begin{lem}
  \label{lem:aux}
  We have $(F_v \circ r)(s, t) = F_v (s, t)$.
\end{lem}

\begin{remark}
  In general it may very well happen that
  \begin{align*}
    \op{Im} F(r_1(s) \leq v_1, v_2)
    & \neq \op{Im} F(s \leq v_1, v_2)
    \\
    \text{or} \quad
    \op{Im} F(v_1, v_2 \leq r_2(t))
    & \neq \op{Im} F(v_1, v_2 \leq t)
      .
  \end{align*}
  Thus, it is crucial
  to consider the two summands of $F_v (s, t)$ in conjunction.
\end{remark}

\begin{proof}
  It suffices to show that
  \begin{align}
    \label{eq:auxX}
    F_v (r_1(s), t) & = F_v (s, t)
    \\
    \text{and} \quad
    \label{eq:auxY}
    F_v (s, r_2(t)) & = F_v (s, t)
    ,
  \end{align}
  independent of $s$ and $t$,
  since this implies that
  \begin{equation*}
    (F_v \circ r)(s, t)
    =
    F_v (r_1(s), r_2(t))
    =
    F_v (s, r_2(t))
    =
    F_v (s, t)
    .
  \end{equation*}
  We show \eqref{eq:auxY},
  our proof of \eqref{eq:auxX} is similar.
  Now
  \begin{equation*}
    F_v (s, r_2(t)) =
    \op{Im} F(s \leq v_1, v_2) + \op{Im} F(v_1, v_2 \leq r_2 (t))
    ,
  \end{equation*}
  which is a subspace of
  \begin{equation*}
    \op{Im} F(s \leq v_1, v_2) + \op{Im} F(v_1, v_2 \leq t) =
    F_v (s, t) \subseteq F(v)
    .
  \end{equation*}
  Moreover, this inclusion
  \begin{equation*}
    F_v (s, r_2(t)) \hookrightarrow
    F_v (s, t)
  \end{equation*}
  induces a canonical map
  \begin{equation*}
    \pi_v \colon
    \frac{F(v)}{F_v (s, r_2(t))}
    \longrightarrow
    \frac{F(v)}{F_v (s, t)}
    .
  \end{equation*}
  Now in order to prove \eqref{eq:auxY},
  it suffices to show that $\pi_v$ is an isomorphism.
  In the special case that $v = u$,
  we already have
  $F_u (s, r_2(t)) = F_u (s, t)$
  by the second equation from \cref{lem:discImages},
  hence
  \begin{equation*}
    \pi_u = \op{id} \colon
    \frac{F(u)}{F_u (s, r_2(t))}
    \xlongrightarrow{=}
    \frac{F(u)}{F_u (s, t)}
    .
  \end{equation*}
  Our approach is to reduce the general case for $\pi_v$
  to this special case of $\pi_u = \op{id}$.
  To this end,
  we consider the commutative diagram
  \begin{equation}
    \label{eq:stSixSq}
    \begin{tikzcd}[row sep=5ex]
      F(s, r_2(t))
      \arrow[r]
      \arrow[d]
      &
      F(v_1, r_2(t))
      \arrow[r]
      \arrow[d]
      &
      F(x, r_2(t))
      \arrow[d]
      \\
      F(s, t)
      \arrow[r]
      \arrow[d]
      &
      F(v_1, t)
      \arrow[r]
      \arrow[d]
      &
      F(x, t)
      \arrow[d]
      \\
      F(s, v_2)
      \arrow[r]
      \arrow[d]
      &
      F(v)
      \arrow[r]
      \arrow[d]
      \arrow[rd, "{F(u \preceq v)}" description]
      &
      F(x, v_2)
      \arrow[d]
      \\
      F(s, y)
      \arrow[r]
      &
      F(v_1, y)
      \arrow[r]
      &
      F(u)
      .
    \end{tikzcd}
  \end{equation}
  By \cref{prp:cohomological}.4
  all axis-aligned squares and rectangles in this diagram are middle exact.
  Now $F(u \preceq v)$ maps
  $\op{Im} F(v_1, v_2 \leq r_2 (t))$
  to a subspace of
  $\op{Im} F(x, y \leq r_2 (t))$,
  hence
  $F_v (s, r_2(t))$
  is mapped to a subspace of
  $F_u (s, r_2(t))$.
  Similarly $F(u \preceq v)$ maps
  $F_v (s, t)$
  to a subspace of
  $F_u (s, t)$.
  As a result we obtain the commutative diagram
  \begin{equation*}
    \begin{tikzcd}[column sep=9ex, row sep=6ex]
      F_v (s, r_2(t))
      \arrow[r, "F(u \preceq v)"]
      \arrow[d, hook]
      &
      F_u (s, r_2(t))
      \arrow[d, equal]
      \\
      F_v (s, t)
      \arrow[r, "F(u \preceq v)"']
      \arrow[d, hook]
      &
      F_u (s, t)
      \arrow[dd, hook]
      \\
      F(v)
      \arrow[rd, "F(u \preceq v)"']
      \\
      &
      F(u)
    \end{tikzcd}
  \end{equation*}
  from which we obtain the induced commutative square
  \begin{equation*}
    \begin{tikzcd}[column sep=6ex, row sep=6ex]
      \dfrac{F(v)}{F_v (s, r_2(t))}
      \arrow[r]
      \arrow[d, "\pi_v"']
      &
      \dfrac{F(u)}{F_u (s, r_2(t))}
      \arrow[d, "\pi_u = \op{id}", equal]
      \\
      \dfrac{F(v)}{F_v (s, t)}
      \arrow[r]
      &
      \dfrac{F(u)}{F_u (s, t)}
      .
    \end{tikzcd}
  \end{equation*}
  As all axis-aligned squares and rectangles of \eqref{eq:stSixSq}
  are middle exact,
  the two horizontal maps of this square are isomorphisms
  by \cref{prp:4Squares},
  hence $\pi_v$ is an isomorphism as well.
\end{proof}

\begin{lem}
  \label{lem:diagramSupportGrid}
  The restriction of
  $\op{Dgm}(F)$
  to
  ${(\uparrow u) \cap \op{int} (\downarrow T(u))}$
  is supported on the grid
  $\{u_{\xi} \mid \xi \in I\}$.
\end{lem}

\begin{proof}
  Let
  ${v := (v_1, v_2) \in (\uparrow u) \cap
    \op{int} (\downarrow T(u)) \setminus \{u_{\xi} \mid \xi \in I\}}$.
  We have to show that $\op{Dgm}(F)(v) = 0$.
  As $v \notin \{u_{\xi} \mid \xi \in I\}$
  we have $v \neq r(v)$,
  which implies $v \prec r(v)$.
  Thus, we have $v_1 > r_1(v_1)$ or $v_2 < r_2(v_2)$.
  Without loss of generality we assume that $v_2 < r_2(v_2)$.
  Now let $j = 0, \dots, l-1$ be such that $y_j = r_2(v_2) > v_2$.
  Considering the commutative diagram
  \begin{equation*}
    \begin{tikzcd}
      F(x_0, y_j)
      \arrow[r]
      \arrow[d]
      &
      F(v_1, y_j)
      \arrow[d]
      \\
      F(x_0, v_2)
      \arrow[r]
      &
      F(v)
    \end{tikzcd}
  \end{equation*}
  we see that
  \begin{align*}
    (F_v \circ r) (x_0, v_2)
    & =
      F_v (x_0, y_j)
    \\
    & =
      \op{Im} F(x_0 \leq v_1, v_2) + \op{Im} F(v_1, v_2 \leq y_j)
    \\
    & \subseteq
      \sum_{w \succ v} \op{Im} F(v \preceq w)
      .
  \end{align*}
  Moreover, \cref{lem:aux} implies that
  \begin{align*}
    (F_v \circ r) (x_0, v_2)
    & =
      F_v (x_0, v_2)
    \\
    & =
      \op{Im} F(x_0 \leq v_1, v_2) + \op{Im} F(v_1, v_2 \leq v_2)
    \\
    & =
      \op{Im} F(x_0 \leq v_1, v_2) + F(v)
    \\
    & =
      F(v)
      .
  \end{align*}
  The previous two chains of equations (and an inclusion)
  taken together we obtain
  \begin{equation*}
    F(v) \subseteq
    \sum_{w \succ v} \op{Im} F(v \preceq w) \subseteq
    F(v)
    ,
  \end{equation*}
  hence
  \begin{align*}
    \op{Dgm}(F)(v)
    & =
      \dim_K F(v) - \dim_K \sum_{w \succ v} \op{Im} F(v \preceq w)
    \\
    & =
      \dim_K F(v) - \dim_K F(v)
      =
      0
      .
      \qedhere
  \end{align*}
\end{proof}

\begin{cor}
  \label{cor:isoH}
  For $H = \bigoplus_{v \in \op{int} \M} (B^v)^{\oplus \mu(v)}$,
  the canonical map
  \begin{equation}
    \label{eq:isoH}
    \frac{
      H(u_{i,j})
    }{
      \op{Im} H(u_{i,j} \preceq u_{i-1,j})
      +
      \op{Im} H(u_{i,j} \preceq u_{i,j-1})
    }
    \longrightarrow
    \frac{H(u_{i,j})} 
    {\sum\limits_{w \succ u_{i,j}} \op{Im} H(u_{i,j} \preceq w)}
  \end{equation}
  is an isomorphism.
\end{cor}

\begin{proof}
  We consider the restriction
  of $\op{Dgm}(F)$ to the blue rectangle
  in \cref{fig:filtration}.
  By \cref{lem:diagramSupportGrid},
  this restriction can be non-zero
  only at the vertices
  $u_{i-1,j-1}$, $u_{i,j-1}$, $u_{i-1,j}$, or $u_{i,j}$.
  Thus,
  any indecomposable summand of $H$,
  which is not born at $u_{i,j}$ and yet alive at $u_{i,j}$,
  is alive at $u_{i-1,j}$ or $u_{i,j-1}$,
  hence
  \begin{equation*}
    \op{Im} H(u_{i,j} \preceq u_{i-1,j})
    +
    \op{Im} H(u_{i,j} \preceq u_{i,j-1})
    =
    \sum_{w \succ u_{i,j}} \op{Im} H(u_{i,j} \preceq w)
    .
  \end{equation*}
  As a result, the canonical map \eqref{eq:isoH} is an identity.
\end{proof}

\begin{lem}
  \label{lem:isoF}
  The canonical map
  \begin{equation}
    \label{eq:isoF}
    \frac{
      F(u_{i,j})
    }{
      \op{Im} F(u_{i,j} \preceq u_{i-1,j})
      +
      \op{Im} F(u_{i,j} \preceq u_{i,j-1})
    }
    \longrightarrow
    \frac{F(u_{i,j})} 
    {\sum\limits_{w \succ u_{i,j}} \op{Im} F(u_{i,j} \preceq w)}
  \end{equation}
  is an isomorphism.
\end{lem}

\begin{proof}
  Let ${v := u_{i,j}}$ and let ${R := [v, u_{i-1,j-1}] \setminus \{v\}}$,
  i.e. $R$ is the blue rectangle in \cref{fig:filtration}
  except for the vertex $v = u_{i,j}$.
  Then we have the inclusion
  \[{\op{Im} F(v \preceq w) \subseteq F_v (w)}\]
  for any
  ${w \in R}$,
  and thus
  \begin{equation*}
    \sum_{w \succ u_{i,j}} \op{Im} F(u_{i,j} \preceq w)
    =
    \sum_{w \in R} \op{Im} F(v \preceq w)
    =
    \sum_{w \in R} F_v (w)
    .
  \end{equation*}
  Moreover, by \cref{lem:aux}
  \begin{align*}
    F_v (w)
    & =
      (F_v \circ r)(w) = F_v (u_{i-1, j-1})
    \\
    & =
      \op{Im} F(v \preceq u_{i-1,j})
      +
      \op{Im} F(v \preceq u_{i,j-1})
    \\
    & =
      \op{Im} F(u_{i,j} \preceq u_{i-1,j})
      +
      \op{Im} F(u_{i,j} \preceq u_{i,j-1})
  \end{align*}
  for any $w \in R$
  and as a result the denominators of the domain and the codomain
  of \eqref{eq:isoF} are identical,
  hence \eqref{eq:isoF} is the identity.
\end{proof}

\subsection{Decomposition of $q$-Tame Cohomological Functors}
\label{sec:qtame}

Now we generalize \cref{thm:decomp}
from pfd cohomological functors
to $q$-tame \mbox{\cite[Section 1.1]{MR3524869}}
cohomological functors.

\begin{dfn}
  We say that a functor
  $F \colon \M^{\circ} \rightarrow \mathrm{Vect}_K$
  is \emph{$q$-tame} if $F(u \preceq v)$ has finite rank
  for all $u \prec v \in \M$.
\end{dfn}

\begin{prp}
  \label{prp:qtame}
  Let $F \colon \M^{\circ} \rightarrow \mathrm{Vect}_K$
  be a cohomological functor
  which is $q$-tame.
  Then $F$ is pfd.  
\end{prp}

\begin{proof}
  Let $(x, y) \in \op{int} \M$.
  We show that $F(x, y)$ is finite-dimensional.
  To this end,
  let $\delta > 0$ be such that
  $(x-\delta, y), (x+\delta, y+\delta) \in \M$.
  Now let
  \begin{align*}
      x_0 & := x-\delta, & x_1 & := x, & x_2 & := x+\delta,
    \\
          & & y_1 & := y, \quad \text{and} & y_2 & := y+\delta
                                 .
  \end{align*}
  We consider the commutative diagram
  \begin{equation*}
    \begin{tikzcd}
      F(x_0, y_2)
      \arrow[r]
      \arrow[d]
      &
      F(x_1, y_2)
      \arrow[r]
      \arrow[d]
      &
      F(x_2, y_2)
      \arrow[d]
      \\
      F(x_0, y_1)
      \arrow[r]
      \arrow[d, equal]
      &
      F(x_1, y_1)
      \arrow[r]
      \arrow[d, equal]
      &
      F(x_2, y_1)
      \arrow[d, equal]
      \\
      F(x_0, y_1)
      \arrow[r]
      &
      F(x_1, y_1)
      \arrow[r]
      &
      F(x_2, y_1)
      .
    \end{tikzcd}
  \end{equation*}
  By \cref{prp:cohomological}.4
  all squares in this diagram are middle exact,
  hence
  \begin{equation*}
    \frac{
      F(x_1,y_1)
    }{
      \op{Im} F(x_0 \leq x_1, y_1) + \op{Im} F(x_1, y_1 \leq y_2)
    }
    \cong
    \frac{
      \op{Im} F(x_1 \leq x_2, y_1) + \op{Im} F(x_2, y_1 \leq y_2)
    }{
      \op{Im} F(x_0 \leq x_2, y_1) +
      \op{Im} F(x_2, y_1 \leq y_2)
    }
  \end{equation*}
  by \cref{prp:4Squares}.
  As $F$ is $q$-tame, the numerator on the right hand side
  and both denominators are finite-dimensional.
  Thus, the numerator $F(x_1, y_1) = F(x, y)$ on the left
  has to be finite-dimensional as well.
\end{proof}

This proposition has the following two corollaries.

\begin{cor}
  Let $F \colon \M^{\circ} \rightarrow \mathrm{Vect}_K$
  be a $q$-tame sequentially continuous cohomological functor.
  Then $F$ is pfd and
  \begin{equation*}
    F
    \cong
    \bigoplus_{v \in \op{int} \M} (B_v)^{\oplus \mu(v)}
    ,
  \end{equation*}
  where $\mu := \op{Dgm}(F)$.
\end{cor}

\begin{cor}
  \label{cor:qtameKtame}
  Any continuous function $f \colon X \rightarrow \R$
  is $K$-tame iff $h(f) \colon \M^{\circ} \rightarrow \VectK$ is $q$-tame,
  in which case it decomposes as
  \begin{equation*}
    h(f)
    \cong
    \bigoplus_{v \in \op{int} \M} (B_v)^{\oplus \mu(v)}
    ,
  \end{equation*}
  where $\mu := \op{Dgm}(f)$.
\end{cor}

\begin{proof}
  If $h(f) \colon \M^{\circ} \rightarrow \VectK$ is $q$-tame,
  then it is pfd by \cref{prp:qtame}.
  As the cohomology groups of all open interlevel sets
  appear as values of $h(f)$,
  this in turn implies that $f \colon X \rightarrow \R$
  is $K$-tame.
  The other implications are provided by
  \cref{lem:pfd,thm:decomp}.
\end{proof}

\subsection{Connections to Level Set and Extended Persistence}
\label{sec:connVar}

We now use \cref{thm:decomp} to connect RISC
to two other variants of persistence,
namely level set persistence \cite{Carlsson:2009:ZPH:1542362.1542408}
and extended persistence \cite{MR2472288}.
A posteriori this also implies
that our \cref{dfn:diagramFunction} of the extended persistence diagram
is consistent with the original definition by \cite{MR2472288}.
We make use of the connection to extended persistence in our discussion
of \cref{exm:hood}.
Other than this,
we describe these connections for the reader who is already familiar
with either of these two notions.
To this end,
we assume $l_0$ and $l_1$ intersect the $x$-axis in $-\pi$
and $\pi$ respectively.

\subsubsection{The Level Set Barcode}
\label{sec:connLevelSet}

Here we connect the extended persistence diagram,
as defined
in \cref{dfn:diagramFunction},
to the \emph{level set barcode},
as introduced by \cite{Carlsson:2009:ZPH:1542362.1542408}
as \emph{levelset zigzag persistence};
see also \cite{MR3031814,MR3924175}.
The extended persistence diagram is a multiset of points in $\M$
while the level set barcode is a multiset of pairs
of an integer degree and an interval in $\R$.
To connect the two notions we describe a bijection
between $\op{int} \M$ and $\Z \times \mathcal{I}$,
where $\mathcal{I} \subset 2^{\R}$ is the set of non-empty intervals in $\R$.
To this end, let $u \in \op{int} \M$ and let $n \in \Z$
be the unique integer with $T^n(u) \in D$.
We define
\begin{equation*}
  \nu(u) := n
  \quad \text{and} \quad
  I(u) :=
  \left(\rho_1 \circ T^n\right)(u) \setminus
  \left(\rho_0 \circ T^n\right)(u)
\end{equation*}
and we think of $\nu(u)$ as the degree associated to $u$
and of $I(u)$ as the associated interval.
With this we obtain the bijection
\begin{equation*}
  \beta \colon
  \op{int} \M \rightarrow \Z \times \mathcal{I}, ~
  u \mapsto
  \left(
    \nu(u),
    I(u)
  \right)
  .
\end{equation*}
Now let
$f \colon X \rightarrow \R$ be a piecewise linear function
with $X$ a finite simplicial complex
and let
$\mu := \op{Dgm}(f) \colon \op{int} \M \rightarrow \N_0$
be the associated extended persistence diagram.
We argue that
\[\mu \circ \beta^{-1} \colon \Z \times \mathcal{I} \rightarrow \N_0\]
is the level set barcode of $f$ in the following way.
We choose a \emph{representation}
$p \colon S \rightarrow \op{int} \M$
of the multiset $\mu = \op{Dgm}(f)$.
This means that $p \colon S \rightarrow \op{int} \M$
is some map of sets with
\begin{equation*}
  \mu(v) = \# p^{-1} (v)
  \quad
  \text{for all $v \in \op{int} \M$.}
\end{equation*}
We now describe how
$\beta \circ p \colon S \rightarrow \Z \times \mathcal{I}$
is a representation of the levelset barcode of
$f \colon X \rightarrow \R$.
For each $v \in \op{int} \M$ we choose a basis
$\{\omega_s\}_{s \in p^{-1}(v)}$
for a complement of
\[\sum_{w \succ v} \op{Im} h(f)(v \preceq w)\]
in
$h(f)(v)$.
Moreover, for $n \in \Z$ and $t \in \R$ we set
\begin{equation*}
  S_{n,t} :=
  \left\{s \in S \mid
    n = (\nu \circ p)(s) ~~ \text{and} ~~
    t \in (I \circ p)(s)\right\}
  .
\end{equation*}
As ${t \in \left(\rho_1 \circ T^n \circ p\right)(s)}$
for any ${s \in S_{n,t}}$,
the corresponding cohomology class
${\omega_s \in \left(\mathcal{H}^n \circ \rho \circ T^n \circ p\right)(s)}$
has a pullback
\[\omega_s |_t \in \mathcal{H}^{n} (f^{-1}(t)).\]
Now suppose we can show that 
$\{\omega_s |_t\}_{s \in S_{n,t}}$
is a basis of $\mathcal{H}^{n} (f^{-1}(t))$
for any $t \in \R$,
then this suggests that $\beta \circ p$
is a representation of the levelset barcode
for the following reason.
For each $s \in S$ we obtain a degree $n := (\nu \circ p)(s)$
as well as an interval $(I \circ p)(s)$.
Moreover, for each $t \in (I \circ p)(s)$ we have a basis
element $\omega_s |_t \in \mathcal{H}^{n} (f^{-1}(t))$
in the $n$-th cohomology of the fiber of $t$.
Thus, the entire family $\{\omega_s\}_{s \in S}$
induces a simultaneous decomposition of the cohomology spaces
of all fibers of $f$ in such a way that any two basis elements
associated to the same $s \in S$ arise as pullbacks of the same
cohomology class $\omega_s$.
While this is not the original definition of the level set barcode from
\cite{Carlsson:2009:ZPH:1542362.1542408},
it comes very close and we omit the remaining details.
We conclude with the following proposition
providing the missing ingredient to the above argument.

\begin{prp}
  Let $t \in \R$.
  Then the family
  $\{\omega_s |_t\}_{s \in S_{n,t}}$
  is a basis of $\mathcal{H}^{n} (f^{-1}(t))$.
\end{prp}

\begin{figure}[t]
  \centering
  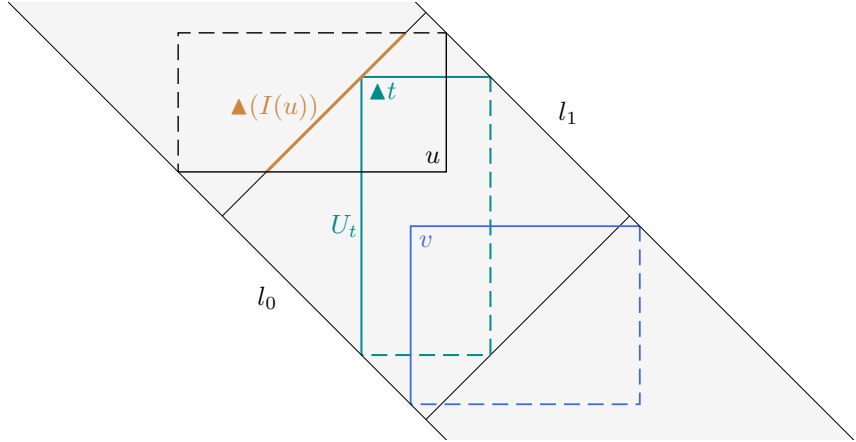
  \caption{
    The axis-aligned rectangle $U_t$
    and the support of the indecomposable $B_v$.
  }
  \label{fig:connLevelSet}
\end{figure}

\begin{proof}
  Without loss of generality we assume
  that $f$ is a simplicial map to some
  finite geometric simplicial complex with $\R$ as its ambient space
  and vertex set $V \subset \R$.
  Now let
  \[U_t :=
    \{u \in D \mid t \in I(u)\} =
    \left(\downarrow \blacktriangle t\right) \cap
    \op{int} \left(\uparrow \left(T^{-1} \circ \blacktriangle\right)(t)\right),
  \]
  see also \cref{fig:connLevelSet}.
  Then we have
  \begin{equation}
    \label{eq:shiftRestrRepr}
    S_{n,t} = \left(T^n \circ p\right)^{-1} (U_t)
    .
  \end{equation}
  Moreover,
  let $\varepsilon > 0$ be such that
  $(t - \varepsilon, t + \varepsilon) \cap V \setminus \{t\} = \emptyset$.
  Then the homotopy invariance of $\mathcal{H}^n$ implies
  that the induced map
  \begin{equation*}
    \mathcal{H}^n (f^{-1}(t - \delta, t + \delta)) \rightarrow
    \mathcal{H}^{n} (f^{-1}(t))
  \end{equation*}
  is an isomorphism
  for any $0 < \delta \leq \varepsilon$.
  Now
  \[\{t\} = \bigcap_{u \in U_t} \rho_1 (u)\]
  and furthermore,
  $\{(t - \delta, t + \delta) \mid \delta > 0\}$
  is a final subset of
  $\{\rho_1(u) \mid u \in U_t\}$.
  Thus, we have
  \begin{equation}
    \label{eq:dirLimStab}
    \begin{split}
      \mathcal{H}^{n} (f^{-1}(t))
      & \cong
      \varinjlim_{u \in U_t}
      \left(\mathcal{H}^{n} \circ f^{-1} \circ \rho_1\right)(u)
      \\
      & \cong
      \varinjlim_{u \in U_t}
      \left(\mathcal{H}^{n} \circ f^{-1} \circ \rho\right)(u)
      \\
      & =
      \varinjlim_{u \in U_t} \left(h(f) \circ T^{-n}\right)(u)
      .
    \end{split}
  \end{equation}
  Here the second isomorphism follows from the fact
  that $U_t$ has a final subset on which $\rho_0$ is empty,
  hence $(\rho_1, \emptyset)$ and $\rho$ agree on this subset.
  By the Yoneda \cref{lem:yonedaContra}
  there is a unique natural transformation
  \begin{equation*}
    \varphi_s \colon
    B_{(T^n \circ p)(s)} \rightarrow
    h(f) \circ T^{-n}
  \end{equation*}
  sending $1 \in K = B_{(T^n \circ p)(s)} ((T^n \circ p)(s))$
  to $\omega_s \in \left(h(f) \circ T^{-n} \circ p\right)(s)$
  for any $s \in S$.
  By \cref{prp:decomp} the family
  $\{\varphi_s\}_{s \in S}$
  yields a natural isomorphism
  \begin{equation*}
    \varphi \colon
    \bigoplus_{s \in S} B_{(T^n \circ p)(s)} \rightarrow
    h(f) \circ T^{-n}
    .
  \end{equation*}
  Now for $v \in \op{int} \M$ we have
  \begin{equation*}
    \varinjlim_{u \in U_t} B_v (u) \cong
    \begin{cases}
      K & v \in U_t
      \\
      \{0\} & v \notin U_t
      ,
    \end{cases}
  \end{equation*}
  see also \cref{fig:connLevelSet}.
  In conjunction with
  \eqref{eq:shiftRestrRepr} and \eqref{eq:dirLimStab}
  we obtain
  that
  $\{\omega_s |_t\}_{s \in S_{n,t}}$
  is a basis for $\mathcal{H}^{n} (f^{-1}(t))$.
\end{proof}

\subsubsection{Extended Persistence}
\label{sec:connExt}

Here we describe how the extended persistence diagram
as defined by \cite{MR2472288}
corresponds to our \cref{dfn:diagramFunction}.
We note that the connection between the levelset barcode
and extended persistence is well-understood
\cite{Carlsson:2009:ZPH:1542362.1542408,MR3031814}.
Here we provide a correspondence between RISC
and extended persistence without requiring any finiteness assumptions
other than the RISC being pfd.

\begin{figure}[t]
  \centering
  \begin{subfigure}[t]{0.51\textwidth}
    \centering
\begingroup%
  \makeatletter%
  \providecommand\color[2][]{%
    \errmessage{(Inkscape) Color is used for the text in Inkscape, but the package 'color.sty' is not loaded}%
    \renewcommand\color[2][]{}%
  }%
  \providecommand\transparent[1]{%
    \errmessage{(Inkscape) Transparency is used (non-zero) for the text in Inkscape, but the package 'transparent.sty' is not loaded}%
    \renewcommand\transparent[1]{}%
  }%
  \providecommand\rotatebox[2]{#2}%
  \newcommand*\fsize{\dimexpr\f@size pt\relax}%
  \newcommand*\lineheight[1]{\fontsize{\fsize}{#1\fsize}\selectfont}%
  \ifx\svgwidth\undefined%
    \setlength{\unitlength}{192.72681427bp}%
    \ifx\svgscale\undefined%
      \relax%
    \else%
      \setlength{\unitlength}{\unitlength * \real{\svgscale}}%
    \fi%
  \else%
    \setlength{\unitlength}{\svgwidth}%
  \fi%
  \global\let\svgwidth\undefined%
  \global\let\svgscale\undefined%
  \makeatother%
  \begin{picture}(1,0.94174752)%
    \lineheight{1}%
    \setlength\tabcolsep{0pt}%
    \put(0,0){\includegraphics[width=\unitlength,page=1]{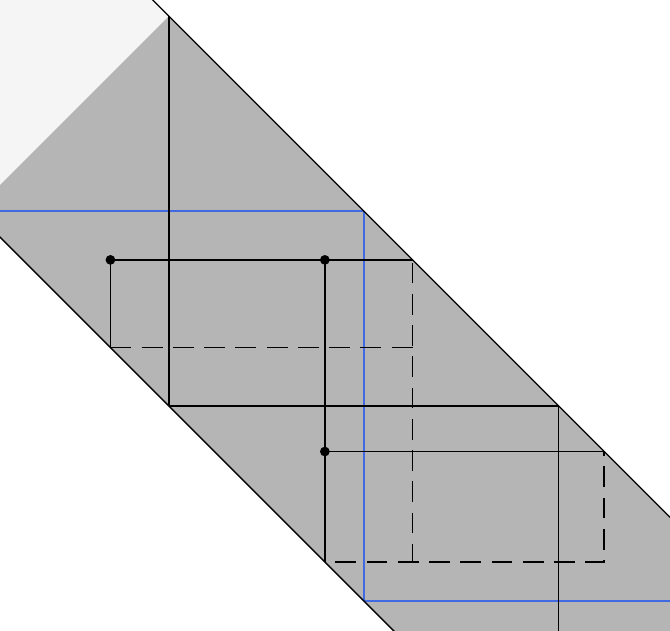}}%
    \put(0.77098729,0.44745927){\makebox(0,0)[t]{\lineheight{1.25}\smash{\begin{tabular}[t]{c}$l_1$\end{tabular}}}}%
    \put(0.34273228,0.18396688){\makebox(0,0)[t]{\lineheight{1.25}\smash{\begin{tabular}[t]{c}$l_0$\end{tabular}}}}%
  \end{picture}%
\endgroup%

  \end{subfigure}
  \begin{subfigure}[t]{0.46\textwidth}
    \centering
\begingroup%
  \makeatletter%
  \providecommand\color[2][]{%
    \errmessage{(Inkscape) Color is used for the text in Inkscape, but the package 'color.sty' is not loaded}%
    \renewcommand\color[2][]{}%
  }%
  \providecommand\transparent[1]{%
    \errmessage{(Inkscape) Transparency is used (non-zero) for the text in Inkscape, but the package 'transparent.sty' is not loaded}%
    \renewcommand\transparent[1]{}%
  }%
  \providecommand\rotatebox[2]{#2}%
  \newcommand*\fsize{\dimexpr\f@size pt\relax}%
  \newcommand*\lineheight[1]{\fontsize{\fsize}{#1\fsize}\selectfont}%
  \ifx\svgwidth\undefined%
    \setlength{\unitlength}{174.01544952bp}%
    \ifx\svgscale\undefined%
      \relax%
    \else%
      \setlength{\unitlength}{\unitlength * \real{\svgscale}}%
    \fi%
  \else%
    \setlength{\unitlength}{\svgwidth}%
  \fi%
  \global\let\svgwidth\undefined%
  \global\let\svgscale\undefined%
  \makeatother%
  \begin{picture}(1,1.04301084)%
    \lineheight{1}%
    \setlength\tabcolsep{0pt}%
    \put(0,0){\includegraphics[width=\unitlength,page=1]{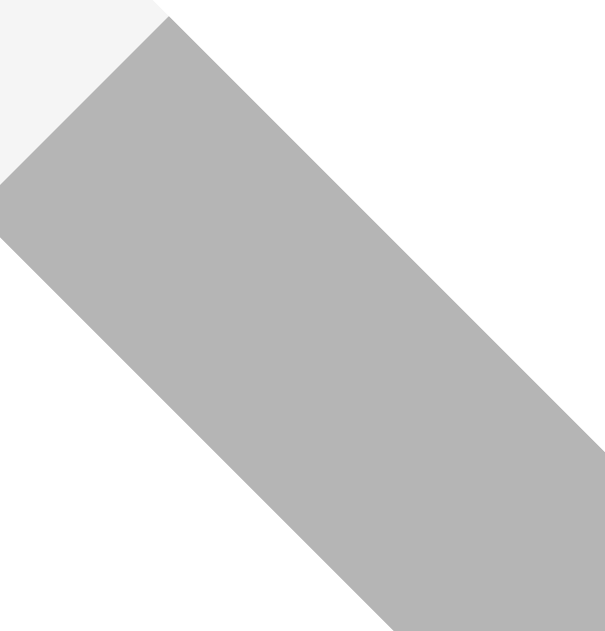}}%
    \put(0.7096775,0.42473125){\makebox(0,0)[t]{\lineheight{1.25}\smash{\begin{tabular}[t]{c}{\large$\mathrm{Rel}_2$}\end{tabular}}}}%
    \put(0.38709667,0.74731209){\makebox(0,0)[t]{\lineheight{1.25}\smash{\begin{tabular}[t]{c}{\large$\mathrm{Ord}_0$}\end{tabular}}}}%
    \put(0.18279584,0.59139792){\makebox(0,0)[t]{\lineheight{1.25}\smash{\begin{tabular}[t]{c}{\large$\mathrm{Rel}_1$}\end{tabular}}}}%
    \put(0.50537625,0.26881708){\makebox(0,0)[t]{\lineheight{1.25}\smash{\begin{tabular}[t]{c}{\large$\mathrm{Ord}_1$}\end{tabular}}}}%
    \put(0.76344083,0.19354833){\makebox(0,0)[t]{\lineheight{1.25}\smash{\begin{tabular}[t]{c}{\large$\mathrm{Ext}_2$}\end{tabular}}}}%
    \put(0.44086042,0.51612917){\makebox(0,0)[t]{\lineheight{1.25}\smash{\begin{tabular}[t]{c}{\large$\mathrm{Ext}_1$}\end{tabular}}}}%
    \put(0.14516124,0.86559168){\makebox(0,0)[t]{\lineheight{1.25}\smash{\begin{tabular}[t]{c}{\large$\mathrm{Ext}_0$}\end{tabular}}}}%
    \put(0,0){\includegraphics[width=\unitlength,page=2]{subdiagrams-coho.pdf}}%
  \end{picture}%
\endgroup%

  \end{subfigure}
  \caption{
    In the graphic on the left we see the subposet of $\M$
    corresponding to extended persistent cohomology
    shaded in blue
    as well as three vertices contained in the domains corresponding to
    $1$-dimensional relative, extended,
    and ordinary persistent cohomology.
    On the right we see the regions
    in the strip \(\M\) corresponding to the
    ordinary, relative, and extended subdiagrams \cite{MR2472288}.
    In both figures the support \(\downarrow \op{Im} \blacktriangle\)
    of RISC and of extended persistence diagrams
    is shaded in grey.    
  }
  \label{fig:extended}
\end{figure}

Now let $f \colon X \rightarrow \R$ be a continuous function
such that its RISC $h(f) \colon \M^{\circ} \rightarrow \mathrm{Vect}_K$ is pfd.
We consider the left hand side of \cref{fig:extended}
and the restriction of $h(f)$ to the subposet of $\M$,
which is shaded in blue in this figure.
Here each point on the horizontal blue line segment
to the upper left is assigned the cohomology space in degree $0$
of an open sublevel set of $f$, of $X$,
or of a pair with $X$ as the first component and
an open superlevel set as the second component.
Up to isomorphism of posets,
this is the extended persistent cohomology
of $f \colon X \rightarrow \R$ in degree $0$.
(Strictly speaking, in the original definition,
which is for piecewise linear functions,
closed sublevel sets and closed superlevel sets are used.
When considering continuous functions with fewer restrictions,
it is not uncommon to consider preimages of open subsets
in place of closed sets.)
Similarly, any point on the vertical blue line segment in the center
is assigned the cohomology space of some pair of preimages in degree $1$
and any point on the horizontal blue line at the lower right
is assigned the cohomology of some pair in degree $2$.
By \cref{thm:decomp} the RISC $h(f)$ decomposes
into contravariant blocks as in \cref{dfn:contraBlock}.
Now the support of each such contravariant block
intersects exactly one of these blue line segments.
We focus on the vertical blue line segment in the center
of the graphic on the left in \cref{fig:extended},
which carries the extended persistent cohomology
of $f \colon X \rightarrow \R$ in degree $1$.
Any choice of decomposition
of $h(f) \colon \M^{\circ} \rightarrow \vectK$
yields a decomposition of its restriction
to this line segment and thus of persistent cohomology in degree $1$.
Moreover, the support of the contravariant block
assigned to any of the black dots
in the graphic on the left of \cref{fig:extended}
intersects this vertical line segment.

First assume that
the black dot on the lower right appears in
$\op{Dgm}(f)$.
Then the restriction of the associated contravariant block
to the vertical blue line segment
is a direct summand of the persistent cohomology
of $f \colon X \rightarrow \R$ in degree $1$
and the intersection of its support with the blue line segment
is the life span of the corresponding feature
in the sense that the point of intersection of the upper edge
marks the birth of a cohomology class that dies
as soon as it is pulled back to the open sublevel set corresponding
to the point of intersection of the lower edge.
Moreover, this life span is encoded by the position of this black dot.
Now this particular black dot on the lower right
of the left graphic in \cref{fig:extended}
is contained in the triangular region labeled as $\mathrm{Ord}_1$
in the graphic on the right hand side of \cref{fig:extended}.
Furthermore, any vertex of the extended persistence diagram $\op{Dgm}(f)$
contained in the triangular region labeled $\mathrm{Ord}_1$
describes a feature of $f \colon X \rightarrow \R$,
which is born at some open sublevel set
and also dies at some open sublevel set.
Thus, up to reparametrization,
the ordinary persistence diagram
of $f \colon X \rightarrow \R$ in degree $1$
is the restriction of $\op{Dgm}(f) \colon \M \rightarrow \N_0$
to the region labeled $\mathrm{Ord}_1$.

Now suppose that the black dot on the upper left of
the left hand image in \cref{fig:extended}
appears in $\op{Dgm}(f)$.
Then the intersection of the support of the associated
contravariant block with the vertical blue line segment
describes the life span of a feature which is born
at the cohomology of $X$ relative to some open superlevel set in degree $1$
and also dies at some relative cohomology space.
Moreover, this is true for any vertex of $\op{Dgm}(f)$
contained in the triangular region labeled $\mathrm{Rel}_1$
in the graphic on the right in \cref{fig:extended}.
Thus, up to reparametrization,
the relative subdiagram of $f \colon X \rightarrow \R$ in degree $1$
is the restriction of $\op{Dgm}(f) \colon \M \rightarrow \N_0$
to the region labeled $\mathrm{Rel}_1$.

Finally, the black dot to the upper right
in the left hand image of \cref{fig:extended}
(if in $\op{Dgm}(f)$),
or any other vertex of $\op{Dgm}(f)$
in the square region labeled $\mathrm{Ext}_1$
in the graphic on the right in \cref{fig:extended}, describes a feature,
which is born
at the cohomology of $X$ relative to some open superlevel set
and dies at the cohomology of some sublevel set of $f \colon X \rightarrow \R$.
Thus,
the extended subdiagram of $f \colon X \rightarrow \R$ in degree $1$
is the restriction of $\op{Dgm}(f) \colon \M \rightarrow \N_0$
to the square region labeled $\mathrm{Ext}_1$.

As we have analogous correspondences for each line segment
of the subposet of $\M$ shaded in blue
in the left image of \cref{fig:extended},
we obtain a partition of the lower right part of the strip $\M$ into regions,
corresponding to ordinary, relative, and extended subdiagrams
of $\op{Dgm}(f)$ analogous to \cite{MR2472288}.

\subsection{Decomposition of Homological Functors}

We note that $\M$ is self-dual as a lattice.
Thus, there is an obvious dual version of \cref{thm:decomp}.
Now we state this result for the sake of completeness.

\begin{remark}
  \label{remark:dual}
  The reflection at the diagonal
  \[\M \rightarrow \M,\, (x, y) \mapsto (y, x)\]
  is a self-duality of the lattice $\M$
  in the sense that it is order-reversing
  and interchanges joins and meets.
  In particular, any covariant functor on $\M$
  can be made into a contravariant functor and vice versa
  by precomposition with this reflection.
\end{remark}

The following is dual to \cref{dfn:contraBlock}
in the sense of this remark,
see also \cref{fig:block}.

\begin{dfn}[Block]
  For $v \in \op{int} \M$ we define
  \begin{equation*}
    B^v \colon \M \rightarrow \mathrm{Vect}_K,\,
    w \mapsto
    \begin{cases}
      K & w \in (\uparrow v) \cap \op{int} (\downarrow T(v))
      \\
      \{0\} & \text{otherwise}
      ,
    \end{cases}
  \end{equation*}
  where $\op{int} (\downarrow T(v))$ is the interior
  of the downset of $T(v)$ in $\M$.
  The internal maps are identities whenever both domain and codomain are $K$,
  otherwise they are zero.
\end{dfn}



With this we may state the dual of \cref{thm:decomp}
in the sense of \cref{remark:dual}.


\begin{thm}
  \label{thm:decompDual}
  Any sequentially continuous homological functor
  $F \colon \M \rightarrow \vectK$
  decomposes as
  \begin{equation*}
    F \cong \bigoplus_{v \in \op{int} \M} (B^v)^{\oplus \nu(v)}
    ,
  \end{equation*}
  where $\nu := \op{Dgm}(F)$ is defined dually to \eqref{eq:diagram} as
  \begin{equation*}
    \op{Dgm} (F) : \op{int} \M \rightarrow \N_0,
    \,
    v \mapsto
    \dim_K F(v) -
    \dim_K \sum_{u \prec v} \op{Im} F(u \preceq v)
    .
  \end{equation*}
\end{thm}

\section{Interleavings}
\label{sec:interleavings}

Let $X$ be a non-empty topological space.
For two functions $f, g \colon X \rightarrow \R$
we define \[\mathbf{d}(f, g) := (\inf (g-f), \sup (g-f)),\]
then $\mathbf{d}$ can be thought of as a type of distance assigning two values
as it satisfies a form of triangle inequality:
If $\mathbf{d}(f_1, f_2), \mathbf{d}(f_2, f_3) \in \R^{\circ} \times \R$,
then also $\mathbf{d}(f_1, f_3) \in \R^{\circ} \times \R$ and
\begin{equation}
  \label{eq:triaIneq}
  \mathbf{d}(f_1, f_3) \preceq \mathbf{d}(f_1, f_2) + \mathbf{d}(f_2, f_3).
\end{equation}
Now let $(x, y) = \mathbf{d}(f, g) \in \R^2$ for two functions
$f, g \colon X \rightarrow \R$,
then $\frac{1}{2}(y-x)$ is the minimal infinity distance between $f$
and any shift of $g$.
Moreover, $-\frac{1}{2}(x+y)$ is the shift,
for which this minimum is attained.
This is the information we obtain from $\mathbf{d}$ about $f$ and $g$
and \eqref{eq:triaIneq} captures in a single inequality
how shifts and the infinity distance interact for three functions.
In this section we describe how this type of distance $\mathbf{d}$
and the triangle inequality \eqref{eq:triaIneq} resurface
in a different form
when considering relative interlevel set cohomology.

\begin{figure}[t]
  \centering
\begingroup%
  \makeatletter%
  \providecommand\color[2][]{%
    \errmessage{(Inkscape) Color is used for the text in Inkscape, but the package 'color.sty' is not loaded}%
    \renewcommand\color[2][]{}%
  }%
  \providecommand\transparent[1]{%
    \errmessage{(Inkscape) Transparency is used (non-zero) for the text in Inkscape, but the package 'transparent.sty' is not loaded}%
    \renewcommand\transparent[1]{}%
  }%
  \providecommand\rotatebox[2]{#2}%
  \newcommand*\fsize{\dimexpr\f@size pt\relax}%
  \newcommand*\lineheight[1]{\fontsize{\fsize}{#1\fsize}\selectfont}%
  \ifx\svgwidth\undefined%
    \setlength{\unitlength}{365.24180603bp}%
    \ifx\svgscale\undefined%
      \relax%
    \else%
      \setlength{\unitlength}{\unitlength * \real{\svgscale}}%
    \fi%
  \else%
    \setlength{\unitlength}{\svgwidth}%
  \fi%
  \global\let\svgwidth\undefined%
  \global\let\svgscale\undefined%
  \makeatother%
  \begin{picture}(1,0.55442722)%
    \lineheight{1}%
    \setlength\tabcolsep{0pt}%
    \put(0,0){\includegraphics[width=\unitlength,page=1]{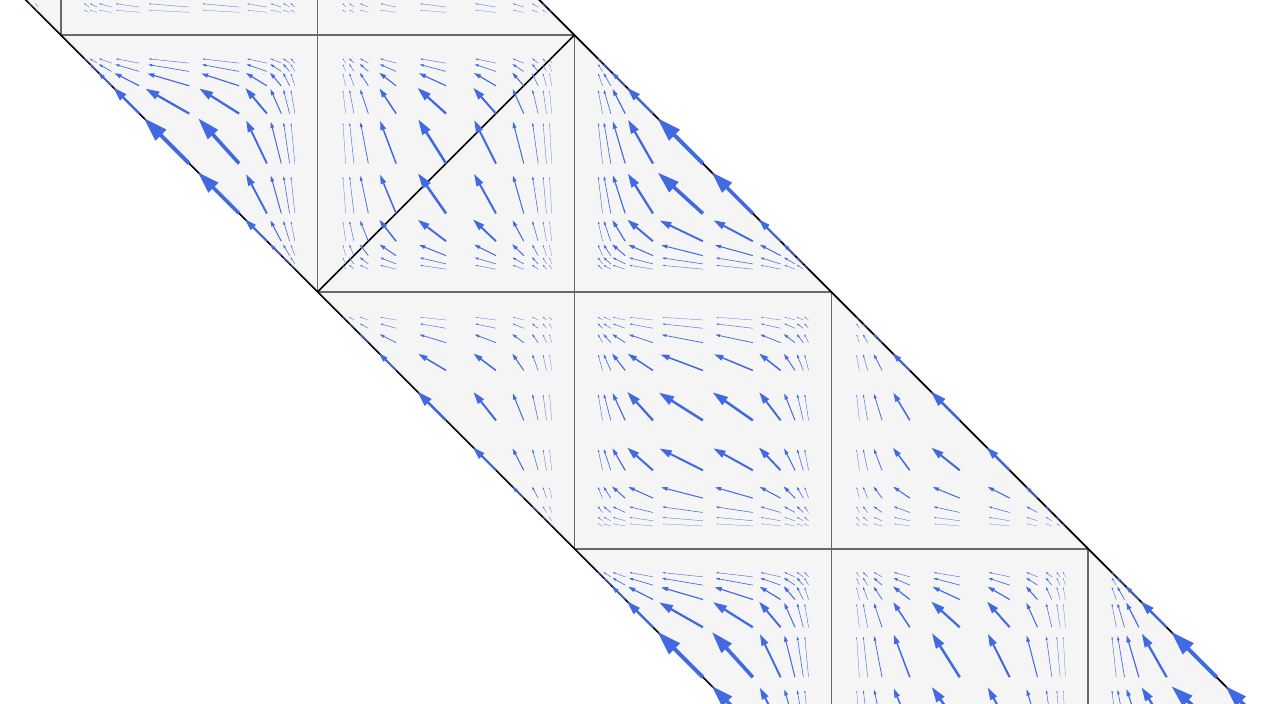}}%
    \put(0.32981965,0.22295758){\makebox(0,0)[rt]{\lineheight{1.25}\smash{\begin{tabular}[t]{r}\(l_0\)\end{tabular}}}}%
    \put(0.66921212,0.32422412){\makebox(0,0)[lt]{\lineheight{1.25}\smash{\begin{tabular}[t]{l}\(l_1\)\end{tabular}}}}%
  \end{picture}%
\endgroup%

  \caption{
    The lattice automorphism $\alpha_a \colon \M \rightarrow \M$
    for $a = (-0.315, 0.525)$.}
  \label{fig:alpha}
\end{figure}

As in several instances before,
suppose that $l_0$ and $l_1$ intersect the $x$-axis in $-\pi$ and $\pi$,
respectively.
Moreover, let $\op{Aut}(\M)$ be the automorphism group of $\M$
in the category of lattices.
Then there is a unique group homomorphism
\begin{equation*}
  \alpha \colon \R^2 \rightarrow \op{Aut}(\M)
  \quad \text{with} \quad
  \op{ev}_0 \circ \, \alpha = \arctan \times \arctan,
\end{equation*}
where
$\op{ev}_0 \colon \op{Aut}(\M) \rightarrow \M, \, \varphi \mapsto \varphi(0)$
is the evaluation at the origin,
see also \cref{fig:alpha}.
We provide an explicit description of $\alpha$ in \cref{remark:alpha} below.
We also note that $\R^{\circ} \times \R$ is an abelian
\href{https://en.wikipedia.org/wiki/Group_object}{group object}
in the category of lattices and the
\href{https://en.wikipedia.org/wiki/Currying}{uncurrying}
\begin{equation*}
  \alpha^b \colon (\R^{\circ} \times \R) \times \M \rightarrow \M
\end{equation*}
of $\alpha$ exhibits $\M$ as a
\href{https://ncatlab.org/nlab/show/module+over+a+monoid}{module}
over $\R^{\circ} \times \R$.
In particular,
$\alpha \colon \R^{\circ} \times \R \rightarrow \op{Aut}(\M)$ is monotone.
For better readability we will often write the argument of $\alpha$
as an index
by writing $\alpha_a$ in place of $\alpha(a)$
for $a \in \R^{\circ} \times \R$.

\begin{remark}[An Explicit Description of $\alpha$]
  \label{remark:alpha}
  Let $a := (a_1, a_2) \in \R^{\circ} \times \R$.
  We construct
  $\alpha_{a} \colon \M \rightarrow \M$
  as follows.
  First we have the inverse radial projection or
  $S^1$-valued inverse tangent function
  \begin{equation*}
    \phi \colon
    \R \rightarrow S^1, ~
    y \mapsto \frac{1}{\sqrt{1+y^2}} (1, y)
    =
    \exp(i \arctan(y))
    .
  \end{equation*}
  Then we define
  \begin{equation*}
    g_{a} \colon S^1 \rightarrow S^1,
    (x, y) \mapsto
    \begin{cases}
      (0, y),
      &
      x = 0
      \\
      \phi
      \left(
        \frac{y}{x} + a_2
      \right),
      &
      x > 0
      \\
      - \phi
      \left(
        \frac{y}{x} - a_1
      \right),
      &
      x < 0
      .
    \end{cases}
  \end{equation*}
  As $g_{a} (i) = i$,
  where $i = (0, 1)$ is the imaginary unit,
  there is a unique continuous map
  $\tilde{g}_{a} \colon \R \rightarrow \R$
  with
  \begin{equation*}
    \tilde{g}_{a} \left(\frac{\pi}{2}\right) = \frac{\pi}{2}
    \quad \text{and} \quad
    \left(g_{a} \circ \exp\right)(it) =
    \exp \left(i \tilde{g}_{a}(t)\right)
    \quad
    \text{for all $t \in \R$.}
  \end{equation*}
  In other words,
  $\tilde{g}_{a} \colon \R \rightarrow \R$
  is the unique continuous
  --
  or equivalently $2\pi\Z$-equivariant
  --
  map
  such that the diagram
  \begin{equation*}
    \begin{tikzcd}
      &[-5ex]
      \frac{\pi}{2}
      \arrow[r, dashed, maps to]
      &[+2ex]
      \frac{\pi}{2}
      &[-5ex]
      \\[-4ex]
      t
      \arrow[d, maps to]
      &
      \R
      \arrow[r, dashed, "\tilde{g}_a"]
      \arrow[d]
      &
      \R
      \arrow[d]
      &
      t
      \arrow[d, maps to]
      \\[+1.5ex]
      e^{it}
      &
      S^1
      \arrow[r, "g_a"']
      &
      S^1
      &
      e^{it}
    \end{tikzcd}
  \end{equation*}
  commutes.
  With $\sigma \colon \R \rightarrow \R, \ t \mapsto \pi - t$
  being the reflection at $\pi / 2$ we have
  \begin{equation*}
    \alpha_{a} =
    \left(\sigma \circ \tilde{g}_{a} \circ \sigma\right) \times
    \tilde{g}_{a}
    .
  \end{equation*}
\end{remark}

Now suppose we have
$a := \mathbf{d}(f, g) \in \R^{\circ} \times \R$.
In the following we describe how this induces a natural transformation
\begin{equation*}
  \vec{h}(f, g) \colon h(g) \circ \alpha_a \rightarrow h(f)
  .
\end{equation*}
To this end,
let $F \colon \M \rightarrow \big(\mathrm{Vect}_K^{\Z}\big)^\circ$
be defined as in the construction of $h(f)$ in \cref{sec:risc}.
Or, in other words,
we define $F$ to be the transform of $(h(f))^{\circ}$
under the $2$-adjunction from \cref{lem:2adj}.
Completely analogously we have a functor
$G \colon \M \rightarrow \big(\mathrm{Vect}_K^{\Z}\big)^\circ$
as we would use it in the construction of $h(g)$.
We use \cref{lem:natExt} to construct a natural transformation
${\varphi \colon F \rightarrow G \circ \alpha_a}$
and then we obtain
$\vec{h}(f, g) \colon h(g) \circ \alpha_a \rightarrow h(f)$
as
$\vec{h}(f, g) := \op{ev}^0 \circ \, \varphi^{\circ}$
by whiskering with the evaluation at $0$
analogously to the construction of $h(f)$ from $F$.
Now suppose $a = (x, y)$ and $(r, s) \subseteq \R$ is an open interval,
then we have
\begin{equation*}
  f^{-1} (r, s) \subseteq g^{-1} (r+x, s+y)
\end{equation*}
and thus
\begin{equation*}
  (f^{-1} \circ \rho)(u) \subseteq (g^{-1} \circ \rho \circ \alpha_a)(u)
\end{equation*}
for all $u \in \M$.
As in \cite[Section 3]{MR3413628}, this induces a linear map
\begin{equation*}
  \mathcal{H}^{\bullet} (f^{-1} \circ \rho)(u)
  \leftarrow
  \mathcal{H}^{\bullet} (g^{-1} \circ \rho \circ \alpha_a)(u)
  ,
\end{equation*}
which is natural in $u \in \M$.
Now we have
\begin{align*}
  F(u) &= \left(\mathcal{H}^{\bullet} \circ f^{-1} \circ \rho\right)(u)
  & & \text{for all $u \in D$ and}
  \\
  (G \circ \alpha_a)(u)
       &= \left(
         \mathcal{H}^{\bullet} \circ g^{-1} \circ \rho \circ \alpha_a
         \right)(u)
  & & \text{for all $u \in E := \alpha_{(-a)}(D)$,}
\end{align*}
see also \cref{fig:DandE}.
\begin{figure}[t]
  \centering
\begingroup%
  \makeatletter%
  \providecommand\color[2][]{%
    \errmessage{(Inkscape) Color is used for the text in Inkscape, but the package 'color.sty' is not loaded}%
    \renewcommand\color[2][]{}%
  }%
  \providecommand\transparent[1]{%
    \errmessage{(Inkscape) Transparency is used (non-zero) for the text in Inkscape, but the package 'transparent.sty' is not loaded}%
    \renewcommand\transparent[1]{}%
  }%
  \providecommand\rotatebox[2]{#2}%
  \newcommand*\fsize{\dimexpr\f@size pt\relax}%
  \newcommand*\lineheight[1]{\fontsize{\fsize}{#1\fsize}\selectfont}%
  \ifx\svgwidth\undefined%
    \setlength{\unitlength}{308.37929535bp}%
    \ifx\svgscale\undefined%
      \relax%
    \else%
      \setlength{\unitlength}{\unitlength * \real{\svgscale}}%
    \fi%
  \else%
    \setlength{\unitlength}{\svgwidth}%
  \fi%
  \global\let\svgwidth\undefined%
  \global\let\svgscale\undefined%
  \makeatother%
  \begin{picture}(1,0.53505538)%
    \lineheight{1}%
    \setlength\tabcolsep{0pt}%
    \put(0,0){\includegraphics[width=\unitlength,page=1]{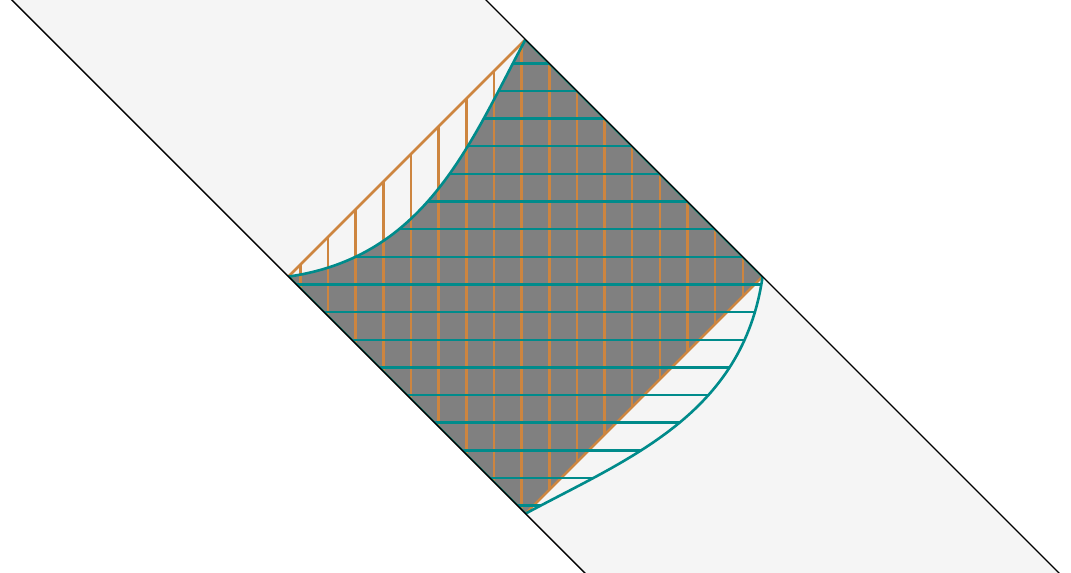}}%
    \put(0.62477006,0.40253627){\makebox(0,0)[t]{\lineheight{1.25}\smash{\begin{tabular}[t]{c}$l_1$\end{tabular}}}}%
    \put(0.32888346,0.16927158){\makebox(0,0)[t]{\lineheight{1.25}\smash{\begin{tabular}[t]{c}$l_0$\end{tabular}}}}%
  \end{picture}%
\endgroup%

  \caption{
    The fundamental domain $\textcolor{Peru}{D}$,
    its shift $\textcolor{DarkCyan}{E}$ by $\alpha(-a)$,
    and their intersection shaded in grey.
  }
  \label{fig:DandE}
\end{figure}
Thus, we may restrict $F$ and $G \circ \alpha_a$
to the intersection $D \cap E$ to obtain the natural transformation
\[\eta \colon
  F |_{D \cap E} \rightarrow (G \circ \alpha_a) |_{D \cap E}
  .
\]
As shown in \cref{fig:DandE},
the intersection $D \cap E$ is no fundamental domain in general,
so this does not describe a natural transformation from $F$
to $G \circ \alpha_a$.
In order to extend $\eta$ to a natural transformation
${\varphi \colon F \rightarrow G \circ \alpha_a}$
we use \cref{lem:natExt}.
So we need to supply the missing ingredient,
which is a natural transformation
\[\nu \colon
  F |_{D \cap T(E)} \rightarrow (G \circ \alpha_a) |_{D \cap T(E)}.
\]
To this end, we consider the monotone map
\begin{equation*}
  \xi \colon
  u \mapsto
  (
  (g^{-1} \circ \rho_1 \circ \alpha_a)(u),
  (f^{-1} \circ \rho_0                )(u)
  )
\end{equation*}
from $\M$ to the set of pairs of open subspaces of $X$.
In some sense $\xi$ interpolates between
$f^{-1} \circ \rho$ and $g^{-1} \circ \rho \circ \alpha_a$,
since we have the chain of inclusions
\begin{equation*}
  \label{eq:xiSandwich}
  f^{-1} \circ \rho \subseteq
  \xi \subseteq
  g^{-1} \circ \rho \circ \alpha_a
\end{equation*}
pointwise in $\M$
and moreover, $\xi$ agrees with $f^{-1} \circ \rho$
when restricted to the region
shaded in red in \cref{fig:xi}.
\begin{figure}[t]
  \centering
\begingroup%
  \makeatletter%
  \providecommand\color[2][]{%
    \errmessage{(Inkscape) Color is used for the text in Inkscape, but the package 'color.sty' is not loaded}%
    \renewcommand\color[2][]{}%
  }%
  \providecommand\transparent[1]{%
    \errmessage{(Inkscape) Transparency is used (non-zero) for the text in Inkscape, but the package 'transparent.sty' is not loaded}%
    \renewcommand\transparent[1]{}%
  }%
  \providecommand\rotatebox[2]{#2}%
  \newcommand*\fsize{\dimexpr\f@size pt\relax}%
  \newcommand*\lineheight[1]{\fontsize{\fsize}{#1\fsize}\selectfont}%
  \ifx\svgwidth\undefined%
    \setlength{\unitlength}{308.37929535bp}%
    \ifx\svgscale\undefined%
      \relax%
    \else%
      \setlength{\unitlength}{\unitlength * \real{\svgscale}}%
    \fi%
  \else%
    \setlength{\unitlength}{\svgwidth}%
  \fi%
  \global\let\svgwidth\undefined%
  \global\let\svgscale\undefined%
  \makeatother%
  \begin{picture}(1,0.53505538)%
    \lineheight{1}%
    \setlength\tabcolsep{0pt}%
    \put(0,0){\includegraphics[width=\unitlength,page=1]{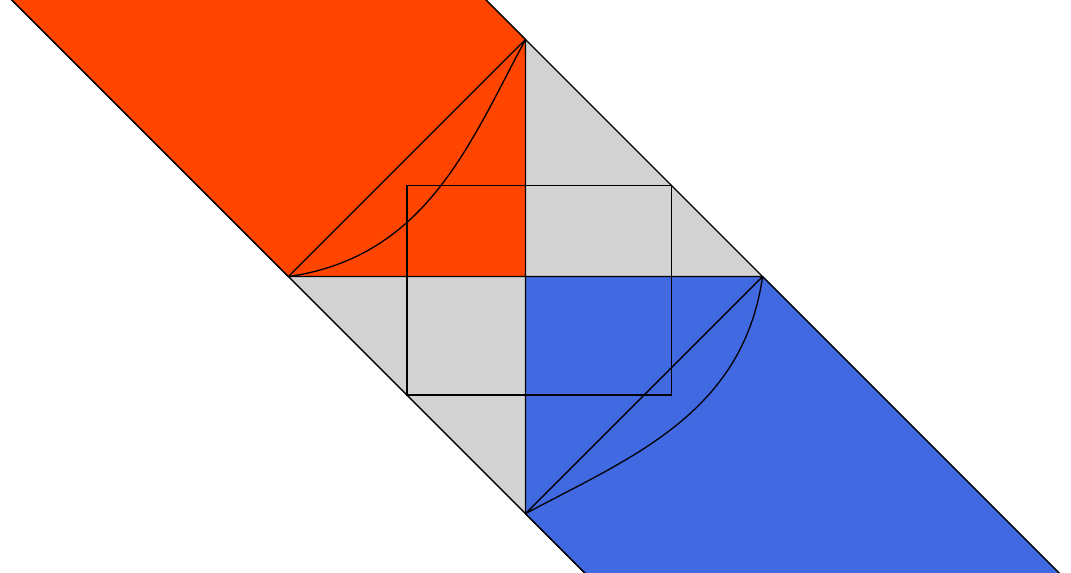}}%
    \put(0.69003684,0.15387463){\makebox(0,0)[t]{\lineheight{1.25}\smash{\begin{tabular}[t]{c}$T^{-1}(v)$\end{tabular}}}}%
    \put(0.3948339,0.36826573){\makebox(0,0)[t]{\lineheight{1.25}\smash{\begin{tabular}[t]{c}$v$\end{tabular}}}}%
    \put(0,0){\includegraphics[width=\unitlength,page=2]{regions-stability.pdf}}%
    \put(0.62477006,0.40253627){\makebox(0,0)[t]{\lineheight{1.25}\smash{\begin{tabular}[t]{c}$l_1$\end{tabular}}}}%
    \put(0.32888346,0.16927158){\makebox(0,0)[t]{\lineheight{1.25}\smash{\begin{tabular}[t]{c}$l_0$\end{tabular}}}}%
  \end{picture}%
\endgroup%

  \caption{
    The maps $\xi$ and $f^{-1} \circ \rho$ coincide on the red region,
    whereas $\xi$ and $g^{-1} \circ \rho \circ \alpha_a$
    agree on the blue region.
  }
  \label{fig:xi}
\end{figure}
In particular we have
\begin{equation}
  \label{eq:xiRed}
  \left(
    \mathcal{H}^{\bullet} \circ
    \xi
  \right)(u)
  =
  \left(
    \mathcal{H}^{\bullet} \circ
    f^{-1} \circ
    \rho
  \right)(u)
  = F(u)
\end{equation}
for any point $u$ contained in the red region.
Furthermore, $\xi$ and $g^{-1} \circ \rho \circ \alpha_a$ agree,
when restricted to the region shaded in blue.
Thus, if $T^{-1}(w)$ is contained in the blue region for some $w \in \M$,
then
\begin{equation}
  \label{eq:xiBlue}
  \begin{split}
    \left(
      \mathcal{H}^{\bullet-1} \circ
      \xi \circ
      T^{-1}
    \right)(w)
    &=
    \left(
      \mathcal{H}^{\bullet-1} \circ
      g^{-1} \circ
      \rho \circ
      \alpha_a \circ
      T^{-1}
    \right)(w)
    \\
    &=
    (\Sigma \circ G \circ T^{-1} \circ \alpha_a)(w)
    \\
    &=
    (G \circ \alpha_a)(w)
    .
  \end{split}
\end{equation}
Here the second equality follows from the fact
that $T$ is a
\href{
  https://en.wikipedia.org/wiki/Center_(group_theory)
}{central}
automorphism
and the third from $G$ being strictly stable.
By \cref{prp:rho}.(3-4) the map $\xi$ preserves the joins and meets
of any axis-aligned rectangle contained in $D \cup E$.
Thus, any axis-aligned rectangle in $D \cup E$ gives rise
to some relative Mayer--Vietoris sequence of subspaces of $X$.
In particular, if $v \in D \cap T(E)$, then the axis-aligned rectangle
shown in \cref{fig:xi} gives rise to a Mayer--Vietoris sequence
with the differential
\begin{equation*}
  \delta \colon
  \left(
    \mathcal{H}^{\bullet-1} \circ
    \xi \circ
    T^{-1}
  \right)(v)
  \rightarrow
  \left(
    \mathcal{H}^{\bullet} \circ
    \xi
  \right)(v)
  .
\end{equation*}
By combining this with the equations \eqref{eq:xiRed} and \eqref{eq:xiBlue}
and passing to the opposite category we define
\begin{equation*}
  \nu_v := \delta^\circ \colon F(v) \rightarrow (G \circ \alpha_a)(v)
  .
\end{equation*}
The naturality of
$\nu \colon F |_{D \cap T(E)} \rightarrow (G \circ \alpha_a) |_{D \cap T(E)}$
follows from the naturality of the Mayer--Vietoris sequence.

\begin{figure}[t]
  \centering
\begingroup%
  \makeatletter%
  \providecommand\color[2][]{%
    \errmessage{(Inkscape) Color is used for the text in Inkscape, but the package 'color.sty' is not loaded}%
    \renewcommand\color[2][]{}%
  }%
  \providecommand\transparent[1]{%
    \errmessage{(Inkscape) Transparency is used (non-zero) for the text in Inkscape, but the package 'transparent.sty' is not loaded}%
    \renewcommand\transparent[1]{}%
  }%
  \providecommand\rotatebox[2]{#2}%
  \newcommand*\fsize{\dimexpr\f@size pt\relax}%
  \newcommand*\lineheight[1]{\fontsize{\fsize}{#1\fsize}\selectfont}%
  \ifx\svgwidth\undefined%
    \setlength{\unitlength}{326.16277313bp}%
    \ifx\svgscale\undefined%
      \relax%
    \else%
      \setlength{\unitlength}{\unitlength * \real{\svgscale}}%
    \fi%
  \else%
    \setlength{\unitlength}{\svgwidth}%
  \fi%
  \global\let\svgwidth\undefined%
  \global\let\svgscale\undefined%
  \makeatother%
  \begin{picture}(1,0.50588238)%
    \lineheight{1}%
    \setlength\tabcolsep{0pt}%
    \put(0,0){\includegraphics[width=\unitlength,page=1]{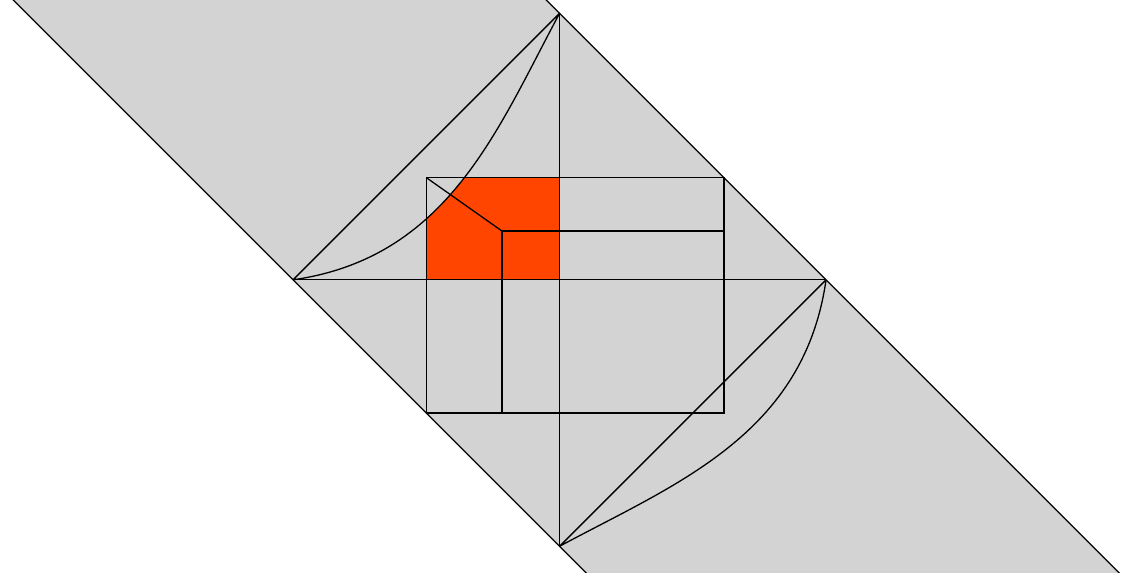}}%
    \put(0.45882359,0.27843145){\makebox(0,0)[t]{\lineheight{1.25}\smash{\begin{tabular}[t]{c}$u$\end{tabular}}}}%
    \put(0.70588229,0.1282352){\makebox(0,0)[t]{\lineheight{1.25}\smash{\begin{tabular}[t]{c}$T^{-1}(v)$\end{tabular}}}}%
    \put(0.3921568,0.35607842){\makebox(0,0)[t]{\lineheight{1.25}\smash{\begin{tabular}[t]{c}$v$\end{tabular}}}}%
    \put(0,0){\includegraphics[width=\unitlength,page=2]{regions-stability-prop1.pdf}}%
    \put(0.63652053,0.39249919){\makebox(0,0)[t]{\lineheight{1.25}\smash{\begin{tabular}[t]{c}$l_1$\end{tabular}}}}%
    \put(0.31097648,0.15569013){\makebox(0,0)[t]{\lineheight{1.25}\smash{\begin{tabular}[t]{c}$l_0$\end{tabular}}}}%
  \end{picture}%
\endgroup%

  \caption{
    The points $u$ and $v$ with corresponding axis-aligned rectangles
    in $D \cup E$.
  }
  \label{fig:glueingProp1}
\end{figure}

Next we show that the natural transformations $\mu$ and $\nu$
satisfy property (1) from \cref{lem:natExt}.
Here it suffices to show that the solid square in
\begin{equation}
  \label{eq:stabProp1}
  \begin{tikzcd}[column sep=large, row sep=large]
    F(v)
    \arrow[r, "\nu_v = \delta^{\circ}"]
    &
    (G \circ \alpha_a)(v)
    \\
    F(u)
    \arrow[u, "F(u \preceq v)"]
    \arrow[r, "\eta_u"']
    \arrow[ru, "(\delta')^{\circ}" description, dashed]
    &
    (G \circ \alpha_a)(u)
    \arrow[u, "(G \circ \alpha_a)(u \preceq v)"']
  \end{tikzcd}
\end{equation}
commutes for all $u$ contained in the region shaded in red
in \cref{fig:glueingProp1}.
For such a point $u$, we construct a linear map
$(\delta')^{\circ}$ as indicated with the dashed arrow in
\eqref{eq:stabProp1} 
and then we show that both triangles commute.
To this end,
we consider the axis-aligned rectangle in \cref{fig:glueingProp1}
with $u$ a vertex.
As $\xi$ preserves the corresponding join $u$
and the meet $T^{-1}(v)$, this
gives rise to a Mayer--Vietoris sequence with differential
\begin{equation*}
  \delta' \colon
  \left(
    \mathcal{H}^{\bullet-1} \circ
    \xi \circ
    T^{-1}
  \right)(v) \rightarrow
  \left(
    \mathcal{H}^{\bullet} \circ
    \xi
  \right)(u)
  .
\end{equation*}
By equations \eqref{eq:xiRed} and \eqref{eq:xiBlue}
the domain and codomain of $(\delta')^{\circ}$ match up
with the dashed arrow in \eqref{eq:stabProp1}.
Now the upper triangle commutes by the naturality
of the Mayer--Vietoris sequence.
Moreover, we have pointwise
$\xi \subseteq g^{-1} \circ \rho \circ \alpha_a$
and thus, using the naturality of the Mayer--Vietoris sequence once again,
we obtain
the commutativity of the lower triangle in \eqref{eq:stabProp1}.

\begin{figure}[t]
  \centering
\begingroup%
  \makeatletter%
  \providecommand\color[2][]{%
    \errmessage{(Inkscape) Color is used for the text in Inkscape, but the package 'color.sty' is not loaded}%
    \renewcommand\color[2][]{}%
  }%
  \providecommand\transparent[1]{%
    \errmessage{(Inkscape) Transparency is used (non-zero) for the text in Inkscape, but the package 'transparent.sty' is not loaded}%
    \renewcommand\transparent[1]{}%
  }%
  \providecommand\rotatebox[2]{#2}%
  \newcommand*\fsize{\dimexpr\f@size pt\relax}%
  \newcommand*\lineheight[1]{\fontsize{\fsize}{#1\fsize}\selectfont}%
  \ifx\svgwidth\undefined%
    \setlength{\unitlength}{326.16277313bp}%
    \ifx\svgscale\undefined%
      \relax%
    \else%
      \setlength{\unitlength}{\unitlength * \real{\svgscale}}%
    \fi%
  \else%
    \setlength{\unitlength}{\svgwidth}%
  \fi%
  \global\let\svgwidth\undefined%
  \global\let\svgscale\undefined%
  \makeatother%
  \begin{picture}(1,0.50588238)%
    \lineheight{1}%
    \setlength\tabcolsep{0pt}%
    \put(0,0){\includegraphics[width=\unitlength,page=1]{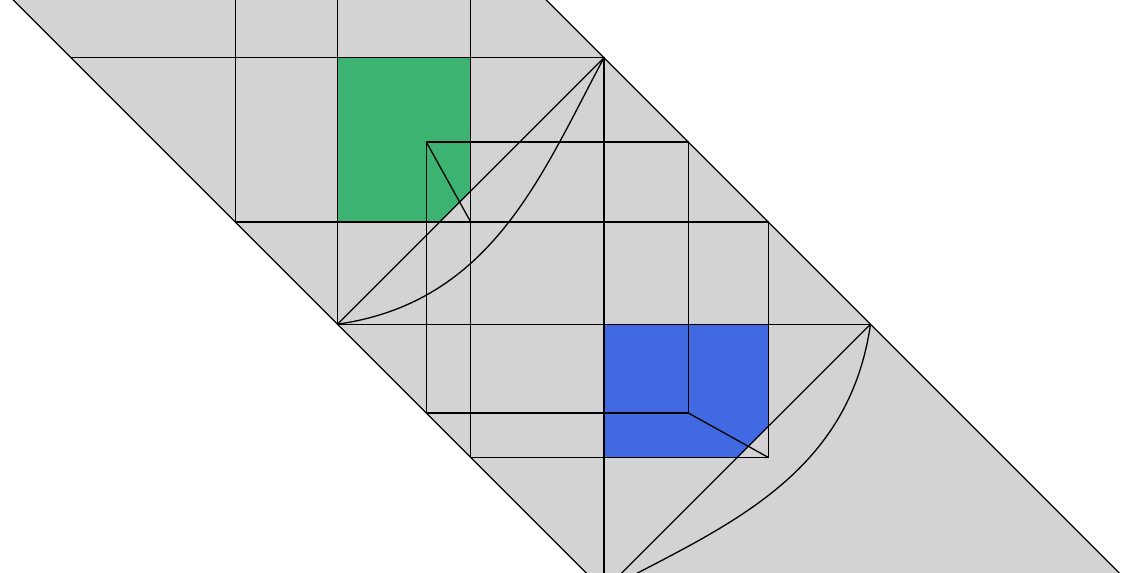}}%
    \put(0.602353,0.15294112){\makebox(0,0)[t]{\lineheight{1.25}\smash{\begin{tabular}[t]{c}$T^{-1}(w)$\end{tabular}}}}%
    \put(0.36078428,0.38745095){\makebox(0,0)[t]{\lineheight{1.25}\smash{\begin{tabular}[t]{c}$w$\end{tabular}}}}%
    \put(0.74509806,0.08901966){\makebox(0,0)[t]{\lineheight{1.25}\smash{\begin{tabular}[t]{c}$T^{-1}(v)$\end{tabular}}}}%
    \put(0.43137257,0.31686265){\makebox(0,0)[t]{\lineheight{1.25}\smash{\begin{tabular}[t]{c}$v$\end{tabular}}}}%
    \put(0,0){\includegraphics[width=\unitlength,page=2]{regions-stability-prop2.pdf}}%
    \put(0.67573607,0.35328365){\makebox(0,0)[t]{\lineheight{1.25}\smash{\begin{tabular}[t]{c}$l_1$\end{tabular}}}}%
    \put(0.31691646,0.14975015){\makebox(0,0)[t]{\lineheight{1.25}\smash{\begin{tabular}[t]{c}$l_0$\end{tabular}}}}%
  \end{picture}%
\endgroup%

  \caption{
    The points $v$ and $w$ with corresponding axis-aligned rectangles
    in $D \cup E$.
  }
  \label{fig:glueingProp2}
\end{figure}

Next we show that $\mu$ and $\nu$
satisfy property (2) from \cref{lem:natExt}.
Here it suffices to show that the solid square in
\begin{equation}
  \label{eq:stabProp2}
  \begin{tikzcd}[column sep=5.5em, row sep=large]
    F(w)
    \arrow[r, "(\Sigma \circ \eta \circ T^{-1})_w"]
    &
    (G \circ \alpha_a)(w)
    \\
    F(v)
    \arrow[u, "F(v \preceq w)"]
    \arrow[r, "\nu_v = \delta^{\circ}"']
    \arrow[ru, "(\delta'')^{\circ}" description, dashed]
    &
    (G \circ \alpha_a)(v)
    \arrow[u, "(G \circ \alpha_a)(v \preceq w)"']
  \end{tikzcd}
\end{equation}
commutes for all $w$ contained in the region shaded in green
in \cref{fig:glueingProp2}.
As with property (1) we provide the dashed arrow
and then we show that both triangles commute.
To this end, we consider the axis-aligned rectangle
shown in \cref{fig:glueingProp2}
with $v$ and $T^{-1}(w)$ vertices.
As $\xi$ preserver the join $v$ and the meet $T^{-1}(w)$, this
gives rise to a Mayer--Vietoris sequence with differential
\begin{equation*}
  \delta'' \colon
  \left(
    \mathcal{H}^{\bullet-1} \circ
    \xi \circ
    T^{-1}
  \right)(w) \rightarrow
  \left(
    \mathcal{H}^{\bullet} \circ
    \xi
  \right)(v)
  .
\end{equation*}
By equations \eqref{eq:xiRed} and \eqref{eq:xiBlue}
the domain and the codomain of $(\delta'')^{\circ}$
match up with the dashed arrow in \eqref{eq:stabProp2}.
The lower triangle commutes
by the naturality of the Mayer--Vietoris sequence.
Moreover, we have pointwise
$f^{-1} \circ \rho \subseteq \xi$
and thus, using the naturality of the Mayer--Vietoris sequence once again,
we obtain
the commutativity of the upper triangle in \eqref{eq:stabProp2}.

With the assumptions of \cref{lem:natExt} satisfied
we may extend $\eta$ and $\nu$ to a unique strictly stable
natural transformation
$\varphi \colon F \rightarrow G \circ \alpha_a$.
We define
$\vec{h}(f, g) \colon h(g) \circ \alpha_a \rightarrow h(f)$
as
$\vec{h}(f, g) := \op{ev}^0 \circ \varphi^\circ$
by whiskering with the evaluation at $0$.

With the next proposition we show that $\vec{h}$ and $h$
preserves triangles in some sense.
To this end, let
$f_1, f_2, f_3 \colon X \rightarrow \R$
be functions with
$\mathbf{d}(f_1, f_2), \mathbf{d}(f_2, f_3) \in \R^{\circ} \times \R$.
Moreover, let
\begin{equation*}
  a := \mathbf{d}(f_1, f_2)
  ,
  ~~
  b := \mathbf{d}(f_2, f_3)
  , ~ \text{and}
  ~~
  c := \mathbf{d}(f_1, f_3)
  ,
\end{equation*}
then we have $c \preceq a+b$ by the triangle inequality \eqref{eq:triaIneq}.

\begin{prp}[Compatibility with Composition]
  \label{prp:compatCompos}
  The following square of functors and natural transformations
  \begin{equation}
    \label{eq:compatCompos}
    \begin{tikzcd}[column sep=6em, row sep=large]
      h(f_3) \circ \alpha_{a+b}
      \arrow[r, "{\vec{h}(f_2, f_3) \circ \alpha_a}"]
      \arrow[d, "{h(f_3) \circ \alpha_{c \preceq a+b}}"']
      &
      h(f_2) \circ \alpha_a
      \arrow[d,"{\vec{h}(f_1, f_2)}"]
      \\
      h(f_3) \circ \alpha_c
      \arrow[r, "{\vec{h}(f_1, f_3)}"']
      &
      h(f_1)
    \end{tikzcd}
  \end{equation}
  commutes.
\end{prp}

\begin{proof}
  Let ${h^{\#}(f_1) \colon \M^{\circ} \rightarrow \VectK^{\Z}}$
  be the transform of ${h(f_1) \colon \M^{\circ} \rightarrow \VectK}$
  under the $2$-adjunction from \cref{lem:2adj}.
  We define $h^{\#}(f_2)$, $h^{\#}(f_3)$,
  $\vec{h}^{\#}(f_1, f_2)$, $\vec{h}^{\#}(f_2, f_3)$,
  and $\vec{h}^{\#}(f_1, f_3)$ analogously.
  Then the commutativity of \eqref{eq:compatCompos} is equivalent
  to the commutativity of
  \begin{equation*}
    \begin{tikzcd}[column sep=6em, row sep=large]
      h^{\#}(f_3) \circ \alpha_{a+b}
      \arrow[r, "{\vec{h}^{\#}(f_2, f_3) \circ \alpha_a}"]
      \arrow[d, "{h^{\#}(f_3) \circ \alpha_{c \preceq a+b}}"']
      &
      h^{\#}(f_2) \circ \alpha_a
      \arrow[d,"{\vec{h}^{\#}(f_1, f_2)}"]
      \\
      h^{\#}(f_3) \circ \alpha_c
      \arrow[r, "{\vec{h}^{\#}(f_1, f_3)}"']
      &
      h^{\#}(f_1)
    \end{tikzcd}    
  \end{equation*}
  by \cref{lem:2adj}.
  As all functors and natural transformations in this square
  are strictly stable,
  it suffices to check the commutativity for any point $v \in D$.
  To this end,
  we partition $D$ into the three regions
  \begin{align*}
    D_1 & := D \cap \alpha_{(-a-b)} (D),
    \\
    D_2 & := \alpha_{(-a)} (D) \setminus \alpha_{(-a-b)} (D),
          ~~ \text{and}
    \\
    D_3 & := D \setminus \alpha_{(-a)} (D)
          .
  \end{align*}
  For $v \in D_1$ the commutativity follows
  from the functoriality of the cohomology theory $\mathcal{H}^{\bullet}$.
  For $v \in D_2 \cup D_3$
  we consider the monotone map
  \begin{equation*}
    \xi_0 \colon
    u \mapsto
    ((f_3^{-1} \circ \rho_1 \circ \alpha_{a+b})(u), (f_1^{-1} \circ \rho_0)(u))
    ,
  \end{equation*}
  where $\alpha_{a+b} := \alpha(a+b)$.
  Then the axis-aligned rectangle shown in \cref{fig:xi}
  gives rise to a Mayer--Vietoris sequence with differential
  \begin{equation*}
    \delta \colon
    \left(
      \mathcal{H}^{\bullet-1} \circ
      \xi_0 \circ
      T^{-1}
    \right)(v)
    \rightarrow
    \left(
      \mathcal{H}^{\bullet} \circ
      \xi_0
    \right)(v)
    .
  \end{equation*}
  By equations \eqref{eq:xiRed} and \eqref{eq:xiBlue}
  we have
  \begin{equation}
    \label{eq:xi0Compat}
    \left(
      \mathcal{H}^{\bullet} \circ
      \xi_0
    \right)(v) =
    h^{\#}(f_1)(v)
    \quad \text{and} \quad
    \left(
      \mathcal{H}^{\bullet-1} \circ
      \xi_0 \circ
      T^{-1}
    \right)(v) =
    \left(h^{\#}(f_3) \circ \alpha_{a+b}\right)(v)
    .
  \end{equation}
  Moreover, by the naturality of the Mayer--Vietoris sequence,
  the triangle
  \begin{equation*}
    \begin{tikzcd}[column sep=6em, row sep=large]
      \left(h^{\#}(f_3) \circ \alpha_{a+b}\right)(v)
      \arrow[d, "{\left(h^{\#}(f_3) \circ \alpha_{c \preceq a+b}\right)_v}"']
      \arrow[dr, "\delta"]
      \\
      \left(h^{\#}(f_3) \circ \alpha_c\right)(v)
      \arrow[r, "{\vec{h}^{\#}(f_1, f_3)_v}"']
      &
      h^{\#}(f_1)(v)
    \end{tikzcd}    
  \end{equation*}
  commutes.
  It remains to show that the triangle
  \begin{equation}
    \label{eq:compatComposTria}
    \begin{tikzcd}[column sep=7em, row sep=large]
      \left(h^{\#}(f_3) \circ \alpha_{a+b}\right)(v)
      \arrow[r, "{\left(\vec{h}^{\#}(f_2, f_3) \circ \alpha_a\right)_v}"]
      \arrow[rd, "\delta"']
      &
      \left(h^{\#}(f_2) \circ \alpha_a\right)(v)
      \arrow[d,"{\vec{h}^{\#}(f_1, f_2)_v}"]
      \\
      &
      h^{\#}(f_1)(v)
    \end{tikzcd}    
  \end{equation}
  commutes.
  For $v \in D_3$ we consider the monotone map
  \begin{equation*}
    \xi_1 \colon
    u \mapsto
    (
    (f_2^{-1} \circ \rho_1 \circ \alpha_a)(u),
    (f_1^{-1} \circ \rho_0               )(u)
    )
    ,
  \end{equation*}
  which we used for the construction of
  $\vec{h}^{\#} (f_1, f_2)$.
  By equations \eqref{eq:xiRed} and \eqref{eq:xiBlue} we have
  \begin{equation*}
    \left(
      \mathcal{H}^{\bullet} \circ
      \xi_1
    \right)(v)
    =
    h^{\#}(f_1)(v)
    \quad \text{and} \quad
    \left(
      \mathcal{H}^{\bullet-1} \circ
      \xi_1 \circ
      T^{-1}
    \right)(v)
    =
    \left(h^{\#}(f_2) \circ \alpha_a\right)(v)
    .
  \end{equation*}
  In conjunction with \eqref{eq:xi0Compat}
  this allows us to rewrite \eqref{eq:compatComposTria} as
  \begin{equation*}
    \begin{tikzcd}[column sep=7em, row sep=large]
      \left(
        \mathcal{H}^{\bullet-1} \circ
        \xi_0 \circ
        T^{-1}
      \right)(v)
      \arrow[r, "{\left(\vec{h}^{\#}(f_2, f_3) \circ \alpha_a\right)_v}"]
      \arrow[d, "\delta"']
      &
      \left(
        \mathcal{H}^{\bullet-1} \circ
        \xi_1 \circ
        T^{-1}
      \right)(v)
      \arrow[d,"{\vec{h}^{\#}(f_1, f_2)_v}"]
      \\
      \left(
        \mathcal{H}^{\bullet} \circ
        \xi_0
      \right)(v)
      \arrow[r, equal]
      &
      \left(
        \mathcal{H}^{\bullet} \circ
        \xi_1
      \right)(v)
      .
    \end{tikzcd}
  \end{equation*}
  Moreover, we have pointwise
  $\xi_1 \subseteq \xi_0$,
  and thus the commutativity of this square follows
  from the naturality of the Mayer--Vietoris sequence.
  For $v \in D_2$ we consider the monotone map
  \begin{equation*}
    \xi_2 \colon
    u \mapsto
    (
    (f_3^{-1} \circ \rho_1 \circ \alpha_b)(u),
    (f_2^{-1} \circ \rho_0               )(u)
    )
    ,
  \end{equation*}
  which we used for the construction of
  $\vec{h}^{\#} (f_2, f_3)$.
  By equations \eqref{eq:xiRed} and \eqref{eq:xiBlue} we have
  \begin{align*}
    \left(\mathcal{H}^{\bullet} \circ \xi_2 \circ \alpha_a\right)(v)
    & =
      \left(h^{\#}(f_2) \circ \alpha_a\right)(v)
    \\
    \text{and}
    \quad
    \left(
    \mathcal{H}^{\bullet-1} \circ
    \xi_2 \circ
    \alpha_a
    \circ
    T^{-1}
    \right)(v)
    =
    \left(
      \mathcal{H}^{\bullet-1} \circ
      \xi_2 \circ
      T^{-1} \circ
      \alpha_a
    \right)(v)
    & =
      \left(h^{\#}(f_3) \circ \alpha_{a+b}\right)(v)
      .
  \end{align*}
  Here the first equality on the second line
  follows from the fact that $T$ is a
  \href{
    https://en.wikipedia.org/wiki/Center_(group_theory)
  }{central}
  automorphism.
  In conjunction with \eqref{eq:xi0Compat}
  this allows us to rewrite \eqref{eq:compatComposTria} as
  \begin{equation*}
    \begin{tikzcd}[column sep=7em, row sep=large]
      \left(
        \mathcal{H}^{\bullet-1} \circ
        \xi_2 \circ
        \alpha_a
        \circ
        T^{-1}
      \right)(v)
      \arrow[d, equal]
      \arrow[r, "{\left(\vec{h}^{\#}(f_2, f_3) \circ \alpha_a\right)_v}"]
      &
      \left(\mathcal{H}^{\bullet} \circ \xi_2 \circ \alpha_a\right)(v)
      \arrow[d,"{\vec{h}^{\#}(f_1, f_2)_v}"]
      \\
      \left(
        \mathcal{H}^{\bullet-1} \circ
        \xi_0 \circ
        T^{-1}
      \right)(v)
      \arrow[r, "\delta"']
      &
      \left(
        \mathcal{H}^{\bullet} \circ
        \xi_0
      \right)(v)
      .
    \end{tikzcd}    
  \end{equation*}
  Moreover, we have pointwise
  $\xi_0 \subseteq \xi_2 \circ \alpha_a$,
  and thus the commutativity of this square
  follows from the naturality of the Mayer--Vietoris sequence.
\end{proof}

Next we show how this proposition implies
\emph{stability} for RISC
in the sense of \cite{MR3413628}.
The group homomorphism
\[\Omega \colon
  \R \rightarrow \op{Aut}(\M),\,
  \delta \mapsto \Omega_\delta := \alpha_{(-\delta, \delta)},
\]
describes a \emph{(super)linear family} on $\M$
in the sense of \cite[Section 2.5]{MR3413628}.

\begin{dfn}[$\delta$-Interleaving, \cite{MR3413628}]
  \label{dfn:interleaving}
  Let $\delta \geq 0$ and let $F, G \colon \M^{\circ} \rightarrow \VectK$
  be contravariant functors on $\M$.
  Then a \emph{$\delta$-interleaving} of $F$ and $G$
  is a pair of natural transformations
  \[\varphi \colon G \circ \Omega_\delta \rightarrow F
    \quad \text{and} \quad
    \psi \colon F \circ \Omega_\delta \rightarrow G
  \]
  such that both triangles in the diagram
  \begin{equation}
    \label{eq:interleaving}
    \begin{tikzcd}[row sep=large, column sep=10ex]
      G \circ \Omega_{2 \delta}
      \arrow[dd, bend right=55,
      "G \circ \Omega_{0 \preceq 2 \delta}"']
      &
      F \circ \Omega_{2 \delta}
      \arrow[dl, "\psi \circ \Omega_\delta"' near end]
      \arrow[dd, bend left=55,
      "F \circ \Omega_{0 \preceq 2\delta}"]
      \\
      G \circ \Omega_\delta
      &
      F \circ \Omega_\delta
      \arrow[from=ul, crossing over, "\varphi \circ \Omega_\delta" near end]
      \arrow[dl, "\psi"' near end]
      \\
      G
      &
      F
      \arrow[from=ul, crossing over, "\varphi" near end]
    \end{tikzcd}    
  \end{equation}
  commute.
  Moreover, we say that $F$ and $G$ are \emph{$\delta$-interleaved}
  if there is a $\delta$-interleaving.
\end{dfn}

\begin{thm}[Stability]
  \label{thm:stability}
  Let $\delta \geq 0$ and let
  $f, g \colon X \rightarrow \R$ be continuous functions
  with $|f(x) - g(x)| \leq \delta$ for all $x \in X$.
  Then $h(f)$ and $h(g)$ are $\delta$-interleaved.
  More specifically, the interleaving natural transformations
  can be given as the compositions
  \begin{align*}
    \varphi & \colon
              h(g) \circ \Omega_{\delta}
              \xrightarrow{h(g) \circ \alpha_{a \preceq (-\delta, \delta)}}
              h(g) \circ \alpha_a
              \xrightarrow{\vec{h}(f, g)}
              h(f)
    \\
    \text{and} \quad
    \psi & \colon
           h(f) \circ \Omega_{\delta}
           \xrightarrow{h(f) \circ \alpha_{a' \preceq (-\delta, \delta)}}
           h(f) \circ \alpha_{a'}
           \xrightarrow{\vec{h}(g, f)}
           h(g)
           ,
  \end{align*}
  where $(x, y) := a := \mathbf{d}(f, g)$
  and $a' := (-y, -x) = \mathbf{d}(g, f)$.
\end{thm}

\begin{proof}
  By \cref{prp:compatCompos} we have the commutative diagram
  \begin{equation*}
    \begin{tikzcd}[column sep=9em, row sep=6em]
      h(f) \circ \Omega_{2 \delta}
      \arrow[r, equal]
      \arrow[d,
      "{h(f) \circ \alpha_{(-\delta, \delta) + a' \preceq (-2\delta, 2\delta)}}"]
      \arrow[dd, bend right=55, "\psi \circ \Omega_{\delta}"']
      &
      h(f) \circ \Omega_{2 \delta}
      \arrow[d, "{h(f) \circ \alpha_{a+a' \preceq (-2\delta, 2\delta)}}" description]
      \arrow[dr, "h(f) \circ \Omega_{0 \preceq 2\delta}"]
      \\
      h(f) \circ \alpha_{(-\delta, \delta) + a'}
      \arrow[r, "{h(f) \circ \alpha_{a+a' \preceq (-\delta, \delta) + a'}}"]
      \arrow[d, "{\vec{h}(g, f) \circ \Omega_{\delta}}"]
      &
      h(f) \circ \alpha_{a + a'}
      \arrow[r, "h(f) \circ \alpha_{0 \preceq a + a'}"]
      \arrow[d, "{\vec{h}(g, f) \circ \alpha_a}" description]
      &
      h(f)
      \arrow[d, equal]
      \\
      h(g) \circ \Omega_{\delta}
      \arrow[r, "{h(g) \circ \alpha_{a \preceq (-\delta, \delta)}}"']
      \arrow[rr, bend right=25, "\varphi"']
      &
      h(g) \circ \alpha_a
      \arrow[r, "{\vec{h}(f, g)}"']
      &
      h(f)
      ,      
    \end{tikzcd}
  \end{equation*}
  which yields the commutativity
  of the triangle on the right hand side in \eqref{eq:interleaving}.
  By symmetry, the proof that the left triangle commutes
  is completely analogous.
\end{proof}

\begin{figure}[t]
  \centering
  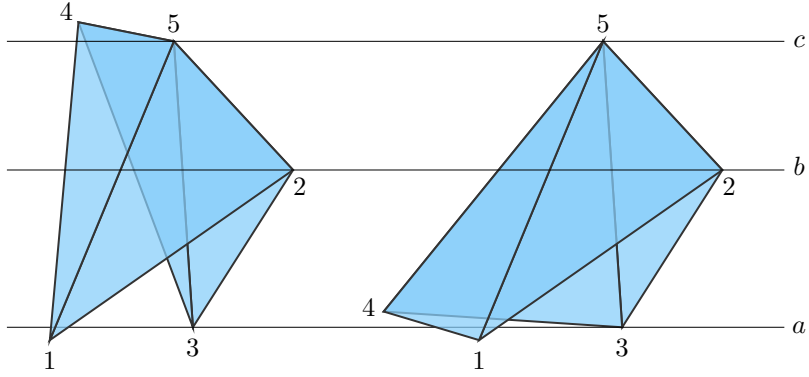
  \caption{The function $f$ depicted as a height function
    on the left hand side
    and the height function $g'$ on the right,
    see \cref{exm:hood}.
  }
  \label{fig:hood}
\end{figure}

\begin{figure}[t]
  \centering
  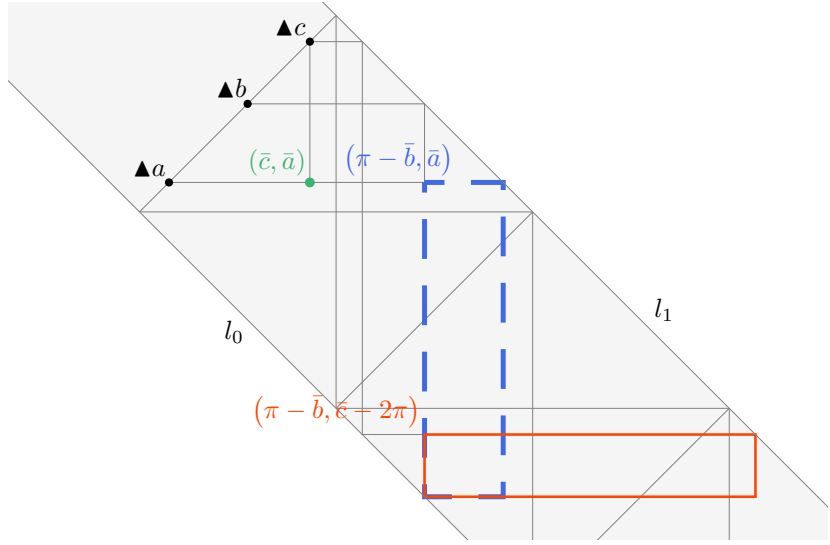
  \caption{
    The indecomposables
    $B_{\left(\pi-\bar{b}, \bar{a}\right)}$
    and
    $B_{\left(\pi-\bar{b}, \bar{c}-2\pi\right)}$.
  }
  \label{fig:hoodDiagram}
\end{figure}

\begin{exm}
  \label{exm:hood}
  Let $C_4$ be the cyclic graph on four vertices
  $\{1, 2, 3, 4\}$.
  Moreover,
  let $5 * C_4$ be the abstract simplicial cone over $C_4$
  with $5$ as the tip of the cone
  and let $X := |5 * C_4|$ be the geometric realization of $5 * C_4$.
  Furthermore,
  let $a < b < c$ and let
  $f, g' \colon X \rightarrow \R$
  be the unique simplexwise linear functions with
  \begin{align*}
    f(1) = f(3) & = a = g'(1) = g'(3) = g'(4) \\
    f(2) & = b = g'(2), \\
    \text{and} \quad
    f(4) = f(5) & = c = g'(5),
  \end{align*}
  which are also depicted as height functions in
  \cref{fig:hood}.
  Now let $\delta := \frac{1}{2} (c - a)$
  and let $g := \delta + g'$.
  Then we have $\|f - g\|_{\infty} = \delta$
  and thus there is a $\delta$-interleaving with
  $\varphi := \vec{h}(f, g)$ as one of the two
  interleaving natural transformations
  by \cref{thm:stability}.
  Moreover,
  we have $h(g) \circ \alpha_{(\delta, \delta)} = h(g')$
  and thus
  \[h(g) \circ \Omega_{\delta} =
    h(g) \circ \alpha_{(\delta, \delta)} \circ \alpha_{(-2\delta, 0)} =
    h(g') \circ \alpha_{(-2\delta, 0)}
  \]
  and $\varphi = \vec{h}(f, g) = \vec{h}(f, g')$.
  Now let
  \begin{align*}
    \bar{a} & := \arctan a, \\
    \bar{b} & := \arctan b, \\
    \bar{c} & := \arctan c, \\
    \text{and} \quad
    \tilde{c} & := \arctan (2c - a).
  \end{align*}
  Then we have
  \begin{align*}
    \op{Dgm}(f)
    =
    \mathbf{1}_{
    \left\{(\bar{c}, \bar{a}), \left(\pi-\bar{b}, \bar{a}\right)\right\}
    }
    \quad \text{and} \quad
    \op{Dgm}(g')
    =
    \mathbf{1}_{
    \left\{(\bar{c}, \bar{a}), \left(\pi-\bar{b}, \bar{c}-2\pi\right)\right\}
    }
    ,
  \end{align*}
  where
  ${\mathbf{1}_{
      \left\{(\bar{c}, \bar{a}), \left(\pi-\bar{b}, \bar{a}\right)\right\}
    } \colon \op{int} \M \rightarrow \N_0}$
  is the indicator function of the subset
  ${\left\{(\bar{c}, \bar{a}), \left(\pi-\bar{b}, \bar{a}\right)\right\} \subset
    \op{int} \M}$;
  see also \cref{fig:hoodDiagram}.
  As
  \begin{equation*}
    \alpha_{(-2 \delta, 0)} (\tilde{c}, \bar{a}) = (\bar{c}, \bar{a})
    \quad \text{and} \quad
    \alpha_{(-2 \delta, 0)} \left(\pi-\bar{b}, \bar{c}-2\pi\right) =
    \left(\pi-\bar{b}, \bar{c}-2\pi\right)
    ,
  \end{equation*}
  we have
  $h(g') \circ \alpha_{(-2\delta, 0)} \cong
  B_{(\tilde{c}, \bar{a})} \oplus
  B_{\left(\pi-\bar{b}, \bar{c}-2\pi\right)}$
  and thus there is a commutative diagram
  \begin{equation*}
    \begin{tikzcd}[row sep=4em]
      B_{(\tilde{c}, \bar{a})} \oplus
      B_{\left(\pi-\bar{b}, \bar{c}-2\pi\right)}
      \arrow[r, "\sim"]
      \arrow[
      d,
      "B_{(\tilde{c}, \bar{a}) \preceq (\bar{c}, \bar{a})} \oplus
      B_{\left(\pi-\bar{b}, \bar{c}-2\pi\right) \preceq \left(\pi-\bar{b}, \bar{a}\right)}"
      description]
      &
      h(g') \circ \alpha_{(-2\delta, 0)}
      \arrow[r, equal]
      \arrow[d, "{\vec{h}(f, g')}"]
      &
      h(g) \circ \Omega_{\delta}
      \arrow[d, "\varphi"]
      \\
      B_{(\bar{c}, \bar{a})} \oplus
      B_{\left(\pi-\bar{b}, \bar{a}\right)}
      \arrow[r, "\sim"]
      &
      h(f)
      \arrow[r, equal]
      &
      h(f)
      .
    \end{tikzcd}
  \end{equation*}
  We consider the natural transformation
  \begin{equation*}
    \eta := 
    B_{\left(\pi-\bar{b}, \bar{c}-2\pi\right) \preceq \left(\pi-\bar{b}, \bar{a}\right)}
    \colon
    B_{\left(\pi-\bar{b}, \bar{c}-2\pi\right)}
    \rightarrow
    B_{\left(\pi-\bar{b}, \bar{a}\right)}
    .
  \end{equation*}
  As is shown in \cref{fig:hoodDiagram} the supports of
  $B_{(\pi-\bar{b}, \bar{a})}$ and
  $B_{(\pi-\bar{b}, \bar{c}-2\pi)}$
  have a non-empty intersection.
  For any $u$ in this intersection,
  $\eta_u$
  maps
  $B_{(\pi-\bar{b}, \bar{c}-2\pi)} (u) = K$
  identically onto
  $B_{(\pi-\bar{b}, \bar{a})} (u) = K$.
  Moreover, \cref{fig:hoodDiagram} also shows that the intersection
  of supports is disjoint from the regions of $\M$
  covered by extended persistence
  or the south face of the Mayer--Vietoris pyramid,
  see also \cref{sec:connExt}.
  In particular, the corresponding interleaving homomorphisms
  of extended persistence
  \cite[Section 6]{MR3201246}
  and of Mayer--Vietoris systems
  \cite[Section 2.3]{2019arXiv190709759B}
  do not detect the non-trivial homomorphism
  between the corresponding summands.
  Nevertheless, derived level sets persistence
  and derived extended (or even derived sublevel sets) persistence
  will recognize the corresponding homomorphism as well.
  This is consistent with
  \cite[Remark 4.9]{2019arXiv190709759B},
  where the corresponding homomorphism of derived level sets persistence
  is provided as a counter-example to the canonical functor
  to Mayer--Vietoris systems being faithful.
\end{exm}

Next we describe how $\vec{h}$ is compatible with precomposition.
To this end,
let $f, g \colon Y \rightarrow \R$ be continuous functions,
let $\varphi \colon X \rightarrow Y$ be a continuous map,
let $a := \mathbf{d}(f, g)$,
and let $b := \mathbf{d}(f \circ \varphi, g \circ \varphi)$,
then we have $b \preceq a$.

\begin{prp}[Compatibility with Precomposition]
  \label{prp:compatPrecomp}
  We have the commutative diagram
  \begin{equation*}
    \begin{tikzcd}[row sep=3em, column sep=5em]
      h(g) \circ \alpha_a
      \arrow[r, "h(\varphi) \circ \alpha_a"]
      \arrow[dd, "{\vec{h}(f, g)}"']
      &
      h(g \circ \varphi) \circ \alpha_a
      \arrow[d, "h(g \circ \varphi) \circ \alpha_{b \preceq a}"]
      \\
      &
      h(g \circ \varphi) \circ \alpha_b
      \arrow[d, "{\vec{h}(f \circ \varphi, g \circ \varphi)}"]
      \\
      h(f)
      \arrow[r, "h(\varphi)"']
      &
      h(f \circ \varphi)
      .
    \end{tikzcd}
  \end{equation*}
\end{prp}

\begin{proof}
  Here the horizontal natural transformations
  are induced by inclusions.
  The vertical transformations are induced by inclusions
  for some points, and they are given as differentials
  of Mayer--Vietoris sequences for other points.
  In the former case the commutativity follows
  from the functoriality of $\mathcal{H}^{\bullet}$.
  In the latter case the commutativity follows
  from the naturality of the Mayer--Vietoris sequence.
\end{proof}

Finally we provide a form of homotopy invariance
for $h$ and $\vec{h}$.
To this end, we consider a commutative square
\begin{equation*}
  \begin{tikzcd}[column sep=large, row sep=large]
    X \times [0, 1]
    \arrow[r, "\Gamma"]
    \arrow[d, "\pi_1"']
    &
    Y
    \arrow[d, "g"]
    \\
    X
    \arrow[r, "f"']
    &
    \R
  \end{tikzcd}
\end{equation*}
of spaces
as well as the continuous maps
$\varphi := \Gamma(-, 0)$ and $\psi := \Gamma(-, 1)$
from $X$ to $Y$.
Then we may think of $\varphi$ and $\psi$
as homomorphisms in the category of spaces over the reals $\R$
with domain $f$ and codomain $g$.
Moreover, we may think of $\Gamma$ as a \emph{homotopy}
from $\varphi$ to $\psi$.
The following lemma states
that \emph{homotopic} homomorphisms of spaces over the reals
are identified under $h$.

\begin{lem}
  \label{lem:ordHomotopyInv}
  We have
  $h(\varphi) = h(\psi) \colon h(f) \rightarrow h(g)$.
\end{lem}

As usual, we say that two homomorphisms of spaces over the reals
are \emph{homotopy inverses} of one another
if they are composable both ways and both compositions
are homotopic to the corresponding identities.
Moreover, we say that a homomorphism is
a \emph{homotopy equivalence},
if it has a homotopy inverse.

\begin{cor}
  \label{cor:equivsToIsos}
  The functor $h$ maps homotopy equivalences
  to natural isomorphisms.
\end{cor}

Now from this corollary to \cref{lem:ordHomotopyInv}
we may in turn deduce a generalization of \cref{lem:ordHomotopyInv}.
Suppose we have functions
$f \colon X \rightarrow \R$ and $g \colon Y \rightarrow \R$
as well as a map
${\Gamma \colon X \times [0, 1] \rightarrow Y}$
with $\mathbf{d}(f \circ \pi_1, g \circ \Gamma) \in \R^{\circ} \times \R$.
Morally the following proposition says that $h$ and $\vec{h}$ also
\enquote{identify approximately homotopic maps}.

\begin{prp}[Quantitative Homotopy Invariance]
  \label{prp:quantHomotopyInv}
  For $\varphi := \Gamma(-, 0)$ and $\psi := \Gamma(-, 1)$
  as well as
  \begin{align*}
    a & := \mathbf{d}(f, g \circ \varphi)
        ,
    \\
    b & := \mathbf{d}(f, g \circ \psi)
        ,
    \\
    \text{and}
    \quad
    c & := \mathbf{d}(f \circ \pi_1, g \circ \Gamma)
        ,
  \end{align*}
  we have the commutative square
  \begin{equation}
    \label{eq:homotopicProxy}
    \begin{tikzcd}[column sep=5em, row sep=large]
      h(g)
      \arrow[rr, "h(\psi) \circ \alpha_c"]
      \arrow[dd, "h(\varphi) \circ \alpha_c"']
      & &
      h(g \circ \psi) \circ \alpha_c
      \arrow[d, "h(g \circ \psi) \circ \alpha_{b \preceq c}"]
      \\
      & &
      h(g \circ \psi) \circ \alpha_b
      \arrow[d, "{\vec{h}(f, g \circ \psi)}"]
      \\
      h(g \circ \varphi) \circ \alpha_c
      \arrow[r, "h(g \circ \varphi) \circ \alpha_{a \preceq c}"']
      &
      h(g \circ \varphi) \circ \alpha_a
      \arrow[r, "{\vec{h}(f, g \circ \varphi)}"']
      &
      h(f)
      .
    \end{tikzcd}
  \end{equation}
\end{prp}

\begin{proof}
  We consider the trapezium
  \begin{equation*}
    \begin{tikzcd}[row sep=3em, column sep=5em]
      h(g) \circ \alpha_c
      \arrow[d, "h(\Gamma)"']
      \arrow[rd, "h(\varphi) \circ \alpha_c"]
      \\
      h(g \circ \Gamma) \circ \alpha_c
      \arrow[r, "h(i_0) \circ \alpha_c"']
      \arrow[dd, "{\vec{h}(f \circ \pi_1, g \circ \Gamma)}"']
      &
      h(g \circ \varphi) \circ \alpha_c
      \arrow[d, "h(g \circ \varphi) \circ \alpha_{a \preceq c}"]
      \\
      &
      h(g \circ \varphi) \circ \alpha_a
      \arrow[d, "{\vec{h}(f, g \circ \varphi)}"]
      \\
      h(f \circ \pi_1)
      \arrow[r, "h(i_0)"', "\sim"]
      &
      h(f)
      \arrow[ld, equal]
      \\
      h(f)
      \arrow[u, "h(\pi_1)"', "\sim" {anchor=south, rotate=90}]
    \end{tikzcd}
  \end{equation*}
  of functors and natural transformations.
  Here both triangles commute by the functoriality of $h$.
  The square in the center commutes by \cref{prp:compatPrecomp}.
  Moreover,
  $\pi_1$ and $i_0$ are homotopy inverses of one another,
  hence both are homotopy equivalences.
  By \cref{cor:equivsToIsos}
  the natural transformations $h(\pi_1)$ and $h(i_0)$
  are natural isomorphisms.
  Thus,
  the three natural transformations along the left edge
  of this trapezium
  provide a natural transformation
  \begin{equation*}
    \eta \colon h(g) \circ \alpha_c \rightarrow h(f)
  \end{equation*}
  from
  $h(g) \circ \alpha_c$ to $h(f)$,
  which is equal to the composition of the natural transformations
  along the other edges of the trapezium.
  These natural transformations along the other edges
  are the same as the transformations
  on the left and at the bottom of \eqref{eq:homotopicProxy}.
  Completely analogously we see that $\eta$ is identical
  to the composition of the other two sides of the square
  \eqref{eq:homotopicProxy}.
\end{proof}

\bibliographystyle{alpha}
\bibliography{bib/deheuvels-1955.bib,bib/lawvere-1973.bib,bib/bauer-botnan-fluhr-2021-2.bib,bib/carlsson-2009.bib,bib/bendich-edelsbrunner-morozov-patel-2013.bib,bib/berkouk-ginot-oudot-2019.bib,bib/cohen-steiner-edelsbrunner-harer-2009.bib,bib/tomDieck-2008.bib,bib/may-1999.bib,bib/chazal-deSilva-glisse-oudot-2016.bib,bib/carlsson-deSilva-kalisnik-morozov-2019.bib,bib/barratt-whitehead-1956.bib,bib/kashiwara-schapira-1990.bib,bib/botnan-crawley-boevey-2020.bib,bib/crawley-boevey-2015.bib,bib/bubenik-deSilva-scott-2015.bib,bib/botnan-lebovici-oudot-2020.bib,bib/bubenik-scott-2014.bib,bib/cohen-steiner-2007.bib,bib/curry-2014.bib,bib/kashiwara-2018.bib,bib/ghrist-2008.bib,bib/cochoy-oudot-2020.bib,bib/bubenik-deSilva-scott-2017.bib,bib/scoccola-2020.bib}

\appendix

\section{Stable Functors on the Strip $\M$}
\label{sec:stableFun}

\begin{dfn}
  A \emph{strictly stable category} is a pair
  of a category $\mathcal{C}$
  and an automorphism of categories
  $\Sigma \colon \mathcal{C} \rightarrow \mathcal{C}$.
\end{dfn}

\begin{exm}
  The poset $\M \subset \R^{\circ} \times \R$, seen as a thin category,
  is a strictly stable category when endowed with the automorphism
  $T \colon \M \rightarrow \M$.
\end{exm}

Now let $\mathcal{C}_1$ and $\mathcal{C}_2$
be strictly stable categories endowed with automorphisms
$\Sigma_1$ and $\Sigma_2$, respectively.

\begin{dfn}
  A functor $F \colon \mathcal{C}_1 \rightarrow \mathcal{C}_2$
  is \emph{strictly stable} if
  $F \circ \Sigma_1 = \Sigma_2 \circ F$.
  \emph{Strictly stable} natural transformations
  are defined analogously.
\end{dfn}

With these definitions strictly stable categories
form a strict $2$-category.
Moreover, by dropping the associated automorphism,
we obtain a strict \emph{forgetful} $2$-functor
from strictly stable categories to categories.
We will now construct a right adjoint to this forgetful functor.
To this end, let $\mathcal{D}$ be an ordinary category,
then we may associate a strictly stable category $\mathcal{D}^{\Z}$
to $\mathcal{D}$.

\begin{dfn}
  The objects of $\mathcal{D}^{\Z}$ are maps
  $M_{\bullet} \colon n \mapsto M_n$
  from $\Z$ to the class of objects in $\mathcal{D}$.
  Homomorphisms, composition, and identities
  are defined pointwise in $\mathcal{D}^{\Z}$.
  As the associated automorphism we choose
  \begin{equation*}
    \Sigma \colon \mathcal{D}^{\Z} \rightarrow \mathcal{D}^{\Z}, ~~
    M_{\bullet} \mapsto M_{\bullet-1}
    .
  \end{equation*}
\end{dfn}

This construction yields a strict $2$-functor
from categories to strictly stable categories.

\begin{exm}
  The category of $\Z$-graded vector spaces $\mathrm{Vect}_K^{\Z}$ over $K$
  is the strictly stable category associated to $\mathrm{Vect}_K$.
\end{exm}

\begin{lem}
  \label{lem:2adj}
  The forgetful functor and $(-)^{\Z}$
  form a strict $2$-adjunction with
  the evaluation at $0$ as the counit
  \begin{equation*}
    \op{ev}_0 \colon \mathcal{D}^{\Z} \rightarrow \mathcal{D}, ~~
    M_{\bullet} \mapsto M_{0}
    .
  \end{equation*}
\end{lem}

Depending on the context,
we may also write $M^{\bullet}$ in place of $M_{\bullet}$ and
$\op{ev}^0$
in place of $\op{ev}_0$.
Now let $\mathcal{A}$ be an additive (or pointed) category,
let $C \subset \M$ be a convex subposet,
and let $F \colon C \rightarrow \mathcal{A}$
be a functor vanishing on $C \cap \partial \M$.

\begin{lem}
  \label{lem:zero}
  Let $u, v \in C$ with
  $u \preceq v \npreceq T(u)$.
  Then
  $F(u \preceq v) = 0$.
\end{lem}

Now let $\Sigma$ be an automorphism of $\mathcal{A}$
and let $F \colon \M \rightarrow \mathcal{A}$
be a strictly stable functor vanishing on $\partial \M$,
let $D$ be a convex subposet of $\M$ that is a fundamental domain
with respect to the action of $\langle T \rangle$,
and let $F' := F |_D$.

\begin{dfn}
  \label{dfn:R_D}
  We set
  $R_D :=
  \set[\big]{
  (v, w) \in D \times T(D) \given
   v \preceq w \preceq T(v)
  }
  .$
\end{dfn}

If we view $R_D$ as a subposet of $D \times T(D)$ with the product order,
we obtain the two functors
$F' \circ \op{pr}_1 = F \circ \op{pr}_1$ and
$\Sigma \circ F' \circ T^{-1} \circ \op{pr}_2 = F \circ \op{pr}_2$,
where $\op{pr}_1 \colon R_D \rightarrow D$
and $\op{pr}_2 \colon R_D \rightarrow T(D)$
are the projections to the first and the second component, respectively.
The following definition provides a natural transformation $\partial (F, D)$
as in the diagram
\begin{equation*}
  \begin{tikzcd}
    R_D
    \arrow[rrr, "\op{pr}_1"]
    \arrow[ddd, "\op{pr}_2"']
    &[-23pt] & &[-18pt]
    D
    \arrow[ddd, "F"]
    \\[-17pt]
    & &
    {}
    \arrow[ld, Rightarrow, "{\partial (F, D)}"]
    \\
    &
    {}
    \\[-20pt]
    T(D)
    \arrow[rrr, "F"']
    & & &
    \mathcal{A}
    .
  \end{tikzcd}
\end{equation*}
\begin{dfn}
  We set
  $~
  \partial (F, D)
  \colon
  F \circ \op{pr}_1 \Rightarrow F \circ \op{pr}_2,
  (v, w) \mapsto F(v \preceq w)
  .$
\end{dfn}

\sloppy
In the following statement we will use
${w \preceq T(v) \in T^{n+1}(D)}$
as a shorthand for
${w, T(v) \in T^{n+1}(D)}$, $n \in \Z$,
and $w \preceq T(v)$.

\begin{lem}
  \label{lem:reduction}
  Suppose $\partial' := \partial (F, D)$,
  then
  \begin{equation}
    \label{eq:constr_ext}
    F(v \preceq w) =
    \begin{cases}
      (\Sigma^n \circ F' \circ T^{-n})(v \preceq w)
      & v, w \in T^n(D)
      \\
      (\Sigma^n \circ \partial')_{
        \left(T^{-n} (v), T^{-n} (w)\right)
      }
      &
      w \preceq T(v) \in T^{n+1}(D)
      \\
      0
      & \text{otherwise}
    \end{cases}
  \end{equation}
  for all $v \preceq w \in \M$.
\end{lem}

This lemma shows that $F$ is determined by
its restriction $F |_D$ and
the natural transformation $\partial(F, D)$.

Now let $F' \colon D \rightarrow \mathcal{A}$ be an arbitrary functor
vanishing on $D \cap \partial \M$,
let $\partial'$ be a natural transformation as in the diagram
\begin{equation*}
  \begin{tikzcd}
    R_D
    \arrow[rrr, "\op{pr}_1"]
    \arrow[ddd, "\op{pr}_2"']
    &[-23pt] & &[-18pt]
    D
    \arrow[ddd, "F'"]
    \\[-17pt]
    & &
    {}
    \arrow[ld, Rightarrow, "{\partial'}"]
    \\
    &
    {}
    \\[-20pt]
    T(D)
    \arrow[rrr, "\Sigma \circ F' \circ T^{-1}"']
    & & &
    \mathcal{A}
    ,
  \end{tikzcd}
\end{equation*}
and let $F \colon \M \rightarrow \mathcal{A}$
be defined by equation \eqref{eq:constr_ext}.
We aim to show that $F$ is a functor.

\begin{lem}
  \label{lem:zero_constr}
  Let
  $v \preceq w \npreceq T(v)$,
  then
  $F(v \preceq w) = 0$.  
\end{lem}

\begin{proof}
  If $v, w \in T^n (D)$ for some $n$,
  the statement follows from \cref{lem:zero}
  and the defining equation \eqref{eq:constr_ext}.
  In any other case the result follows directly from
  the construction \eqref{eq:constr_ext}.
\end{proof}

\begin{lem}
  \label{lem:rectCompose}
  Let
  $u \preceq v \preceq w \preceq T(u)$,
  then
  \[F(u \preceq w) =
    F(v \preceq w)
    \circ
    F(u \preceq v)
    .\]
\end{lem}

\begin{proof}
  Without loss of generality we assume
  $u \in D$.
  Since $D \cup T(D)$ is convex we have
  $v, w \in D \cup T(D)$.
  If $w \in D$ we are done, since $D$ is convex and $F'$ is a functor.
  Suppose $w \in T(D)$ and $v \in D$, then
  \begin{equation*}
    \begin{split}
      F(u \preceq w)
      & =
      \partial'_{(u, w)}
      \\
      & =
      \partial'_{(v, w)}
      \circ
      F' (u \preceq v)
      \\
      & =
      F (v \preceq w)
      \circ
      F (u \preceq v)
    \end{split}
  \end{equation*}
  by the naturality of $\partial'$ in the first argument.
  Similarly if $v, w \in T(D)$, then
  \begin{equation*}
    \begin{split}
      F(u \preceq w)
      & =
      \partial'_{(u, w)}
      \\
      & =
      (\Sigma \circ F' \circ T^{-1}) (v \preceq w)
      \circ
      \partial'_{(u, v)}
      \\
      & =
      F (v \preceq w)
      \circ
      F (u \preceq v)
    \end{split}
  \end{equation*}
  follows from the naturality of $\partial'$ in its second argument.
\end{proof}

\begin{lem}
  \label{lem:extFunctor}
  The data for $F$ yields a functor.
\end{lem}

\begin{proof}
  For all $u \preceq v \preceq w \in \M$
  we have to show the equation
  \[F(u \preceq w) =
    F(v \preceq w)
    \circ
    F(u \preceq v)
    .\]
  If $w \preceq T(u)$, then we are done
  by \cref{lem:rectCompose}.
  Otherwise \cref{lem:zero_constr} implies that
  $F(u \preceq w) = 0$ and thus we have to show
  \[0 =
    F(v \preceq w)
    \circ
    F(u \preceq v)
    .\]
  
  In case $v \npreceq T(u)$ or
  $w \npreceq T(v)$, \cref{lem:zero_constr}
  applies to the right hand side of this equation as well.
  
  Now suppose $v \preceq T(u)$ and
  $w \preceq T(v)$.
  Since $\partial \left(\downarrow T(u)\right)$
  divides $\M$ into two connected components
  there is some point
  $v' \in [v \preceq w]
  \cap \partial \left(\downarrow T(u)\right)$.
  Two applications of \cref{lem:rectCompose} yield
  \begin{equation*}
    \begin{split}
      F(v \preceq w)
      \circ
      F(u \preceq v)
      & =
      F(v' \preceq w)
      \circ
      F(v \preceq v')
      \circ
      F(u \preceq v)
      \\
      & =
      F(v' \preceq w)
      \circ
      F(u \preceq v')
      .
    \end{split}
  \end{equation*}
  We are done if we can show that
  $F(u \preceq v') = 0$.
  
  Now $F |_{[u, T(u)]}$
  is a functor by \cref{lem:rectCompose}.
  Moreover, $u \preceq v'$
  factors through a point in $\partial \M$ by our choice of $v'$.
  And since $F |_{\partial \M} = 0$ we obtain
  $F(u \preceq v') = 0$
  and thus the desired result.
\end{proof}

\goodbreak

\cref{lem:extFunctor} and \cref{lem:reduction} in conjunction
imply the following.

\begin{prp}
  \label{prp:fundaExt}
  For any functor $F' \colon D \rightarrow \mathcal{A}$
  vanishing on $D \cap \partial \M$
  together with a natural transformation
  \begin{equation*}
    \begin{tikzcd}
      R_D
      \arrow[rrr, "\op{pr}_1"]
      \arrow[ddd, "\op{pr}_2"']
      &[-23pt] & &[-18pt]
      D
      \arrow[ddd, "F'"]
      \\[-17pt]
      & &
      {}
      \arrow[ld, Rightarrow, "{\partial'}"]
      \\
      &
      {}
      \\[-20pt]
      T(D)
      \arrow[rrr, "\Sigma \circ F' \circ T^{-1}"']
      & & &
      \mathcal{A}
      ,
    \end{tikzcd}
  \end{equation*}
  there is a unique strictly stable functor
  $F \colon \M \rightarrow \mathcal{A}$
  with
  \[F |_D = F', \quad
    F |_{\partial \M} = 0, \quad \text{and} \quad
    \partial (F, D) = \partial' .\]
  Moreover, this construction is natural
  in ${F' \colon D \rightarrow \mathcal{A}}$.
\end{prp}

We end this appendix by discussing how one may extend partially
defined natural transformations.

\begin{lem}
  \label{lem:natExt}
  Let ${F, G \colon \M \rightarrow \mathcal{A}}$
  be strictly stable functors
  vanishing on $\partial \M$
  and let ${D, E \subset \M}$ be convex fundamental domains
  with ${D \subset E \cup T(E)}$
  and both ${\op{int} \M \setminus (D \cap E)}$
  and ${\op{int} \M \setminus (D \cap T(E))}$ disconnected.
  Moreover, let ${\eta \colon F |_{D \cap E} \rightarrow G |_{D \cap E}}$
  and ${\nu \colon F |_{D \cap T(E)} \rightarrow G |_{D \cap T(E)}}$
  be natural transformations with the following two properties:
  \begin{enumerate}[(1)]
  \item
    For any $u \in D \cap E$ and $v \in D \cap T(E)$ with
    $u \preceq v \preceq T(u)$ the diagram
    \begin{equation*}
      \begin{tikzcd}[column sep=large, row sep=large]
        F(v)
        \arrow[r, "\nu_v"]
        &
        G(v)
        \\
        F(u)
        \arrow[u, "F(u \preceq v)"]
        \arrow[r, "\eta_u"']
        &
        G(u)
        \arrow[u, "G(u \preceq v)"']
      \end{tikzcd}
    \end{equation*}
    commutes.
  \item
    For any $v \in D \cap T(E)$ and $w \in T(D \cap E)$ with
    $v \preceq w \preceq T(v)$ the diagram
    \begin{equation*}
      \begin{tikzcd}[column sep=5.5em, row sep=large]
        F(w)
        \arrow[r, "(\Sigma \circ \eta \circ T^{-1})_w"]
        &
        G(w)
        \\
        F(v)
        \arrow[u, "F(v \preceq w)"]
        \arrow[r, "\nu_v"']
        &
        G(v)
        \arrow[u, "G(v \preceq w)"']
      \end{tikzcd}
    \end{equation*}
    commutes.
  \end{enumerate}
  Then the natural transformations $\eta$ and $\nu$ extend uniquely to
  a single strictly stable natural transformation from $F$ to $G$.
\end{lem}

\section{Middle Exact Squares}
\label{sec:middleExact}

\begin{dfn}
  \label{dfn:middleExact}
  We say that a commutative square
  \begin{equation*}
    \begin{tikzcd}
      A
      \arrow[r, "f_{AB}"]
      \arrow[d, "f_{AC}"']
      &
      B
      \arrow[d, "f_{BD}"]
      \\
      C
      \arrow[r, "f_{CD}"']
      &
      D
    \end{tikzcd}
  \end{equation*}
  of vector spaces (or modules)
  is \emph{middle exact}
  if the sequence
  \begin{equation*}
    A
    \xrightarrow{\begin{pmatrix} f_{AB} \\ f_{AC} \end{pmatrix}}
    B \oplus C
    \xrightarrow{\begin{pmatrix} f_{BD} & -f_{CD} \end{pmatrix}}
    D
  \end{equation*}
  is exact (at the middle term $B \oplus C$).  
\end{dfn}

Now suppose we have two adjacent middle exact squares
\begin{equation*}
  \begin{tikzcd}
    A
    \arrow[r]
    \arrow[d]
    &
    B
    \arrow[r]
    \arrow[d]
    &
    C
    \arrow[d]
    \\
    D
    \arrow[r]
    &
    E
    \arrow[r]
    &
    F
    ,
  \end{tikzcd}
\end{equation*}
with the maps denoted by $f_{AB}$, $f_{AC} = f_{BC} \circ f_{AB}$,
and so forth.

\begin{lem}
  \label{lem:twoMiddleExactSquares}
  We have
  \(f_{EF}^{-1} \left(\op{Im} f_{DF} + \op{Im} f_{BF}\right)
  =
  \op{Im} f_{DE} + \op{Im} f_{BE}
  .\)
\end{lem}

\begin{proof}
  It is clear that the right hand side is a subspace
  of the left hand side.
  The other inclusion follows from the following diagram chase.
  Suppose we have
  \begin{equation*}
    u \in
    f_{EF}^{-1} \left(\op{Im} f_{DF} + \op{Im} f_{BF}\right)
    ,
  \end{equation*}
  then there are vectors
  $v \in B$ and $w \in D$
  with
  \begin{equation*}
    f_{EF} (u) = f_{DF} (w) + f_{BF} (v)
    .
  \end{equation*}
  Moreover, we have that
  \begin{equation*}
    f_{EF} (u - f_{DE} (w) - f_{BE} (v)) = 0,
  \end{equation*}
  hence the term
  \[\begin{pmatrix}
    0
    \\
    u - f_{DE} (w) - f_{BE} (v)
  \end{pmatrix} \in C \oplus E\]
  is in the kernel of
  \(\begin{pmatrix}
    f_{CF} & -f_{EF}
  \end{pmatrix}.\)
  By the exactness of the sequence
  \begin{equation*}
    B
    \xrightarrow{\begin{pmatrix} f_{BC} \\ f_{BE} \end{pmatrix}}
    C \oplus E
    \xrightarrow{\begin{pmatrix} f_{CF} & -f_{EF} \end{pmatrix}}
    F
  \end{equation*}
  there is a $v' \in B$ with
  \[
    f_{BE} (v') =
    u - f_{DE} (w) - f_{BE} (v)
    ,
  \]
  which is equivalent to
  $f_{DE} (w) + f_{BE} (v + v') = u.$
\end{proof}

Now we consider a diagram of four adjacent middle exact squares
\begin{equation}
  \label{eq:middleExactFourSquares}
  \begin{tikzcd}
    A
    \arrow[r]
    \arrow[d]
    &
    B
    \arrow[r]
    \arrow[d]
    &
    C
    \arrow[d]
    \\
    D
    \arrow[r]
    \arrow[d]
    &
    E
    \arrow[r]
    \arrow[d]
    &
    F
    \arrow[d]
    \\
    G
    \arrow[r]
    &
    H
    \arrow[r]
    &
    I
    .
  \end{tikzcd}
\end{equation}
We note that middle exact squares \enquote{compose}
to middle exact squares.
If we compose the two squares in the first row
of \eqref{eq:middleExactFourSquares}
as well as the two squares in the second row
and then transpose
the diagram,
we obtain this diagram
\begin{equation}
  \label{eq:middleExactTranspose}
  \begin{tikzcd}
    A
    \arrow[r]
    \arrow[d]
    &
    D
    \arrow[r]
    \arrow[d]
    &
    G
    \arrow[d]
    \\
    C
    \arrow[r]
    &
    F
    \arrow[r]
    &
    I
  \end{tikzcd}
\end{equation}
of two adjacent middle exact squares,
which we will use at the end of the proof of the following proposition.

\begin{prp}
  \label{prp:4Squares}
  The map $f_{EI}$ induces a natural isomorphism
  \begin{equation*}
    \frac{E}{
      \op{Im} f_{DE} + \op{Im} f_{BE}
    }
    \cong
    \frac{
      \op{Im} f_{EI} + \op{Im} f_{CI}
    }{
      \op{Im} f_{DI} + \op{Im} f_{CI}
    }
    .
  \end{equation*}
\end{prp}

\begin{proof}
  As the upper right square
  of \eqref{eq:middleExactFourSquares}
  is middle exact
  we have
  \begin{equation}
    \label{eq:auxMiddleExact}
    \op{Im} f_{BF} =
    \op{Im} f_{EF} \cap \op{Im} f_{CF}
    .
  \end{equation}  
  With this we obtain the chain of three isomorphisms
  and two equalities
  \begin{align*}
    \frac{
    E
    }{
    \op{Im} f_{DE} + \op{Im} f_{BE}
    }
    & \cong
      \frac{
      \op{Im} f_{EF}
      }{
      \op{Im} f_{DF} + \op{Im} f_{BF}
      }
    \\
    & =
      \frac{
      \op{Im} f_{EF}
      }{
      \op{Im} f_{DF} + \op{Im} f_{EF} \cap \op{Im} f_{CF}
      }
    \\
    & =
      \frac{
      \op{Im} f_{EF}
      }{
      \op{Im} f_{EF}
      \cap
      \left(\op{Im} f_{DF} + \op{Im} f_{CF}\right)
      }
    \\
    & \cong
      \frac{
      \op{Im} f_{EF} + \op{Im} f_{CF}
      }{
      \op{Im} f_{DF} + \op{Im} f_{CF}
      }
    \\
    & \cong
      \frac{
      \op{Im} f_{EI} + \op{Im} f_{CI}
      }{
      \op{Im} f_{DI} + \op{Im} f_{CI}
      }
      .
    \end{align*}
    Here the first isomorphism follows from
    \cref{lem:twoMiddleExactSquares}
    applied to the two squares at the top of \eqref{eq:middleExactFourSquares}
    and
    \href{
      https://en.wikipedia.org/wiki/Isomorphism_theorems#Theorem_A
    }{the first isomorphism theorem}.
    The first equality follows from \eqref{eq:auxMiddleExact}.
    The second equality follows from
    \href{
      https://en.wikipedia.org/wiki/Modular_lattice
    }{the modular law}
    for the lattice of subspaces.
    The second isomorphism follows from
    \href{
      https://en.wikipedia.org/wiki/Isomorphism_theorems#Theorem_B
    }{the second isomorphism theorem}
    and the last isomorphism again from
    \cref{lem:twoMiddleExactSquares}
    applied to \eqref{eq:middleExactTranspose} and
    \href{
      https://en.wikipedia.org/wiki/Isomorphism_theorems#Theorem_A
    }{the first isomorphism theorem}.    
\end{proof}

\section{Cohomological Functors on $\M$}
\label{sec:cohomological}

\begin{figure}[t]
  \centering
\begingroup%
  \makeatletter%
  \providecommand\color[2][]{%
    \errmessage{(Inkscape) Color is used for the text in Inkscape, but the package 'color.sty' is not loaded}%
    \renewcommand\color[2][]{}%
  }%
  \providecommand\transparent[1]{%
    \errmessage{(Inkscape) Transparency is used (non-zero) for the text in Inkscape, but the package 'transparent.sty' is not loaded}%
    \renewcommand\transparent[1]{}%
  }%
  \providecommand\rotatebox[2]{#2}%
  \newcommand*\fsize{\dimexpr\f@size pt\relax}%
  \newcommand*\lineheight[1]{\fontsize{\fsize}{#1\fsize}\selectfont}%
  \ifx\svgwidth\undefined%
    \setlength{\unitlength}{275.49998474bp}%
    \ifx\svgscale\undefined%
      \relax%
    \else%
      \setlength{\unitlength}{\unitlength * \real{\svgscale}}%
    \fi%
  \else%
    \setlength{\unitlength}{\svgwidth}%
  \fi%
  \global\let\svgwidth\undefined%
  \global\let\svgscale\undefined%
  \makeatother%
  \begin{picture}(1,0.51724141)%
    \lineheight{1}%
    \setlength\tabcolsep{0pt}%
    \put(0,0){\includegraphics[width=\unitlength,page=1]{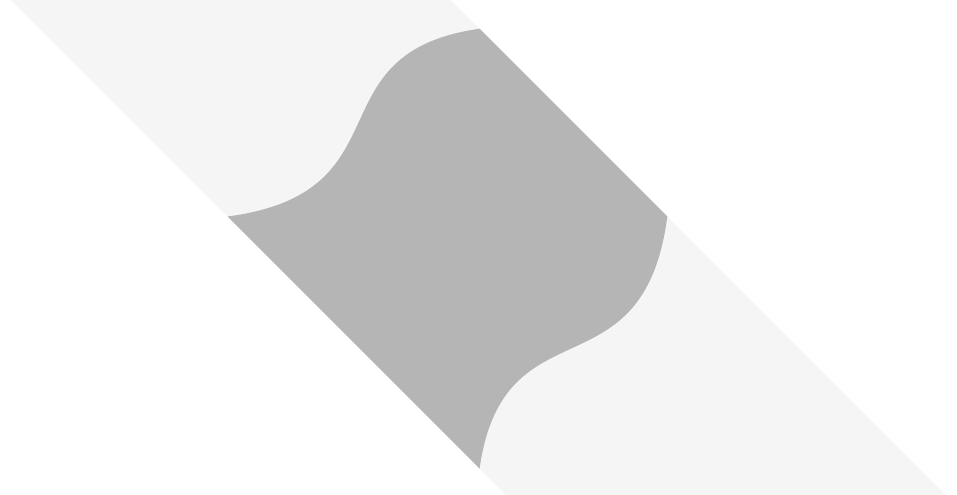}}%
    \put(0.40996173,0.3409963){\makebox(0,0)[t]{\lineheight{1.25}\smash{\begin{tabular}[t]{c}$w$\end{tabular}}}}%
    \put(0.6015327,0.3409963){\makebox(0,0)[t]{\lineheight{1.25}\smash{\begin{tabular}[t]{c}$v_2$\end{tabular}}}}%
    \put(0.40996173,0.18007651){\makebox(0,0)[t]{\lineheight{1.25}\smash{\begin{tabular}[t]{c}$v_1$\end{tabular}}}}%
    \put(0.5977013,0.18007651){\makebox(0,0)[t]{\lineheight{1.25}\smash{\begin{tabular}[t]{c}$u$\end{tabular}}}}%
    \put(0,0){\includegraphics[width=\unitlength,page=2]{general-rect.pdf}}%
    \put(0.78882473,0.23646263){\makebox(0,0)[t]{\lineheight{1.25}\smash{\begin{tabular}[t]{c}$l_1$\end{tabular}}}}%
    \put(0.28998568,0.19277288){\makebox(0,0)[t]{\lineheight{1.25}\smash{\begin{tabular}[t]{c}$l_0$\end{tabular}}}}%
  \end{picture}%
\endgroup%

  \caption{
    An axis-aligned rectangle $u \preceq v_1, v_2 \preceq w \in D$
    contained in the fundamental domain $D$,
    which is shaded in grey.
  }
  \label{fig:axisRect}
\end{figure}

Let $F \colon \M^{\circ} \rightarrow \VectK$
be a contravariant functor vanishing on $\partial \M$.
Moreover, suppose there is a convex fundamental domain $D \subset \M$
with respect to $\langle T \rangle$,
such that for any axis-aligned rectangle $u \preceq v_1, v_2 \preceq w \in D$
as shown in \cref{fig:axisRect},
the long sequence
\begin{equation}
  \label{eq:mvs}
  \begin{tikzcd}
    &
    \cdots
    \arrow[r]
    &
    F(T(u))
    \arrow[dll, out=0, in=180, looseness=2, overlay]
    \\
    F(w)
    \arrow[r]
    &
    F(v_1) \oplus F(v_2)
    \arrow[r, "(1 ~\, -1)"]
    &
    F(u)
    \arrow[dll, out=0, in=180, looseness=2, overlay]
    \\
    F(T^{-1}(w))
    \arrow[r]
    &
    \cdots
  \end{tikzcd}  
\end{equation}
is exact.
We show that \eqref{eq:mvs} is exact for any axis-aligned rectangle
$u \preceq v_1, v_2 \preceq w \in \M$.
We begin with the special case that $v_2 \in \partial \M$;
then we have $F(v_2) \cong \{0\}$ and we set $v := v_1$.
As shown in \cref{fig:homological}
the union of the orbits of $u$, $v$, and $w$ form a subposet,
which is isomorphic to $\Z$.
As $D$ is convex, the intersection of $D$ and this subposet
consists of three consecutive points;
in \cref{fig:homological} these are $T^2(v)$, $T^2(w)$, and $T^3(u)$.
Moreover, these three consecutive points describe an axis-aligned rectangle
contained in $D$.
Thus, the restriction of $F$ to this subposet yields the long exact sequence
\begin{equation}
  \label{eq:restrF}
  \begin{tikzcd}
    &
    \cdots
    \arrow[r]
    &
    F(T(u))
    \arrow[dll, out=0, in=180, looseness=2, overlay]
    \\
    F(w)
    \arrow[r]
    &
    F(v)
    \arrow[r]
    &
    F(u)
    \arrow[dll, out=0, in=180, looseness=2, overlay]
    \\
    F(T^{-1}(w))
    \arrow[r]
    &
    \cdots
    .
  \end{tikzcd}  
\end{equation}

\begin{figure}[t]
  \centering
  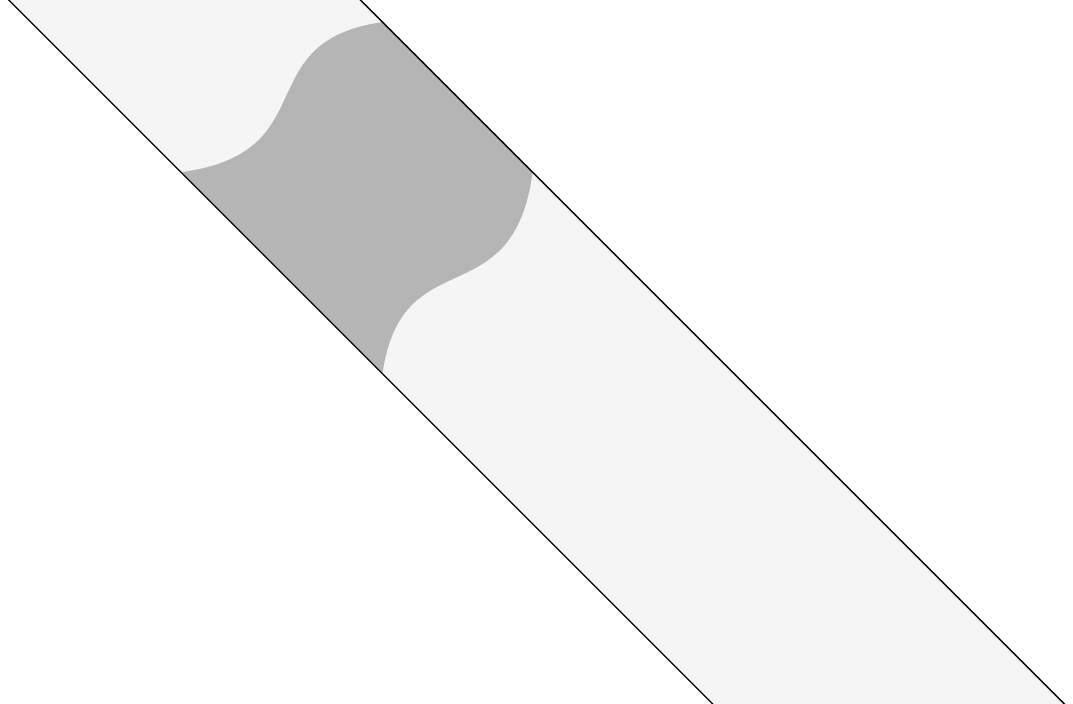
  \caption{
    The linear subposet given by the orbits of $u$, $v$, and $w$.
    The region shaded in dark grey is our fundamental domain $D$.
  }
  \label{fig:homological}
\end{figure}

\begin{dfn}
  \label{dfn:cohomological}
  We say that a contravariant functor
  $F \colon \M^{\circ} \rightarrow \mathrm{Vect}_K$
  vanishing on $\partial \M$ is \emph{cohomological},
  if for any axis-aligned rectangle
  with one corner lying on $l_1$ and the other corners
  ${u \preceq v \preceq w \in \M}$,
  the long sequence
  \eqref{eq:restrF}
  is exact.    
\end{dfn}

This notion of a cohomological functor is inspired by the theory
of \emph{triangulated categories}.
A \emph{cohomological} functor on a triangulated category
yields a long exact sequence for any \emph{distinguished triangle},
see for example \cite[Definitions 1.5.2 and 1.5.1]{Kashiwara1990}.
Here the defining property of a cohomological functor is
that it yields long exact sequences for certain \enquote{triangles}
in $\M$ or certain geodesic triangles on the Möbius strip
$\M / \langle T \rangle$.

\begin{figure}[t]
  \centering
  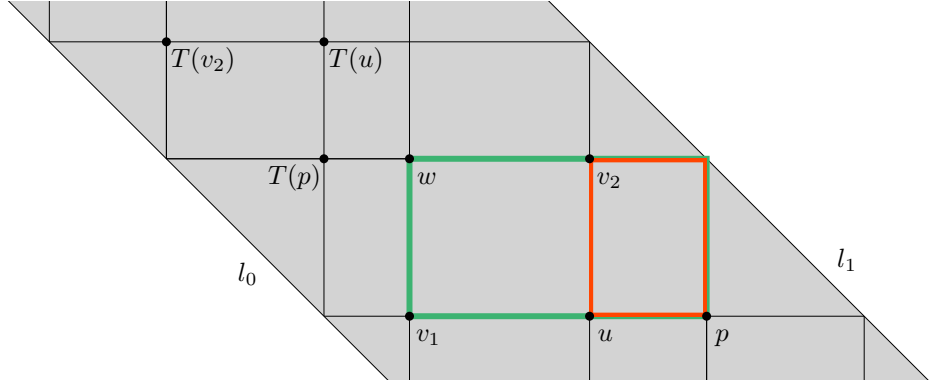
  \caption{
    Axis-aligned rectangles determined by $u$, $v_1$, $v_2$, and $w$.
  }
  \label{fig:rect}
\end{figure}

Now suppose we have an arbitrary axis-aligned rectangle
$u \preceq v_1,\, v_2 \preceq w \in \M$ as shown in \cref{fig:rect}.
Together with the additional point $p \in \M$
we obtain two axis-aligned rectangles with one corner on $l_1$.
As $F$ is cohomological, the commutative diagram
\begin{equation*}
  \begin{tikzcd}
    \dots
    \arrow[r]
    &
    F(T(v_1))
    \arrow[r]
    \arrow[d]
    &
    F(T(p))
    \arrow[r]
    \arrow[d, equal]
    &
    F(w)
    \arrow[r]
    \arrow[d]
    &
    F(v_1)
    \arrow[r]
    \arrow[d]
    &
    F(p)
    \arrow[r]
    \arrow[d, equal]
    &
    \dots
    \\
    \dots
    \arrow[r]
    &
    F(T(u))
    \arrow[r]
    &
    F(T(p))
    \arrow[r]
    &
    F(v_2)
    \arrow[r]
    &
    F(u)
    \arrow[r]
    &
    F(p)
    \arrow[r]
    &
    \dots
  \end{tikzcd}
\end{equation*}
has exact rows.
Thus, by
the Barratt--Whitehead Lemma
\cite[Lemma 7.4]{doi:10.1112/plms/s3-6.3.417}
the long sequence \eqref{eq:mvs} is exact.
From this we obtain the following.

\begin{prp}
  \label{prp:cohomological}
  For a contravariant functor $F \colon \M^{\circ} \rightarrow \VectK$
  vanishing on $\partial \M$
  the following are equivalent.
  \begin{enumerate}[(1)]
  \item
    There is a convex fundamental domain $D \subset \M$
    such that for any axis-aligned rectangle
    $u \preceq v_1, v_2 \preceq w \in D$
    the long sequence
    \eqref{eq:mvs}
    is exact.
  \item
    The contravariant functor $F$ is cohomological.
  \item
    For any axis-aligned rectangle
    $u \preceq v_1, v_2 \preceq w \in \M$
    the long sequence \eqref{eq:mvs} is exact.
  \item
    For any axis-aligned rectangle
    $u \preceq v_1, v_2 \preceq w \in \M$
    the square
    \begin{equation*}
      \begin{tikzcd}
        F(w)
        \arrow[d]
        \arrow[r]
        &
        F(v_2)
        \arrow[d]
        \\
        F(v_1)
        \arrow[r]
        &
        F(u)
      \end{tikzcd}
    \end{equation*}
    is middle exact.
  \end{enumerate}
\end{prp}

\begin{proof}
  Above we have shown that (1) implies (2)
  and that  (2) implies (3).
  It is clear that (3) implies both (1) and (4).
  Moreover, if we consider \cref{fig:homological},
  then we see that any three consecutive points of the
  sub-$\langle T \rangle$-set
  generated by $u$, $v$, and $w$
  describe an axis-aligned rectangle with one vertex
  on the boundary $\partial \M$
  and thus (4) implies (2).
\end{proof}

We may dualize \cref{dfn:cohomological} and \cref{prp:cohomological}
in the sense of \cref{remark:dual} as follows.

\begin{dfn}
  \label{dfn:homological}
  We say that a functor $F \colon \M \rightarrow \mathrm{Vect}_K$
  vanishing on $\partial \M$ is \emph{homological},
  if for any axis-aligned rectangle
  with one corner lying on $l_1$ and the other corners
  ${u \preceq v \preceq w \in \M}$,
  the long sequence
  \begin{equation*}
    \begin{tikzcd}
      &
      \cdots
      \arrow[r]
      &
      F(T^{-1}(w))
      \arrow[dll, out=0, in=180, looseness=2, overlay]
      \\
      F(u)
      \arrow[r]
      &
      F(v)
      \arrow[r]
      &
      F(w)
      \arrow[dll, out=0, in=180, looseness=2, overlay]
      \\
      F(T(u))
      \arrow[r]
      &
      \cdots
    \end{tikzcd}
  \end{equation*}
  is exact.  
\end{dfn}

\begin{prp}
  \label{prp:homological}
  For a functor $F \colon \M \rightarrow \VectK$ vanishing on $\partial \M$
  the following are equivalent.
  \begin{enumerate}[(1)]
  \item
    There is a convex fundamental domain $D \subset \M$
    such that for any axis-aligned rectangle
    $u \preceq v_1, v_2 \preceq w \in D$
    the long sequence
    \begin{equation}
      \label{eq:mvs2}
      \begin{tikzcd}
        &
        \cdots
        \arrow[r]
        &
        F(T^{-1}(w))
        \arrow[dll, out=0, in=180, looseness=2, overlay]
        \\
        F(u)
        \arrow[r]
        &
        F(v_1) \oplus F(v_2)
        \arrow[r, "(1 ~~ -1)"]
        &
        F(w)
        \arrow[dll, out=0, in=180, looseness=2, overlay]
        \\
        F(T(u))
        \arrow[r]
        &
        \cdots
      \end{tikzcd}  
    \end{equation}
    is exact.
  \item
    The functor $F$ is homological.
  \item
    For any axis-aligned rectangle
    $u \preceq v_1, v_2 \preceq w \in \M$
    the long sequence \eqref{eq:mvs2} is exact.
  \item
    For any axis-aligned rectangle
    $u \preceq v_1, v_2 \preceq w \in \M$
    the square
    \begin{equation*}
      \begin{tikzcd}
        F(u)
        \arrow[r]
        \arrow[d]
        &
        F(v_1)
        \arrow[d]
        \\
        F(v_2)
        \arrow[r]
        &
        F(w)
      \end{tikzcd}
    \end{equation*}
    is middle exact.
  \end{enumerate}
\end{prp}

\end{document}